\documentclass[12pt]{amsart}

\usepackage{amsmath,amsfonts,amsthm,amssymb}

\let\origmarginpar=\marginpar
\def\marginpar#1{\origmarginpar{\footnotesize#1}}

\def\T{\mbox{\sc t}}

\def\N{\mathbf{N}}
\def\R{\mathbf{R}}

\def\C{\mathbf{C}}
\def\Sph{{\bf S}}

\def\M{\mathcal{M}} 
\def\UM{{\bf M}} 
\def\MT{{\M_T}} 
\def\UMT{{\UM_T}}
\def\RnT{{\R^n_T}}
\def\vb{\bar{v}}
\def\vbb{\bar{\bar{v}}}
\def\Ceta{C_\eta}
\def\Cetaell{C_{\eta,\ell}}
\def\Cetacr{C_{\eta_{cr}}}

\def\mapT{\mathcal{T}}
\def\uw{w\kern-0.4em\makebox[0pt][c]{\underline{\phantom v}}\kern0.4em}

\let\Laplace=\Delta
\let\boundary=\partial
\let\eps=\varepsilon

\def\arsinh{\mathop{\rm arsinh}}

\let\halpha=\alpha 
\let\hbeta=\beta 
\def\oC{o\makebox[0pt][c]{-}C} 

\def\BC{C_b} 

\def\tr{\mathop{\rm trace}}
\def\re{\mathop{\rm Re}}
\def\im{\mathop{\rm Im}}

\newtheorem{remark}{Remark}[section]
\newtheorem{lemma}[remark]{Lemma}
\newtheorem{cor}[remark]{Corollary}

\newtheorem{definition}[remark]{Definition}
\newtheorem{theorem}[remark]{Theorem}

\renewcommand{\d}{\partial}

\def\x{\mathchoice%
  {\mbox{\boldmath$x$}}{\mbox{\boldmath$x$}}
  {\mbox{\tiny\boldmath$x$}}
  {\mbox{\tiny\boldmath$x$}}
}

\def\y{\mathchoice%
  {\mbox{\boldmath$y$}}{\mbox{\boldmath$y$}}
  {\mbox{\tiny\boldmath$y$}}
  {\mbox{\tiny\boldmath$y$}}
}

\def\cb{\mbox{\boldmath$c$}}

\let\orignabla=\nabla
\def\nabla{\mbox{\boldmath$\orignabla$}}

\def\z{\mbox{\boldmath$z$}}   
\def\e{\mbox{\boldmath$e$}}   
\def\Rn{{\R^n}}
\def\Lop{{\bf L}}
\def\Lopmod{\raisebox{0.3ex}{-}\kern-0.42em\mbox{\bf L}}
\def\Lup{\mathbb{L}}
\def\Hop{{\bf H}}

\def\Pop{{\bf P}}
\def\Qop{{\bf Q}}
\def\Sop{{\bf S}}
\def\Div{\nabla\cdot}

\let\ita=k 

\let\Dst=\displaystyle
\let\Tst=\textstyle
\def\uB{u_{\scriptscriptstyle B}}
\def\rhoB{\rho_{\scriptscriptstyle B}}
\def\rhoBplus{\rho_{\scriptscriptstyle B_+}}
\def\uBplus{u_{\scriptscriptstyle B_+}}
\def\rhoBminus{\rho_{\scriptscriptstyle B_-}}

\newcommand{\spread}{\beta}
\newcommand{\biota}{{\underline \iota}}
\newcommand{\blambda}{{\underline \lambda}}
\newcommand{\origin}{{\mathbf 0}}

\def\vc{\tilde{v}} 
\def\Vc{\tilde{V}}
\def\vcheck{\check{v}}
\def\vcheckc{\tilde{\check{v}}} 

\def\vareta{\hat\eta} %

\def\glresp{\mathbin{\mbox{$\begin{array}{@{}c@{}}>\\[-1ex]<\end{array}$}}}

\mathcode`\`="8000
\begingroup
\makeatletter
\gdef\rqch@r{\hbox{\char'140}}
\catcode`\`=13
\gdef`{\@ifnextchar[{\mathclose}{\@ifnextchar]{\mathopen}{\@ifnextchar<{\mathopen}{\@ifnextchar>{\mathclose}{\rqch@r}}}}}
\makeatother
\endgroup

\begin{document}

\title{Higher-order time asymptotics of fast diffusion in Euclidean space:
 a dynamical systems approach}

\author{Jochen Denzler}
\author{Herbert Koch}
\author{Robert J. McCann}

\email{denzler@math.utk.edu}
\email{koch@math.uni-bonn.de}
\email{mccann@math.toronto.edu}

\address{Jochen Denzler, Department of Mathematics, University of Tennessee,
  Knoxville, TN 37996, USA}
\address{Herbert Koch, Mathematisches Institut, Universit\"at Bonn, Endenicher
  Allee 60, 53115 Bonn, Germany}
\address{Robert McCann, Department of Mathematics, University of Toronto, 
  Toronto, Ontario M5S\,2E4, Canada}

\thanks{The authors are pleased to acknowledge the hospitality of the
  Mathematical Sciences Research Institute at Berkeley,  and of the
  Banff International Research Station,  where parts of this work were
  performed.  They are grateful to Juan-Luis V\'azquez for
  conversations which helped to stimulate this project, and for his
  hospitality in Madrid,  and to Thomas Hagen for providing an early opportunity 
  to present these results in the May 2008 AIMS Conference
  on Dynamical Systems and Differential Equations.
  This research was supported in part by
  Natural Sciences and Engineering Research Council of Canada Grants
  217006-03 and -08 and United States National Science Foundation Grant
  DMS-0354729 to Robert McCann; by a grant from 
  the Simons Foundation (\#208550) to Jochen Denzler; and by the DFG through 
  SFB 611 (Herbert Koch). 
}
\thanks{\copyright 2012 by the authors.
}

\begin{abstract}
This paper quantifies the speed of convergence and higher-order asymptotics
of fast diffusion dynamics on $\Rn$ to the Barenblatt (self similar) solution.
Degeneracies in the parabolicity of this equation are cured by
re-expressing the dynamics on a manifold with a
cylindrical end, called the cigar.
The nonlinear evolution becomes differentiable
in H\"older spaces on the cigar.  The
linearization of the dynamics is 
given by the Laplace-Beltrami operator
plus a transport term (which can be suppressed by introducing
appropriate weights into the function space norm), plus a
finite-depth potential well with a universal profile.
In the limiting case of the (linear) heat equation, the depth diverges,
the number of eigenstates increases without bound,  and the continuous
spectrum recedes to infinity.
We provide a detailed study of
the linear and nonlinear problems in H\"older spaces on the cigar, including
a sharp boundedness estimate for the semigroup, and use this as a tool to
obtain sharp convergence results toward the Barenblatt solution, and higher
order asymptotics. In finer
convergence results (after modding out symmetries of the problem), a subtle
interplay between convergence rates and tail behavior is revealed.
The difficulties involved in choosing the right functional 
spaces in which to carry out the analysis can be interpreted as genuine
features of the equation rather than mere annoying technicalities.
\end{abstract}

\maketitle

\section{Introduction}
\label{SecIntro}

Long-time asymptotics of nonlinear diffusion processes have been a subject of much
recent interest.  The porous medium equation 
\begin{equation}
\label{pm}
 \rho_{\tau} =\frac1m \Laplace \rho^m
\end{equation}
is a prototypical example; it
governs the evolution of a nonnegative density $\rho(\tau,\y)$
on $[0,\infty`[ \times \Rn$,
as described in V\'azquez book \cite{MR2286292} and its references.
The basin of attraction \cite{MR586735} \cite{MR1977429}
of its self-similar solution  \cite{ZeldovichKompaneets50} \cite{Barenblatt52} \cite{Pattle59},
and the rate of convergence of other solutions to it
\cite{CaLeMaTo03}
\cite{MR1986060} 
\cite{MR1982656} 
\cite{LedermanMarkowich03}
\cite{MR2126633} 
\cite{MR2246356} 
\cite{MR2211152} 
\cite{BBDGV07}
\cite{BBDGV09}
\cite{BGV10}
\cite{BDGV10}
\cite{DT11}
have attracted a steady stream of attention since
such rates were first obtained by
Carrillo-Toscani \cite{MR1777035}, Otto \cite{MR1842429}
and del Pino-Dolbeault \cite{MR1940370}.
Sharp results for convergence in entropy sense have since been extended to the
full range of $m \in \R$ by the quintet consisting of
Blanchet, Bonforte, Dolbeault, Grillo and Vazquez \cite{BBDGV09},
the quartet \cite{BDGV10} and the trio \cite{BGV10}. 
However, few results are known concerning higher asymptotics,
beyond the two improvements
accessible by choosing an appropriate translation in space
(Carrillo, di Francesco, Kim, McCann, Slepcev and Toscani and the quartet \cite{BDGV10}
 in various combinations
 and settings
 \cite{MR694373}  
 \cite{MR1986060} 
 \cite{MR2246356} 
 \cite{MR2211152} 
 \cite{MR2255281}) 
and in time by the Dolbeault-Toscani duo
\cite{DT11}
(c.f. \cite{AngenentAronson96} \cite{MR1491842}). 
In the one-dimensional porous medium regime $m>1$, one has a full asymptotic expansion
of Angenent \cite{MR956056} based on the spectral calculation of
Zel'dovitch and Barenblatt \cite{ZeldovichBarenblatt58}.
Koch's habilitation thesis provides a potential
framework for generalizing this to higher dimensions \cite{KochHabil}.
In the fast diffusion regime $m<1$, a spectral
calculation by Denzler and McCann \cite{MR1982656} \cite{MR2126633}
has been used to derive
the first two corrections
\cite{MR2211152} 
\cite{BDGV10}
\cite{DT11}
to the leading order asymptotics
\cite{MR1940370} \cite{MR1842429},
but the higher-order modes never
been successfully
related to the nonlinear dynamics.
It is framework for achieving such a relation which we develop for the first
time below --- inspired by ideas from dynamical systems, and sticking to the
mass-preserving range $m \in `]\frac{n-2}{n},1`[$.

To achieve this,  several obstacles must be confronted.  The spectrum of
Denzler and McCann \cite{MR2126633} contains only finitely many eigenvalues
below the continuum threshold, so only finitely many modes are in principle
accessible.
Moreover, there is an incompatibility between the spaces in which the
eigenfunctions live,  and the spaces in which the dynamics \eqref{pm}
turn out to depend differentiably on their initial conditions.
It is this differentiable dependency that we need to establish
to justify the linearization which leads to the spectral problem.

Guessing a space in which it will hold is far from trivial, however.
One of the technical devices we use to achieve
this appears independently in Bonforte, Grillo and V\'azquez' \cite{BGV10}
entropy-based approach to the special case
$m=\frac{n-4}{n-2}$: namely,  after linearizing
the rescaled evolution in relative variables,  we restore uniform parabolicity
to this apparently degenerate equation by carrying out the analysis on $\R^n$
viewed as a Riemannian manifold $\M$ with an asymptotically cylindrical
metric, known as the cigar.

Reconciling the above-mentioned
incompatibility of spaces for the linearized and nonlinear dynamics
complicates our analysis and yields a richness
to the statements of our theorems which goes beyond the intrinsic complexity of
a spectrum whose features include a web of eigenvalue crossings as $m$ is varied.
To obtain higher-order asymptotics, we need to
work in weighted spaces which discount information appropriately 
at spatial infinity.
The further we wish to penetrate into the spectrum, the more severe the required
discounting. The corresponding linearized operators are no longer self-adjoint;
they act on weighted Banach rather than Hilbert spaces.

Since the qualitative behaviour of \eqref{pm} varies considerable with
$m$ in different regimes,  let us set
\begin{equation}\label{pdef}
m_p := 1  - \frac{2}{n+p},
\end{equation}
where the parameter $p = \frac{2}{1-m} -n$ is the {\em moment-index}
introduced in \cite{MR1982656}. Thus $m_{-2}= \frac{n-4}{n-2}$, $m_0 = \frac{n-2}{n}$,
$m_2=\frac{n}{n+2}$, and $m_n=\frac{n-1}{n}$, for example.

Herrero and Pierre \cite{MR797051} obtained remarkable local estimates for
solutions to the fast diffusion equations
if $0<m< 1$ and $m_0 < m<1$,
which they and Dahlberg and Kenig \cite{MR1009117} used to prove that every
solution has an initial trace, and it can be extended for all $t$, 
and vice versa,
there exists a global solution for every initial data that is a local measure.

For $m>m_0$,  the evolution \eqref{pm} of compactly supported initial data
is mass-preserving;  this is the regime we shall explore.
For $m<m_0$, the mass dwindles to zero in finite time;
the basin of attraction 
and leading-order asymptotics describing this disappearance
have been provided by Daskalopoulos-Sesum \cite{DaskalopoulosSesum06}
and the quintet \cite{BBDGV09}, respectively.

Although the dynamics \eqref{pm} has no fixed point,
a well-known rescaling using the self similar coordinates
\begin{equation}\label{scalingtrf}
\begin{array}{l}
  \x=(1+ 2p\tau)^{-\spread} \y, \qquad
  t= \frac{1}{2p} \ln (1+2p\tau)\,,
\\[1ex]
  \spread = (2-(1-m)n)^{-1} = \frac12 (1 + \frac n p)\,,
\\[1ex]
 u(t,\x) =  e^{(n+p) nt}  \rho((e^{2pt}-1)/(2p),
  e^{(n+p)t} \x)
\end{array}
\end{equation}
yields an alternate description
\begin{equation} \label{transformed}
 \frac{\d u}{\d t}  = \frac1m \Laplace u^m +
 \frac{2}{1-m} \nabla \cdot (\x u)
\end{equation}
of the same evolution, but with a fixed point in the new variables.
  This scaling
  coincides with the one used in our proceedings report \cite{MR1982656}
  and by the quartet~\cite{BDGV10}.
Conversely, we can express  the density of the solution to the fast diffusion
equation by
\begin{equation}\label{back-trf}
\rho(\tau,\y) =    (1+ 2p\tau)^{-\spread n}
u\left( \frac{1}{2p} \ln (1+ 2p\tau),
(1+ 2p\tau)^{-\spread} \y\right) 
\;.
\end{equation}

The character of this rewriting depends on the sign of $\spread$.  If
$m>m_0$,
then $\tau \to \infty$ is equivalent to
$t \to \infty$ and the $\tau \to \infty$ asymptotics  for \eqref{pm}
translates to $t \to \infty$ asymptotics  for \eqref{transformed}.
This remains true after a similar transformation for $m=m_0$. 
For $m< m_0$ solutions may extinguish in finite time and the
asymptotics of \eqref{transformed} for $t \to \infty$  give
information about the asymptotics of \eqref{pm} at the extinguishing
time, as in work of the quintet \cite{BBDGV09} and trio \cite{BGV10}.

The stationary solution
\begin{equation}\label{barenblatt}
 \uB(\x) = ( B + |\x|^2)^{-\frac1{1-m} }
\end{equation}
to \eqref{transformed}  is related to the Barenblatt solution
\cite{ZeldovichKompaneets50} \cite{Barenblatt52} \cite{Pattle59}
\begin{equation}\label{d-rhoB}
\textstyle
\rhoB(\tau, \y)  =
( 2p\tau+1)^{-n\spread}
\uB((2p\tau+1)^{-\spread}\y ),
\end{equation}
where $B$ is determined by a quantity $\int \rho_0 d\x$ called the {\em mass},
at least if this mass is bounded.  It has been known to attract
all solutions which share its mass since the work of Friedman and
Kamin \cite{MR586735}; see also V\'azquez \cite{MR1977429}.
The Barenblatt solution has finite moments of exactly
those orders smaller than~$p$.

It has been shown by V\'azquez \cite[Thm.~21.1]{MR1977429} that weak assumptions like
$$
\int \rho(0,\y) d\y = \int \rhoB(0,\y) d\y
$$
and
$$
\sup_{\y}  |\rho(0,\y)/ \rhoB(0,\y) | < \infty
$$
imply
\begin{equation}\label{reluni}
  \limsup_{\tau \to \infty } \sup_{\y}
  \Big|\frac{\rho(\tau,\y)}{\rhoB(\tau,\y)}-1\Big| =0
\;.
\end{equation}
We investigate the sharp decay rate of
the quantity $|\,\cdot\,|$
in \eqref{reluni},
called the {\em relative} $L^\infty$ distance.
The known sharp rates of decay for integral expressions
--- such as $\|\rho - \rhoB\|_{L^1(\Rn)}$ or the relative entropy
\cite{MR1777035}
\cite{MR1940370} \cite{MR1842429}
\cite{BDGV10} 
--- imply a (non-sharp)  rate of decay in relative $L^\infty$ through 
the work of the quintet \cite{BBDGV09}. 

Because of the  gradient structure with respect to the
Wasserstein distance discovered by Otto \cite{MR1842429}, convergence questions
can be attacked using displacement convexity; see also \cite{MR2211152}.
It is within this framework that the
linearized problem for $m <1$ had been studied in great 
detail \cite{MR2126633}.
Unfortunately, there is no clean way of passing from the linearized operator
to the full nonlinear equation.   However,
by cleverly employing the entropy method,
McCann-Slepcev \cite{MR2211152},
the quartet \cite{BDGV10},
and the duo \cite{DT11}
were able improve on the sharp integral rates of convergence of
Carrillo-Toscani ($m>1$) \cite{MR1777035},
Otto ($m\ge m_n$) \cite{MR1842429}
and del-Pino-Dolbeault ($m\ge m_n$) \cite{MR1940370},
to extract first \cite{MR2211152} \cite{BDGV10} and
second \cite{DT11} corrections.
The first order correction depends on centering the data,  and the second-order
on choosing a suitable translation in~$\tau$.
Our first result 
gives the sharp relative $L^\infty$ rate of convergence for centered initial data.
It improves on the rate found independently
by the quintet  \cite{BBDGV09} without centering,
and implies the sharp entropy rate of convergence
also found independently by the quartet for centered data \cite{BDGV10}.

\begin{theorem}[Exact leading-order asymptotics in the relative $L^\infty$
    norm]
\label{t-1}
\hskip0pt plus 10em 
Fix $0<m \in `]m_0,1`[$ with $m_0=\frac{n-2}{n}$. 
Suppose $\rho(\tau,\y)$ satisfies \eqref{pm} and the condition \eqref{reluni}
holds for some $B>0$.
If $m=m_2=\frac{n}{n+2}$, we assume in addition
\begin{equation} \label{te-1}
 \int_{\Rn}
\left| 1- \frac{ \rho(0,\y)}{ \rhoB(0,\y)} \right|^2 (1+|\y|^2)^{-n/2} d\y
< \infty 
\;.
\end{equation}
Then there exists $\z \in \Rn$ such that
\begin{equation}\label{te-2}
 \limsup_{\tau\to\infty}
\sup_{\y \in \Rn} \tau \Big| \frac{\rho(\tau,\y-\z)}{\rhoB(\tau,\y)} - 1 \Big|
 < \infty
\;.
\end{equation}
Without the additional condition \eqref{te-1} in the case $m=m_2$,
we still get \eqref{te-2}, except that
the leading factor $\tau$ is replaced by  $\tau^{1-\eps}$, for
any~$\eps>0$.
\end{theorem}

Let us note that obtaining the sharp estimate \eqref{te-2} without an
$\eps$ in all cases other than $m=m_2$ is not automatic, but requires
a detailed study of the linearized PDE, as given in the proof 
of Thm.~\ref{semigroupestimates}. We rely on the manner in which the
essential spectral radius depends on decay properties
built in the space. The argument in case $m=m_2$ and assuming the
moment condition~\eqref{te-1} is of a different nature, basically a
spin-off of the $L^2$ theory and relying on self-adjointness in this
case. 

The case $m \in `]m_0,m_2`]$ had already been proved by Kim-McCann
\cite{MR2246356},  
subject to a mild moment condition, which is not needed here, except for the
similar moment-type condition \eqref{te-1} in the case $m=m_2$.
The case $m\ge m_n$ had been handled to order $O(1/\tau^{1-\eps})$ by McCann
and Slep\v{c}ev \cite{MR2211152} in the somewhat weaker $L^1$ norm. Before
this, Carrillo and V\'azquez had given the sharp $1/\tau$ rate in the
relative $L^\infty$  norm for radially symmetric data \cite{MR1986060}, and a
weaker $O(1/\tau^{1/2})$ rate in the $L^1$ norm over the whole parameter range
$m>m_0$. The sharp rate for the
case $m_2<m<m_n$ remained open for a while, even though the role of the
translation had been understood in the linearized setting by
\cite{MR2126633} and \cite{MR1986060}, before being resolved independently
in the work of the quartet \cite{BDGV10} and the present manuscript.
The work of the quintet \cite{BBDGV09} gave sharp but implicit
integral rates which extend
even into the finite extinction range of parameters $m < m_0$.
The work of the trio \cite{BGV10}
addresses the exceptional case $m = m_{-2}$,  in which the essential
spectral gap \eqref{essential spectral gap} described below and found
by Denzler and McCann \cite{MR1982656} \cite{MR2126633} vanishes.
Whereas the works of the quintet and its subgroups are entropy-based,
we use a dynamical systems approach that successfully bridges the
functional analytic gap between the linearization argument
of~\cite{MR2126633} and the nonlinear estimates, allowing higher modes
to be accessed.

The proof, which is found in Sec.~\ref{pthm1},
involves preliminary results of fairly different flavors.  A
prominent role is played by a manifold $\M$ with one cylindrical end, which is
called the cigar in the two dimensional case;  the same manifold was employed
independently by the trio \cite{BGV10}. We express the equation in self
similar variables and turn to the relative size $u/\uB$ as main dependent
variable, as do the quintet \cite{BBDGV09} and their successors
\cite{BDGV10} \cite{BGV10} \cite{DT11}.
The equation for the
quotient can be understood as a uniformly parabolic
reaction/transport/nonlinear diffusion equation
on the manifold $\M$.  Well-posedness and smooth dependence
on the initial data follow by suitable adaptations of known
techniques. However, since our setting  deviates
from more standard situations we provide complete proofs for the convenience of
the reader, and for the convenience of having results tailored to our needs.
The (rescaled and time translated) Barenblatt solutions are barriers, and
local existence immediately implies global existence if $m \ge m_0=\frac{n-2}n$.
Below $m < m_0$ the time translation corresponds to an unstable mode,
which is obvious by looking at the Barenblatt solution, where a time shift
is used to adjust the $\tau$ at which the solution extinguishes.

The asymptotics of Theorem \ref{t-1} are determined by the spectrum of the
linearized operator. Its largest eigenvalue $\lambda_{00}$ is zero.  It
corresponds to the rescaling
$$
\rho_\sigma(\tau, \x) = \rho( \tau/\sigma^2, \x/\sigma) 
\;.
$$
The invariant
manifold of this mode is given by the  set of stationary solutions $\uB$, which is
parametrized by $B$. For $m > \frac{n-2}n$, $B$ is determined by the
mass ($L^1$ norm) of the initial data, which is a constant of motion.
Once we adjust the mass, this spectral value becomes irrelevant.

Equation \eqref{pm} is spatially translation invariant. The corresponding
eigenvalue of the linearization is $\lambda_{10}= -n-p$. 
The invariant manifold
is determined by all translations of Barenblatt solutions. For $m > m_1 =
\frac{n-1}{n+1}$, the first moments of the Barenblatt solutions are defined and
conserved, and, by centering $\rho_0$ we get rid of this mode.

Equation \eqref{pm} is also invariant under time translations (which are
equivalent to rescalings of space). The
corresponding spectral value is $\lambda_{01} = -2p$. This
time there is no related conserved quantity, and there seems to be no direct
way of determining the time translation parameter from the initial data
\cite{MR1491842}.
This eigenvalue is responsible for the convergence rate given
in Theorem~\ref{t-1}.
This suggests that we may improve the results of Theorem \ref{t-1}
by modding out the time shift as well, as was achieved in
the entropy sense by the duo \cite{DT11},
and as we hereafter independently achieve in the stronger senses of
Remark \ref{even higher asymptotics} and Theorem \ref{t-superquadratic}.

Indeed, having quantified a rate of contraction,  higher asymptotics become
accessible, as we now explain.  Accounting for the time shift
is merely the first in a series of many improvements
which are possible for $m>m_2$, if we allow ourselves to measure convergence
in weighted norms which suppress information in the tail region
$|\y| \gtrsim \tau^\spread$ and reveal higher modes.
How strongly this information is suppressed depends on the
degree of asymptotic accuracy desired.  It is necessary to introduce weights
since $\lambda_{01}$ turns out to coincide with the threshold of the essential spectrum
in the unweighted H\"older space $C^\halpha(\M)$
of functions on the cigar defined by \eqref{Hoelder norm}. For $m<m_2$
on the other hand, any eigenvalues apart from $\lambda_{01}$ turn out
to be embedded in a spectral continuum and Theorem \ref{fine:p<2}
shows we can obtain very rapid decay, but only for a more restricted
class of initial data, since the weights in this case amplify the
significance of tail information. 

Given $p=2(1-m)^{-1}-n$ and non-negative integers $k,\ell \in \N$,
define
\begin{equation}\label{eigenvalues}
- \lambda_{\ell k} := (\ell + 2k) p + n\ell + 4k(1-\ell-k)
\end{equation}
and
\begin{equation}\label{essential spectral gap}
- \lambda_0^{\rm cont} 
:= (\frac{p}{2} + 1)^2
\;.
\end{equation}
In a certain critically weighted H\"older space,
the $\lambda_{\ell k}$ for $\ell + 2k < 1+ p/2$ will turn out to be
eigenvalues describing the exponential rate of contraction of $u(t,\x)/\uB(\x)$
towards the constant state $1$,  and $\lambda_0^{\rm cont}$ will turn out
to be the threshold of the essential spectrum.
For irrational $m \in `]m_2,1`[$ and any
$\Lambda \in `]\lambda_0^{\rm cont},\lambda_{01}[$,
varying the weights in the linearization carried out below suggests
there exist $u_{\biota}(\x)$ depending on $u(0,\x)$ such that
\begin{equation}\label{conjectured expansion}
\left\|\frac{
(B + |\x|^2)(u(t,\x)/\uB(\x) - 1)
    - \sum_{\Lambda< \biota \cdot \blambda <0}
          u_{\biota}(\x) e^{\biota\cdot \blambda t}}
{(B+ |\x|^2)^{\left(p + 2 - \sqrt{(p+2)^2 +4\Lambda }\right)/4}}
\right\|_{C^\halpha(\M)} = O(e^{\Lambda t})
\end{equation}
as $t \to \infty$,  where  the sum is over integer-valued multi-indices
$\biota = (i_{\ell k})_{\ell,k \in \N}$
constrained so that $i_{00} = 0 = i_{10}$ (assuming the mass and the
center of mass of $u(0,\x)-\uB(\x)$ vanish),
and so that $\biota \cdot \blambda := \sum_{\ell,k=0}^\infty i_{\ell k} \lambda_{\ell k}$
lies strictly between $\Lambda$ and zero.
The presence of essential spectrum also suggests that incorporation of
further terms $u(\x) e^{\lambda t}$ into the sum cannot generally make
the error term smaller than $O(e^{-(\frac{p}{2}+1)^2 t})$.
In this conjecture we see clearly the trade-off introduced by the choice
of weight between weaker norms and the faster rates $\Lambda$ of
decay.  This is a new feature to emerge from the present work, at
least relative to the existing 
entropy-based results of the quartet \cite{BDGV10} and duo \cite{DT11},
and to the conjecture advanced in \cite{MR2126633}.

When $m$ (or equivalently $p$) is rational,  the coefficients
$u_\biota(\x)$ would need to be replaced by a polynomial function
$u_\biota(\x,t)$ of time
in case of an eigenvalue resonance, i.e.
$\biota \cdot \blambda = \biota' \cdot \blambda$
with $\sum i_{\ell k} > 1$, 
as in Angenent \cite{MR956056}. 
The possibility of such resonances can, for example,
be ruled out either by taking $m$ irrational
(in view of the rational dependence of the eigenvalues
\eqref{eigenvalues} on $p$), or by limiting ourselves to a $\Lambda> 2
\lambda_{01}$ within a factor of two of the spectral gap, which is
$\lambda_{01} = \max_{\biota} \biota \cdot \blambda$ for $m>m_2$
and centered mass distributions.
The restriction $\Lambda>2\lambda_{01}$ also simplifies the conjecture in two
other ways,  allowing us to prove the following theorem.
First, it forces the functions $u_{\biota}(\x)/(B+|\x|^2)$
to be eigenfunctions corresponding to the eigenvalue $\biota \cdot \blambda$ of
the linearized operator \eqref{linlop-cartes};
this operator generates the long-time dynamics near the fixed point
in the critically weighted H\"older space $C^\halpha_{\eta_{cr}}(\M)$
defined at \eqref{Cetadef}.
Moreover, since the nonlinearity is analytic,  its effects at this
level are benign:
the quadratic approximation of an orbit by its
tangent is higher order than the accuracy $\Lambda/\lambda_{01}<2$ 
accessible in the following theorem, one of our main results.

\newcount\recalleqn \recalleqn=\csname c@equation\endcsname          
\newcount\recallsec \recallsec=\csname c@section\endcsname           
\newcount\recallthm \recallthm=\csname c@remark\endcsname            

\long\def\rememberthistheorem{
\protect
\begin{theorem}[Higher-order asymptotics in weighted H\"older spaces]
\protect\label{t-subquadratic}
Fix $p= 2(1-m)^{-1} -n>2$ (equivalently $m \in `]m_2,1`[$) and
$\Lambda\in[\lambda_0^{\rm cont},\lambda_{01}]=[-(\frac p2+1)^2,-2p]$
subject to the condition  $2\lambda_{01}<\Lambda$.
If $u(t,\x)$ is a solution to \protect\eqref{transformed}
with center of mass and 
$\lim_{t \to \infty} \|u(t,\x)/\uB(\x) - 1\|_{L^\infty(\Rn)}=0$ both
vanishing, then
there exist a sequence of polynomials $(u_{\ell k}(\x))_{\ell k}$,
each element of which either vanishes or has degree
$\ell+2k \in `]1,\frac{p}{2}+1`[$,
such that
\protect\begin{equation}\protect\label{subquadratic-asymptotics}
\left\|\frac{
(B + |\x|^2)(u(t,\x)/\uB(\x) - 1) -
\sum_{\Lambda< \lambda_{\ell k} <0}
u_{\ell k}(\x) e^{\lambda_{\ell k} t}}%
{(B + |\x|^2)^{%
  \left(p+2 - \sqrt{(p+2)^2 + 4\Lambda}
  \right)/4} }
\right\|_{C^\halpha(\M)}
\!= O(e^{\Lambda t})
\protect\end{equation}
as $t \to \infty$,
where the sum is over non-negative integers $k,\ell \in \bf N$
for which
$\lambda_{\ell k}$ defined by \protect\eqref{eigenvalues} lies in the interval
$`]\Lambda,0`[$, and for which $\ell \le 1$ if $n=1$.
The functions $u_{\ell k}(\x)/(B+|\x|^2)$
lie in the $\lambda_{\ell k}$ eigenspace
of the linear operator \protect\eqref{linlop-cartes} on
$C^\halpha_{\eta_{cr}}(\M)$,
and the norm $\| \cdot \|_{C^\halpha(\M)} \ge
\| \cdot \|_{L^\infty(\Rn)}$ is defined by \protect\eqref{Hoelder norm}.
\protect\end{theorem}
} 

\rememberthistheorem 

\begin{remark}
The coefficients $u_{\ell\,0}(\x)$ 
are given explicitly by Corollary \ref{c-coefficient formulae};
the invariant manifolds corresponding to these particular modes
can be identified using the same idea.
\end{remark}

\begin{remark}\label{eight}
The sum appearing in \eqref{subquadratic-asymptotics} contains up to
eight non-zero terms, depending on $n$, $m_p$ and on $\Lambda$.  These
correspond to eigenvalues from some subset of
$\{\lambda_{01}, \lambda_{02}, \lambda_{03}, \lambda_{11}, \lambda_{12},
\lambda_{20}, \lambda_{21}, \lambda_{30}\}$.
\end{remark}

\begin{proof}[Proof of Remark \ref{eight}]
Formula \eqref{eigenvalues}
implies the monotone dependence of $\lambda_{\ell k} > \lambda_{\ell k+1}$
on $k$ in the admissible range $\ell + 2k \in `]1,\frac{p}{2}+1`[$, so
$\lambda_{\ell k}$ exceeds $`]\Lambda,0`[$ unless
$-\lambda_{\ell\, 0} = \ell(p+n) <  4p = -2 \lambda_{01}$.
This forces $\ell<4$.  On the other hand $\ell + 2k < \frac{p}{2} +1$
implies $(p+n) \ell + 2k (\frac{p}{2} + 1 -\ell) < \lambda_{\ell k}$.
For this to be within $2\lambda_{10}=-4p$ of the origin we need
$k < 4-\ell  - \frac{2\ell^2 + (n-6)\ell + 8}{p- 2\ell + 2} < 4 - \ell$
which establishes the remark
(except for the case $n=1$ which needs to be argued separately).
If $n=3$ and $m =m_{11}$ and $\Lambda/\lambda_{10} \sim 2$ all eight
terms may appear.
\end{proof}

These conclusions translate back to the original variables easily:

\begin{cor}\label{c-subquadratic-rho-asymptotics}
Fix $p= 2(1-m)^{-1} -n>2$ and
$2\lambda_{01}< \Lambda \in [\lambda_0^{\rm cont},\lambda_{01}]$.
If $\rho(\tau,\y)$ is a solution to \eqref{pm}
with vanishing center of mass and satisfying \eqref{reluni}, then
\begin{equation}\label{subquadratic-rho-asymptotics}
\left\|
\frac{ \Dst
\left(\frac{\rhoB(\tau,\origin)}{\rhoB(\tau,\y)} \right)^{\frac{2}{p+n}} 
\left(\frac{\rho(\tau,\y)}{\rhoB(\tau,\y)} - 1 \right) -
\sum_{\Lambda< \lambda_{\ell k} <0}
\frac{u_{\ell k}\big((1+2p\tau)^{-\spread} \y)\big)}%
     {B(1+2p\tau)^{-\lambda_{\ell k}/2p}}  } 
{ (\rhoB(\tau,\origin)/\rhoB(\tau,\y))^{
\big(p+2 - \sqrt{(p+2)^2 + 4 \Lambda  }\big)
/(2p+2n)}
\times \tau^{\Lambda/2p}
}
\right\|_{L^\infty(\Rn)}
\end{equation}
remains bounded as $\tau \to \infty$, 
the notation being the same as in Theorem \ref{t-subquadratic}.
\end{cor}

\begin{proof}
Evaluate \eqref{subquadratic-asymptotics}
using $2p\tau  = e^{2pt}-1$ to get
$$
\Bigg\|\frac{
\frac{B + |\x|^2}{B} 
\left(\frac{u( \frac{1}{2p} \log |1+ 2p\tau|,\,\x)}{\uB(\x)} - 1 \right) -
\frac{1}{B}\sum 
{u_{\ell k}(\x)}{(1+2p\tau)^{\lambda_{\ell k}/2p}}  } 
{((B+|\x|^2)/B )^{\left(p+2 - \sqrt{(p+2)^2 + 4\Lambda}\right)/4}}
\Bigg\|
= O(\tau^{\Lambda /2p})
$$
as $\tau \to \infty$.
Setting $\x = (1+2p\tau)^{-\spread} \y$,
comparison with \eqref{scalingtrf}--\eqref{d-rhoB}
establishes
\eqref{subquadratic-rho-asymptotics}.
\end{proof}

\begin{remark}\label{r-half integer multiples of p}
From \eqref{eigenvalues} and $\spread = \frac{1}{2}(1+\frac{n}{p})$,
only products of powers of $\tau^{-1/2}$ and of $\tau^{-1/2p}$ can arise
as cofactors of the polynomials appearing
in \eqref{subquadratic-rho-asymptotics}.
Since $\Lambda/2p>-2$, we can actually replace both occurrences
of $1+2p\tau$ by $2p\tau$
in \eqref{subquadratic-rho-asymptotics}. 
\end{remark}

Taking $\Lambda=\lambda_{01} =-2p$,
the sum in \eqref{subquadratic-rho-asymptotics}
disappears and we recover an estimate in an unweighted norm,
which is Theorem \ref{t-1}.  By taking $\Lambda< \lambda_{01}$,
we now concretely identify the leading order correction appearing
in the weighted asymptotic expansion.  This yields a corollary analogous
in our framework to the entropic improvement found by Dolbeault and
Toscani \cite{DT11}.

\begin{figure}[hb]
    \begin{picture}(178,160)
      \put(  0,  0){\includegraphics{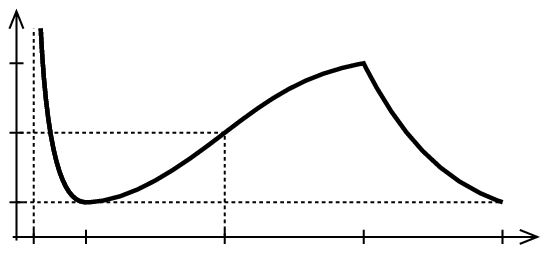}}
      \put(172, 12){$m$}
      \put( 10, 77){$\gamma$}
      \put( 22,  3){$m_0$}
      \put( 37,  3){$m_2$}
      \put( 77,  3){$m_6$}
      \put(117,  3){$m_{n+4}$}
      \put(158,  3){$1$}
      \put( 11, 21){$1$}
      \put( 10, 42){$\frac43$}
      \put( -6, 61){$2\frac{n+2}{n+4}$}
      \put( 35, 61){rate}
      \put(  0,100){\includegraphics{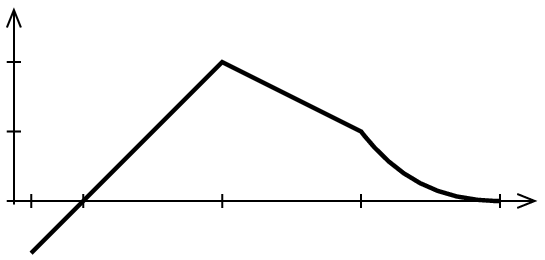}}
      \put(172, 92){$m$}
      \put( 10,153){$\delta$}
      \put( -3,137){$\frac{2}{n+6}$}
      \put( -3,117){$\frac{1}{n+2}$}
      \put(  9, 97){$0$}
      \put( 37, 90){$m_2$}
      \put( 77, 80){$m_6$}
      \put(117, 90){$m_{n+4}$}
      \put( 27,137){weight}
    \end{picture}
\caption{\sl Rate $\gamma$ and weight $\delta$ from Corollary~\ref{c-2} and
  Theorem~\ref{fine:p<2}.
}
\label{graph:t-2}
\end{figure}

\begin{cor}[Second order asymptotics modulo translations and dilations]
  \label{c-2}
\hskip0pt plus 2em 
Suppose $1>m > m_2=\frac{n}{n+2}$  and $\rho(\tau,\y)$ satisfies
the assumptions of Theorem~\ref{t-1}, and the center of mass of~$\rho(0,\y)$
is at the origin.
Then there exists $\tau_0 \in \R$ such that
$$
\limsup_{\tau\to\infty} \sup_{\y} \tau^\gamma
  \left( \frac{\rhoB(\tau,\y)}{\rhoB(\tau,0)} \right)^\delta
   \left|\frac{\rho(\tau, \y)}{\rhoB(\tau-\tau_0,\y)} - 1\right|
< \infty
$$
where, for $n\ge2$, we have
\begin{equation}\label{d-gamma}
\begin{array}{ll}
\gamma  =  \frac{(p+2)^2}{8p} = \frac{[2-(n-2)(1-m)]^2}{8(1-m)[2-n(1-m)]}
    & \Dst\text{ if \ }   m_2 < m \le m_6    \,,
\\[2ex]
\gamma  = \frac{2(p-2)}{p} =  \frac{4-2(n+2)(1-m)}{2-(1-m)n}
    & \Dst \text{ if \ }
       m_6 \le m \le m_{n+4}   \,,
\\[2ex]
\gamma  = \; \frac{n+p}{p} \; =  \frac{2}{ 2-(1-m)n}
    & \Dst \text{ if \ } m_{n+4} \le m < 1   \,,
\end{array}
\end{equation}
and
\begin{equation}\label{d-delta}
\begin{array}{ll}
\delta  = \frac{1}{n+p}(\frac p2-1) = \frac14(2m-n(1-m))
    & \Dst\text{ if \ }   m_2 < m \le m_6    \,,
\\[2ex]
\delta  =  \frac{2}{n+p} = 1-m
    & \Dst \text{ if \ }
       m_6 \le m \le m_{n+4}    \,,
\\[2ex]
\delta  = \frac{1}{n+p}\left(\frac p2-1-\sqrt{(\frac p2-1)^2-2n}\right)
    & \Dst \text{ if \ } m_{n+4} \le m < 1   \;.
\end{array}
\end{equation}
(Recall $m_2=\frac{n}{n+2}$, $m_n=\frac{n-1}{n}$, $m_{n+4}=\frac{n+1}{n+2}$.)

For $n=1$, the first case applies to $m_2<m\le m_{p_*}$, and the third case to
$m_{p_*}\le m<1$, with $p_*=2(\sqrt{2}+1)$, the middle case being omitted.
\end{cor}

\begin{remark}
On compact sets, or sets that grow at rate no faster that $\tau^\spread$,
we therefore get the improved convergence rate of
$O(\tau^{-\gamma})$, whereas an unweighted estimate that is uniform over all
of~$\Rn$ cannot be obtained from this theorem.
\end{remark}

\begin{remark}[Even higher asymptotics]\label{even higher asymptotics}
Here $\gamma = \Lambda/\lambda_{01}$ and $\delta$
correspond to the choice
$\Lambda = \max\{\lambda_{0}^{\rm cont},\lambda_{02},\lambda_{20}\} =: \lambda$
in Corollary \ref{c-subquadratic-rho-asymptotics}. Corollary \ref{c-2}
asserts that with this choice,  a suitable translation of the solution
$\rho(\tau,\y)$ in time makes the summation in
\eqref{subquadratic-rho-asymptotics} vanish.
Assuming $\rho(\tau,\y)$ itself denotes this translation and letting
$u(t,\x)$ be the rescaled solution \eqref{scalingtrf},  hypothesis
\eqref{amplification-hypothesis} of Theorem \ref{t-superquadratic} is
then satisfied,
which allows us to access even more terms in the asymptotic expansion than
provided by Theorem \ref{t-subquadratic}:  up to
$\Lambda > \max\{\lambda + \lambda_{01},\lambda_0^{\rm cont}\}$.
\end{remark}

Looking ahead at Theorem~\ref{spectrum-from-ARMA}, it will be noticed that
$m_n$ is the value where the eigenvalues $\lambda_{10}$ and $\lambda_{01}$
cross, and that $m_{n+4}$ is where the eigenvalues
$\lambda_{20}=-2p-2n$,
$\lambda_{11}=-3p-n+4$
and~$\lambda_{02}=-4p+8$ cross.

\begin{cor}[Third order asymptotics and affinely self similar solutions]
  \label{c-10}
\hskip0pt plus 2em 
Fix $m=m_p = 1 - 2/(p+n)$ with $p > 4$, 
$n \ge 2$,  and let $\rho(\tau,\y)$ satisfy
the assumptions of Theorem \ref{t-1} with the center of mass of~$\rho(0,\y)$
at the origin.  Then a traceless $n \times n$ symmetric matrix $\Sigma_0$ and
function $\sigma(\tau)$ exist satisfying
$(d\sigma/d\tau)^{p+n} = c_B \det \Sigma(\tau)$, where
$\Sigma(\tau) = \Sigma_0 + \sigma(\tau) I \ge 0$ and $c_B$ is a constant
depending on $(n,p,B)$, such that
$\tilde \rho(\tau,\y) =  
\uB(\Sigma^{-1/2}(\tau) \y) \, \det \Sigma^{-1/2}(\tau)$
solves \eqref{pm} for large $\tau$ and
\begin{equation}\label{affine asymptotics}
\limsup_{\tau\to\infty} \sup_{\y} \tau^{\gamma}
  \left( \frac{\rhoB(\tau,\y)}{\rhoB(\tau,0)} \right)^\delta
   \left|\frac{\rho(\tau, \y)}{\tilde \rho(\tau,\y)} - 1\right|
< \infty
\,,
\end{equation}
where $\gamma = -\lambda_{11}/2p$ and
$\delta = \left(
\frac{p}{2}-1 - \sqrt{(\frac{p}{2}+1)^2 + \lambda_{11}}
\right)/(p+n)$.
\end{cor}

A more detailed interpretation of Theorem~\ref{t-subquadratic} will be
given in the next section, after  an exposition of the proofs with
some more technical details. The full proof (relative to all preceding
technical lemmas) will be given in Section~\ref{SecWeights:p>2}. Let
it suffice to say here that there is an intricate interplay between
tail behavior $\delta$ and convergence rate
$\gamma$ that is responsible for the occurrence of the weights
in the theorem.

A partner theorem for the case $m \in `]m_0,m_2`[$ is available below
(Thm.~\ref{fine:p<2}).

Although the results contained herein fall short of 
a full-blown invariant
manifold theory for the fast diffusion equation,
they represent a significant advance in that direction.
Invariant manifolds have a rich history of successful applications
in partial differential equations, as in e.g.~Wayne \cite{MR1465095} and the
references there. However, much of that work has been devoted to 
semilinear heat equations, 
and cannot be directly adapted to the
quasilinear equation now confronted.  Wayne himself raises the question of
whether it is possible to extend such an analysis to nonlinearities which, though
still semilinear, depend on derivatives.  The present work develops the
relevant ideas and techniques, and applies them to
an example which may provide insights and a blueprint for the general problem.
It complements the studies of the porous medium equation $m>1$
by Angenent  \cite{MR956056} \cite{MR0936323} in one-dimension, and
in higher dimensions  by Koch~\cite{KochHabil}.
Still there are some significant ways in which our analysis resembles that
of Wayne \cite{MR1465095} or his results on the convergence of $2D$ Navier-Stokes dynamics
to Oseen vortices with Gallay \cite{MR1912106}:  we study
the size of the solution relative to the fixed point of an appropriately
rescaled dynamics \eqref{scalingtrf} as they do (namely we 
work with $(u-\uB)/\uB$) , and like them we use weighted
spaces to shift the essential spectrum and reveal higher asymptotics in
\eqref{subquadratic-asymptotics}.  However, we are forced to work
in H\"older spaces (like Angenent \cite{MR956056},  but on an unbounded domain),
rather than the Hilbert spaces of Gallay-Wayne --- or for that matter
of McCann-Slepcev \cite{MR2211152}, 
the duo \cite{DT11}, or the quartet \cite{BDGV10} ---
and with a spectral theory for non-self-adjoint operators,
which is more involved than the linear problems analyzed in those works.
Furthermore,  there is a limit $\Lambda \in [\lambda_0^{\rm cont},0]$
to the resolution which we can hope to attain from this linear analysis,  
in sharp contradistinction to the complete asymptotics accessible in the 
problems addressed by Gallay and Wayne, and by Angenent.
Except in the special case $m=m_{-2}$~\cite{BGV10}, we do not know
how dynamical information beyond the threshold imposed by
our continuous spectrum could be resolved.

\section{Overview of Obstructions and Strategies, and Notation}
\label{SecOVNotation}

The basic idea to prove the asymptotic results is taken from dynamical systems
as outlined in \cite{MR2126633}, \cite{MR1465095},  or the references there:
the eigenvalues of the linearization
`ought' to determine the rate of convergence to the equilibrium $\uB$. 
For example, when the analogous smooth finite-dimensional evolution
$$
\frac{d}{dt}x(t) = -V(x(t)) \in \R^n
$$
is linearized around a fixed point $V(x_\infty) = 0$ we get:
$$
\frac{d}{dt}(x(t) - x_\infty)
 =  -DV(x_\infty)  (x(t)-x_\infty) + O(x(t)-x_\infty)^2
\;.
$$
If $\sigma(DV(x_\infty)) = \{ \lambda_1 \le \cdots \le \lambda_n\}$
with eigenvectors $\hat \phi_i$, a naive expectation (albeit not
entirely correct due to the possibility of resonances) is
$$
x(t) = \sum_{i=1}^n c_i \hat \phi_i e^{- \lambda_i t}
+ \sum_{i=1}^n \sum_{j=1}^n 
   c_{ij} \hat \phi_i \hat \phi_j e^{-(\lambda_i+\lambda_j) t}
+ \sum_i \sum_j \sum_k  \cdots
\;.
$$
Notice however,  that we may neglect the iterated summations and 
ignore resonances if we are
content with asymptotics to order $O(e^{- {2} \lambda_1t})$ as $t \to \infty$.
This is the strategy employed in our theorems. 
{\em Differentiability} of $V(x_0)$ or
$x(t)=X(t,x_0)$ with respect to $x_0 \in \R^n$ is crucial.  However for the
infinite dimensional dynamics of $u$, the choice of the right space 
becomes an issue.
Even more than in the semilinear context \cite{MR1210168},
the issue is plagued by conflicting requirements from PDE theory and
from functional analysis.

The Hilbert space setting (a weighted Sobolev space $W^{1,2}_\rho$, which will
be called $W^{1,2}_{\uB}$ here) in which
the spectral analysis was originally carried out in~\cite{MR2126633} is (at
best) inconvenient to deal with the nonlinearity. Its geometric interpretation
in terms of mass transport also restricts the parameter range for~$m$ to the
case where Barenblatt has second moments.

The singularity of the equation near~0 is an obstruction to getting a {\em
smooth\/} semigroup in many Banach spaces: While the parabolic evolution
moves $u$ pointwise away from the singularity 0 immediately, it does not
necessarily move it out of any norm neighborhood of 0, unless weights are built
into the norm that are adjusted to the Barenblatt profile.
On the other hand, as in Angenent \cite{MR956056},
smoothness, rather than mere continuity, of the semiflow
is crucial for the dynamical systems approach via spectrum of a linearization
to have a significance.

The `relative' variable $v=u/\uB$ introduces the appropriate weight, and
V\'azquez' estimate \eqref{reluni} guarantees that after finite time,
every solution is bounded away from 0 in the $L^\infty$ norm
of~$v$. The singularity still shows up in
the fact that the spatial differential operator fails to be {\em uniformly\/}
elliptic in~$\Rn$, due to the decaying Barenblatt whose power multiplies
the Laplacian. But now $\Rn$ can be endowed with a conformally flat
Riemannian metric, such that the Laplace--Beltrami operator with respect to
this metric coincides with the linearized evolution operator in its principal
part.
Working consistently with respect to this metric restores uniform
parabolicity of the equation. $\Rn$ becomes a manifold $\M$ that can be
illustrated by isometrically embedding it into $\R^{n+1}$ like the surface of a
cigar.
$$
\makebox[108bp][l]{\rule{0pt}{51bp}\includegraphics{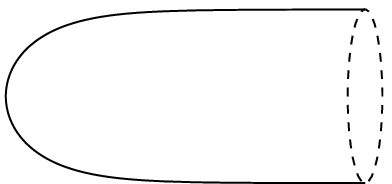}} \quad
\mbox{\raisebox{20pt}{cigar manifold}}
$$
The natural space in which to get a well-posed initial value problem for the
nonlinear equation is a H\"older space, and H\"older quotients need to be
referred to the geodesic distance on the cigar manifold. Likewise, for
derivatives, unit vectors with respect to the Riemannian metric should be used
to get appropriately weighted higher H\"older norms. It is in these
$C^\halpha(\M)$ spaces that we get smoothness of the evolution for the nonlinear
equation.  The detailed discussion of the cigar manifold and the H\"older
spaces are found in Sections~\ref{SecCigarRMfd} and \ref{SecUMHoeld}, respectively.

Section~\ref{SecHeatEq} develops some auxiliary heat estimates which
are used in Section~\ref{SecGlobWP} to carry over
the parabolic Schauder theory for linear and nonlinear equations
from local results in flat space to the uniform setting on the cigar.
(Self-contained proofs for the estimates of Section~\ref{SecHeatEq}
are deferred to an appendix.)
The key point  is that
the cigar can be covered with coordinate charts in which the distortion between
its Riemannian metric and a (local) flat Euclidean metric is bounded with all
derivatives, uniformly over all coordinate charts, even though the global
distortion between the a~priori Euclidean metric on~$\Rn$ and the cigar metric
is of course unbounded. Thereafter, the analytic local
semiflow for the nonlinear equation follows  just like the linear
estimates by contraction mapping, exactly as it would be done in the textbook
model cases on a bounded domain.  Iterative usage of the local estimates can be
controlled by a~priori estimates via the maximum principle to get the global
semiflow.  Let it be stressed that the existence of a semiflow for the FDE has
long been established by Herrero and Pierre \cite{MR797051} for $L^1_{\rm loc}$
data, our concern here is to get  a flow map which depends {\em differentiably\/}
on the initial data.

Before comparing the nonlinear flow with its linearization, let us remember the
obstruction given by V\'azquez \cite[Thm.~1.3]{MR1977429}: There cannot be a
quantitative
convergence rate to Barenblatt for general $L^1$ data. The idea behind this
result is that one can start with small lumps of mass arbitrarily far apart,
and the time until the spreading Barenblatts that arise from these lumps
overlap significantly depends on their distance. So an infinity of lumps (with
finite total mass) can produce arbitrarily slow convergence rates out of
arbitrarily large distances. The spectral analysis two of us developed in
\cite{MR2126633} is unaffected by this obstruction, because it is based on a
linearization within Otto's formalism \cite{MR1842429}, which takes the distance
between lumps into account in the form of transport costs.
Of course, now we have to rephrase the spectral theory within the framework of
the H\"older space $C^\halpha(\M)$, which will introduce some
modifications. However, the key eigenvalues whose geometric meaning was
understood from~\cite{MR2126633} are still  in place, and so is their
interpretation. We do have to cope with tail behavior issues, however, in
particular for small $m$ where the Barenblatt solution has few moments.
The spectral analysis is done in Section~\ref{SecSpecLinEqn}. Much
can be understood in terms of qualitative arguments of asymptotic analysis, only
the precise values of the eigenvalues determining the convergence rates require
the same explicit calculation as in~\cite{MR2126633}. We find it illustrative
to get also the precise spectrum in various weighted H\"older spaces out of the
previously established formulas.

Like for the $L^1$ norm, bounding the H\"older space norm does not prevent us
from starting with lumps that are far away from each other. We therefore cannot
expect a global estimate of the form
$\|u(t)-\uB\|_1\le C(\|u(0)\|_2) \exp[-a t]$
even if we were to attempt to control the constant in terms of a
rather strong norm $\|u(0)\|_2$.   
In contrast to results of the quartet \cite{BDGV10}, our results 
are inherently local. 
It would be interesting to see if Wasserstein distance bounds could
combine our spectral theoretic approach with more global information,
but the functional analytic difficulties to make the `formal infinite
dimensional Riemannian manifold' approach work for the nonlinear
problem have eluded us and appear to be pretty insurmountable. 
The proof of Thm.~\ref{t-1} is carried out as a rather immediate
consequence of the flow properties and spectral analysis developed;
see Sec.~\ref{pthm1}.

Finally, we need to address the influence of tail behavior on the spectral
theory. The Hilbert space in which our first spectral analysis was carried out
carries a power of $\uB$ as a weight, and allows a certain growth rate of
functions in this space if $m>m_2$, while enforcing decay if $m<m_2$. The
growth rates are powers of~$|\x|$, which is exponential in the geodesic
distance~$s$ from the center on~$\M$; and this affects the 
essential spectral radius.
We introduce weighted H\"older spaces $\Ceta^\halpha$, which
allow a spatial growth $|\x|^\eta$  (or equivalently $\cosh^\eta s$
in said geodesic distance from Sec~\ref{SecNLEqCigar}). The
critical growth rate allowed in the Hilbert space is
$\eta_{cr}=\frac{p}{2}-1$ where $p$ is defined by~\eqref{pdef}. In terms of
spectral theory, $C^\halpha(\M)$ is
therefore  closest to the Hilbert space when $m=m_2$, because in this case
$\eta_{cr}=0$.
The eigenfunctions of the linearization have their own characteristic growth
rates (independent of~$m$), which do not automatically coincide with the `no
growth' requirement that came from the need to keep the influence of the
singularity $u=0$ of the nonlinear equation outside a ball around the 
Barenblatt
in the Banach space. There is a trade-off between growth hypotheses and
convergence rates, and once the obvious geometric invariances (space and time
translation) of the unrescaled fast diffusion equation are modded out, finer
asymptotics (like, e.g., from the effect that the diffusion brings anisotropic
initial data closer to the isotropic distribution $\uB$) need to measured in
norms whose weight adjusts to the tail behavior of the linearization in the
corresponding direction in function space, a tail behavior different from the
one of Barenblatt itself. This is what makes Thm.~\ref{t-subquadratic}
and its corollaries look somewhat technical.
But there is good reason to believe that this displays a genuine
phenomenon, not an artefact of the method.
The good news is that, in fact, this discrepancy in tail behavior can be
captured solely in terms of an adjusted weight in the norm in which the
convergence rate is measured, but does not change the convergence rate
itself. This fact arises as a consequence of the maximum principle for the
linearized equation.

Getting precise convergence rates requires an estimate for the semigroup of the
form $O(\exp[\lambda t])$ rather than merely $O(\exp[(\lambda+\eps)t])$.
It is in this step that the tail behavior discrepancy (and its resolution by
juggling the weights) become significant. This is the contents of
Thm.~\ref{semigroupestimates} and its proof. The finer asymptotics provided by
Thm.~\ref{t-subquadratic} and its corollaries
are then proved in Section~\ref{SecWeights:p>2}  as a consequence
in much the same way as
Thm.~\ref{t-1}, after a Lyapunov--Schmidt decomposition of the function space,
with the exponential dichotomy property provided by the spectral gap.

In the following, we will need to rewrite equation \eqref{transformed} in
various forms, bringing to light the various aspects outlined above; different
variable names help to distinguish  the different versions and navigate between
them. The notations collected in Table~\ref{TableNotation} will serve as a
reference in this endeavour.

\begin{table}[htb] \footnotesize
\begin{tabular}{||lp{12cm}||}
\hline
$\rho$  & solution of the unscaled FDE \eqref{pm}
          -- arguments $\tau$ and $\y$  \\
$u$     & solution of the rescaled FDE \eqref{transformed}
          -- arguments $t$ and $\x$\\ 
        & generic name for function on uniform manifold (Sec.~\ref{SecGlobWP})\\
$\uB$   & Barenblatt solution \eqref{barenblatt}, with
          parameter $B$ determining its mass \\
$v$     & $\mbox{}= u/\uB$ solves \eqref{eqv}, \eqref{cyl};
          generic function in heat equation estimates Sec.~\ref{SecHeatEq} \\
$\vb$   & Variable of linearisation with respect to $v$ about 1 \\
$w$     & $w=v-1=u/\uB-1$; generic H\"older function in Sec.~\ref{SecUMHoeld} \\
$v_0$, $\vb_0$
        & initial data at time 0 \\
$\tilde{v}$ & tilde refer to conjugation with a power of $\cosh s$ \\
$\tilde{u}$ & exception: $u=u_0+\tilde{u}$ in Sec.~\ref{SecGlobWP}\\
$v_l,u_l$ & local functions on uniform manifold (pulled back to
         coordinate patch) $v_l=\eta_l u_l$\\
$l$     & index for coordinate patches (distinct from $\ell$!)\\
$\eta_l$ & partition of unity\\
$\chi_l$ & coordinate maps\\
$m_p$   & value of $m$ when $\uB$ has moments below order $p$:
          $m_p=\frac{n+p-2}{n+p}$ \\
$p$     & moment parameter: $n+p=2/(1-m)$ \\
$\Lop$  & linearization operator \eqref{linlop} \\
$\Lop_\eta$ & conjugated linearization operator, see \eqref{conjugated} \\
$\Lup$, $\Lopmod$, $\Lopmod_\eta$
        & various operators with same principal part as $\Lop$, $\Lop_\eta$ \\
$\M$    & cigar manifold \\
$\UM$   & generic uniform manifold (of which $\M$ is a special case) \\
$r$     & $\mbox{}= |\x|$ \\
$s$     & geodesic distance from origin:
          \smash{$B^{1/2}\sinh s= r$};
          $B\cosh^2 s = \uB^{-(1-m)}$%
 \\
$C^\halpha$ & H\"older spaces (space or space-time) --
          Sec.~\ref{SecCigarRMfd} \\
$\Ceta^\halpha$ & H\"older space with weight allowing growth;
         see \eqref{Cetadef} \\
$C_b$     & space of bounded and continuous functions\\
$\halpha$ & H\"older exponent\\
$\spread$ & see \eqref{scalingtrf} \\
$\Qop$  & spectral projection for $\Lop$ or $\Lop_\eta$ on a finite dimensional
      space \\
$\Pop$  & $\mbox{}=1-\Qop$ \\
$\Sop$  & (analytic) linear semigroup generated by $\Lop$  
           -- analogous for $\Sop_\eta$ \\
$\ell$  & angular momentum quantum number -- except for \eqref{cigarmetric} \\
$\lambda_{\ell k}, \lambda_{\ell}^{\rm cont}, \lambda$ 
        & eigenvalue, onset of essential spectrum, spectral parameter \\
$\psi_{\ell k}, v_{\ell k}$ 
        & eigenfunctions in various settings: Thms.~\ref{spectrum-from-ARMA} 
          and \ref{HolderSpectrum} respectively\\
$u_{\ell k}$ & $\mbox{}= v_{\ell k} \uB^{m-1}$\\
$Y_{\ell\mu}$ & spherical harmonics \\
$\eta,\eta_{cr}$
        & parameter for conjugation of operator
          (growth in space); $\eta_{cr}=\frac p2 -1$ \\
$b^\infty(\eta)$ & coefficient of 1st order term of $\Lop_\eta$ at $\infty$ \\
$c^\infty(\eta)$ & coefficient of 0th order term of $\Lop_\eta$ at $\infty$ \\
$T$     & finite time step; we consider time-$T$ maps \\
$\T$    & related to spectral parameter $\lambda$, see
      Fig.~\ref{space+spectrum} on page~\pageref{space+spectrum} \\
\hline
\end{tabular}
\vskip 0.6ex
\caption{\sl Overview of notations used}
\label{TableNotation}
\end{table}

\section{The nonlinear and linear equations in cigar coordinates}
\label{SecNLEqCigar}

We approach the proof of Theorem \ref{t-1} by an analysis  in self-similar
coordinates.

To simplify the discussion we work with the  relative density
$v := u/\uB$
and calculate, using the PDEs for $u$ and $\uB$, 
$$
\begin{array}{r@{}l}
v_t = \uB^{-1}u_t
&\mbox{}
     = \frac1m \uB^{-1}\Laplace(\uB^m v^m)
       + \frac{2}{1-m}\uB^{-1}\Div(\x \uB v)
\,,
\\[2ex]
\uB^{-1}\Laplace(\uB^m v^m)
&\mbox{}
     = \uB^{m-1}\Laplace(v^m)
       + 2 \uB^{-1}\nabla \uB^m \cdot \nabla v^m
       + \uB^{-1}(\Laplace \uB^m) \, v^m
\\[1ex]
&\mbox{}
     = \uB^{m-1}\Laplace(v^m)
       - 2 \frac{m}{1-m}\nabla \uB^{m-1} \cdot \nabla v^m
       - \frac{2m}{1-m} \uB^{-1}\Div(\x \uB) \, v^m
\,,
\\[2ex]
\uB^{-1}\Div(\x \uB v)
&\mbox{}
     = v  \uB^{-1}\Div(\x \uB) + \x\cdot\nabla v
\;.
\end{array}
$$
Now we use Equation \eqref{barenblatt} to get
$$
\uB^{-1} \nabla \cdot (\x \uB)
= n - \frac{2}{1-m}\uB^{1-m} |\x|^2
= n -\frac2{1-m} + \frac2{1-m}\frac{B}{B+ |\x|^2}
\,,
$$
and consequently
\begin{equation}\label{eqv}
\begin{split}
v_t = \mbox{} &
\frac1m \uB^{m-1}\Laplace v^m
+ \frac{2}{1-m}\left(
\x\cdot \nabla (v - 2 v^m) +
\Bigl(n- \frac{2\uB^{1-m}}{1-m}|\x|^2\Bigr)(v-v^m)
\right)
\;.
\end{split}
\end{equation}

This reaction / transport / nonlinear diffusion equation
will play a central role in our analysis.
It will turn out that it is a uniformly parabolic
problem (for $c^{-1} \le v \le c$) on an unbounded
manifold with a cylindrical end, known as the cigar. Its structure can
be seen more easily in polar coordinates. We may and do normalize the
parameter $B$ to $B=1$.

With $|\x|=:r$, we can write $\Laplace =
\d_r^2 + \frac{n-1}{r}\d_r + \frac{1}{r^2}\Laplace_\Sph$, where
$\Laplace_\Sph$ is the Laplace-Beltrami operator on the unit sphere.
So we want to rewrite~\eqref{eqv} as

\begin{equation} \label{polar}
\begin{split}
v_t =\mbox{} &  \textstyle
\frac1m \Bigl( 1 +  r^2\Bigr)
    \Bigl(\frac{\d^2}{\d r^2} v^m + \frac{n-1}r \frac{\d}{\d r} v^m\Bigr)
+ \frac1m \Bigl( \frac1{r^2} + 1 \Bigr) \Laplace_\Sph v^m
 \\
& \textstyle
+ \frac{2}{1-m} r \frac{\d}{\d r} (v-2v^m)
+ \left(\frac{2}{1-m}\right)^2 \Bigl(\frac{n(1-m)}{2} - 1
+ (1 + r^2)^{-1}\Bigr) (v-v^m)
\\ \mbox{}=\mbox{}&\textstyle
\frac1m \frac{\d}{\d r}
\left[ (1+ r^2) \frac{\d}{\d r} v^{m} \right]
+ \frac1m \Bigl( \frac1{r^2} + 1 \Bigr) \Laplace_\Sph v^m
\\&\textstyle
+\frac1m\left[ -\frac{2(1+m)}{1-m} r
     + \frac{n-1}r + (n-1)r
\right]\frac{\d}{\d r} v^m
 \\
& \textstyle
+ \frac{2}{1-m} r\frac{\d}{\d r}  v
+ \left(\frac{2}{1-m}\right)^2 \Bigl(\frac{n(1-m)}{2} - 1
+ (1 + r^2)^{-1} \Bigr) (v-v^m)
\;.
\end{split}
\end{equation}

We set  $r= \sinh s$, hence
$
\frac{\d}{\d r} = \frac{\d s}{\d r} \frac{\d}{\d s} =
 (\cosh s)^{-1}  \,   \frac{\d}{\d s}
$
and
$$
\begin{array}{r@{}l}
\frac1m \frac{\d}{\d r}
\left[ (1+ r^2) \frac{\d}{\d r} v^{m} \right]
= \mbox{}
&
\frac1m (\cosh s)^{-1} \,\frac{\d}{\d s}
\left( (1+\sinh^2 s)(\cosh s)^{-1} \, \frac{\d}{\d s} v^m \right)
\\[1.5ex]
\mbox{} =\mbox{} &
\frac1m \frac{\d^2}{\d s^2} v^m
+ \frac1m  \tanh s \, \frac{\d}{\d s} v^m
\;.
\end{array}
$$
This~$s$ will be  the geodesic distance from
the origin in the Riemannian metric of the cigar manifold introduced
below.
We can rewrite equation \eqref{polar}  as
\begin{equation}\label{cyl}
\begin{split}
v_t = \mbox{} &
\frac1m \left(   \frac{\d^2}{\d s^2}v^m
 +  \frac{2(n-1)}{\sinh (2s)} \frac{\d}{\d s} v^m
 +  (\tanh s)^{-2} \, \Laplace_{\Sph}v^m \right)
\\ &
+ \frac1m
 \Bigl(n-\frac{2(m+1)}{1-m}\Bigr) \tanh s
 \frac{\d}{\d s} v^m
\\  &
+  \frac{2}{1-m}(\tanh s)    \frac{\d}{\d s} v
+ \frac{4}{(1-m)^2}\Bigl(  n\frac{1-m}2 - 1 + (\cosh s)^{-2} \Bigr)
(v-v^m)
\;.
\end{split}
\end{equation}

It is convenient to introduce the operator $\Lop$ as the linearization
about~1 of the operator in \eqref{cyl}. Namely, 
\begin{equation}\label{linlop}
\begin{split}
\Lop := \mbox{}&
     \d_s^2 + \frac{2(n-1)}{\sinh2s}\d_s + (\tanh s)^{-2} \,\Laplace_\Sph
\\&
+ \Bigl( n-\frac{2m}{1-m} \Bigr) \tanh s \, \d_s
+ 2n-\frac{4}{1-m}\tanh^{2} s
\;.
\end{split}
\end{equation}
For later reference, we also calculate~$\Lop$ in cartesian coordinates
from~\eqref{eqv}:
\begin{equation}\label{linlop-cartes}
\begin{array}{r@{}l}
\Lop \vb   & \Dst\mbox{} =
\uB^{m-1}\Laplace \vb
+\frac{2(1-2m)}{1-m} \x\cdot \nabla \vb
+ \Bigl(n- \uB^{1-m}\frac{2}{1-m}|\x|^2\Bigr)2 \vb
\\[1ex]&\Dst\mbox{}
= \uB^{-1}\Div (\uB\nabla(\uB^{m-1}\vb))
\;.
\end{array}
\end{equation}
Using $\Lop$, we will also find it convenient to rewrite \eqref{eqv} 
for $w=v-1$ as
\begin{equation}\label{eqv-principal}
w_t= \Lop h(w) + 
\frac{2}{1-m} \x\cdot\nabla (w-h(w)) 
+ \frac{2}{1-m}
   \Bigl( n - \frac{2\uB^{1-m}}{1-m}|\x|^2 \Bigr) (w-h(w))
\,,
\end{equation}
where $h(w)=\frac{(1+w)^m-1}{m} = w + O(w^2)$.

Later on we will conjugate operator \eqref{linlop} by $\cosh^\eta s$. In order
to keep these calculations at one place we calculate
$$
\cosh^{-\eta}s \circ \d_s \circ\cosh^\eta s = \d_s + \eta \tanh s
$$
and
$$
\cosh^{-\eta}s \circ \d_s^2 \circ \cosh^\eta s= \d_s^2 + 2\eta \tanh s\, \d_s
+ \eta(\eta-1) \tanh^2s + \eta
\,,
$$
which gives after some calculation
\begin{equation}  \label{conjugated}
\begin{split}
   \Lop_\eta := &
 \cosh^{-\eta}s \circ \Lop \circ \cosh^\eta s
\\
  = & \left(\d_s^2 + \frac{2(n-1)}{\sinh2s}\d_s + (\tanh s)^{-2} \, \Laplace_\Sph \right)
   \\
 & \textstyle -2 \left(\frac{1}{1-m} - \frac{n}{2} - 1 -\eta \right) \tanh s \,
  \d_s \\
 & \textstyle
 + \left(\frac{1}{1-m}-\frac{n}{2} - 1 -\eta\right)^2
 - \left( \frac{1}{1-m} -\frac{n}{2} + 1\right)^2 \\
 & \textstyle
   + \Bigl( \left(\frac{1}{1-m} + 1\right)^2  -\left(\frac{1}{1-m} -1 -\eta \right)^2\Bigr)
    \Dst \frac{1}{\cosh^2 s}  
\;.
\end{split}
\end{equation}

Since the leading order terms coincide with the Laplace-Beltrami operator
\eqref{Laplace-Beltrami}, we can see $-\Lop_\eta$ as a Schr\"odinger
operator for a quantum particle moving in a potential-well on the cigar 
manifold,  perturbed by a transport term which spoils self-adjointness.  
The potential-well has a universal $(\cosh s)^{-2}$ profile of depth
$$
d(\eta) = (2(1-m)^{-1}-\eta)(2+\eta)
         =  (p + n -\eta)(2 + \eta)
\,,
$$
and is asymptotic to the constant
\begin{equation}\label{cinfty}
-c^\infty(\eta) = (2(1-m)^{-1} - n -\eta)(2+\eta)
                = (p-\eta)(2+\eta)
\end{equation}
at the $s\to \infty$ end of the cigar.
The transport term shifts mass along the cigar
with outward velocity asymptotic to
\begin{equation}\label{binfty}
-b^{\infty}(\eta) = 2 ((1-m)^{-1} -\eta - 1 - n/2)
=  p - 2\eta -2 
\end{equation}
as $s \to \infty$,
making it harder for the operator to support eigenstates
unless $\eta \ge \frac p2 -1$.

The value $\eta_{cr} = \frac{m}{1-m} - \frac{n}2 = \frac{p}{2}-1$ plays a
special role,  causing the transport term to disappear and restoring the
symmetry property
$$
\int_0^\infty \!\!\int_{\mathbb{S}^{n-1}}  \!\!
(\Lop_{\eta_{cr}}  \vb_1 )\,  \vb_2 \, \tanh^{n-1}s \, d\omega_{n-1} ds
=  \int_{0}^\infty\!\! \int_{\mathbb{S}^{n-1}}  \!\!
\vb_1 \, (\Lop_{\eta_{cr}}  \vb_2 ) \, \tanh^{n-1}s\, d\omega_{n-1} ds
$$
of the operator $\Lop_{\eta_{cr}}$ with respect to
the volume
element \eqref{cigarvol} on the cigar manifold.
Equivalently, $\Lop_0$ is symmetric with respect to
$\uB^{m}$ times the Euclidean volume element in~$\Rn$,
which is the analog of the self-adjointness in \cite{MR2126633}.

This same value $\eta_{cr} = \frac{m}{1-m} - \frac{n}2 = \frac{p}{2}-1$
minimizes $c^\infty(\eta)$, producing a quantity
$$
c^\infty(\eta_{cr})
= - \left(\frac{1}{1-m} - \frac{n}{2} + 1 \right)^2
= - \left(\frac{p}{2}+1\right)^2
$$
which gives the onset of the continuous spectrum,
the negative of what was called $\lambda_0^{\rm  cont}$ in \cite{MR2126633},
see~Thm.~\ref{spectrum-from-ARMA} below.  It combines with depth $d(\eta_{cr})$
to produce a vanishing ground state energy for the Schr\"odinger operator
$\Lop_{\eta_{cr}}$ on $L^2(\M)$; c.f.~Thm.~\ref{HolderSpectrum}(1).
We note that $\eta_{cr}=0$ iff $m= m_2=\frac{n}{n+2}$, it is positive for
$m>m_2$ and negative for $m < m_2$. This is the reason for the case distinction
between Thm.~\ref{t-subquadratic} (proved in Sec.~\ref{SecWeights:p>2}) and
Thm.~\ref{fine:p<2}.

In the remaining sections of this paper we will explore the
consequences of our calculations.  We shall see next section that the
metric defined by the leading part of our operators has a very natural
interpretation as a Riemannian metric on a well-known manifold, the
cigar manifold.

The nonlinear differential equation \eqref{cyl} is (for $v$ bounded above, and
bounded away from 0)
a uniformly parabolic equation on the cigar manifold, and
well-posedness is a standard consequence of estimates for the
linearized equations, which is basically standard parabolic regularity
theory. Therefore we expect that the asymptotic behaviour is controlled by
the spectral properties of the linearized differential equation.
But since the cigar manifold is not
compact, there is nontrivial essential spectrum for the
linear semigroups $\Sop(t) = \exp t\Lop$
and $\Sop_\eta(t) = \exp t \Lop_\eta$ defined through \eqref{linlop}
and \eqref{conjugated}. One may easily read off the
essential spectral radius: It is $e^{c^\infty(\eta)t}$ and it depends
on $\eta$.

We are mainly interested in modes related to eigenvalues outside the
essential spectrum. All eigenvalues and
eigenfunctions have been determined by \cite{MR2126633}, and we can use these
results, with due adaptation, which is done in Sec.~\ref{SecSpecLinEqn}.
In case $\eta=\eta_{cr}$,  we can already anticipate that the number of
eigenvalues will be finite and increase with the depth
$d(\eta_{cr}) =  (\frac{p}{2} + 1 + n)(\frac{p}{2}+1)$
of the potential-well.  In one-dimension the Schr\"odinger operator $\Lop_{\eta_{cr}}$
has been studied in connection with transparent scattering, and
$k(k+1)(\cosh s)^{-2}$ for $k\in\N$ are
known as the Bargmann potentials in that context; a reference for these 
is Sec.~2.6, Exercise 11 of~\cite{Lamb80}, as well as Secs.~2.5 and 3.5 there.
In our case $k=\frac{p}2$.
As mentioned in Section~\ref{SecOVNotation}, we will need the conjugated
operator to accommodate the growth of relevant eigenfunctions.
The dependence of the essential spectral radius on $\eta$ will turn out to be
useful in estimating sharp decay for the linear semigroup.

\section{The cigar as a Riemannian manifold}
\label{SecCigarRMfd}

The cigar is an analytic  Riemannian manifold which we denote by $\M$.
It can be described as $\Rn$ equipped with the metric
\begin{equation}\label{cigarmetric}
d\ell^2_\M := \Bigl(1+|\x|^2\Bigr)^{-1} \sum dx_i^2
= \left( ds^2 + \tanh^2 s\, d\ell_\Sph^2 \right)
\,,
\end{equation}
where $d\ell_{\Sph}^2$ is the length element on the unit sphere and 
$\sinh s = |\x|$ was already introduced above.
It is immediate from this formula that~$s$ is
the geodesic distance from the origin.

In two dimensions, the cigar is a soliton solution for the Ricci flow 
(see ch.2 of~\cite{MR2061425}). For the fast diffusion problem, the
cigar metric is appropriate for various reasons:  The leading part of~$\Lop$ 
\eqref{linlop} coincides with the Laplace-Beltrami operator
\begin{equation}\label{Laplace-Beltrami}
\begin{array}{r@{}l}
\Laplace_{\M} &\Dst \mbox{}
 =
\Bigl(1+|\x|^2\Bigr)^{n/2} \circ \d_i \circ
\Bigl(1+|\x|^2\Bigr)^{-n/2+1} \circ \d_i
\\[1.5ex]&\Dst \mbox{}
=
 \d_s^2 + \frac{2(n-1)}{\sinh 2s} \d_s +
  (\tanh s)^{-2}\Laplace_{\Sph} 
\,,
\end{array}
\end{equation}
which is also the leading part of the operator  $\sum_{i=1}^n X_i^2$,
where $X_i$ are the obvious orthonormal vector fields
$$
X_i := \Bigl(1+|\x|^2\Bigr)^{1/2} \e_i
$$
and $\e_i:=\d/\d x_i$ is the standard basis.
Specifically,
$$
\sum_{i=1}^n X_i^2 = \Laplace_{\M} + (n-1)\tanh s\, \d_s
\,.
$$

As a consequence, we shall see that the parabolic estimates become uniform when
distances are measured with respect to the cigar metric.

The Riemannian volume element is
\begin{equation} \label{cigarvol} \textstyle
d\mu :=  \tanh^{n-1}s\, ds d\omega_{n-1}
\;.
\end{equation}
The vector fields $X_i$ do not commute, but
their commutators of any order can be written as a bounded linear combination
of these vector fields. For instance,
$$
[X_i, X_j]=
\frac{x_i}{(1+|\x|^2)^{1/2}} X_j -
\frac{x_j}{(1+|\x|^2)^{1/2}} X_i
\;.
$$

\section{Uniform manifolds and H\"older spaces }
\label{SecUMHoeld}

We will need uniform Schauder estimates for parabolic equations on~$\M$. The
possibility of such {\em uniform\/} estimates relies on the fact that {\em
  locally\/}, $\M$ can be mapped into~$\Rn$ with bounded distortion, the bound
being global, even though a global map with bounded distortion is not possible
between $\M$ and $\Rn$. It seems expedient to elaborate on this principle in
its natural generality. So we will prove these estimates for parabolic
equations on a uniform manifold~$\UM$, of which the cigar~$\M$ is the example
we are interested in: We study
manifolds $\UM$ (not necessarily Riemannian) with a distance $d$ which turns
$\UM$ into a geodesic  space: A metric space $\UM$ is called a geodesic
space if, given two points $x,y\in\UM$  there exists a path
$\gamma:[0,1]\to\UM$ from $x$ to $y$ such that $d(\gamma(s),\gamma(t))
= |s-t| d(x,y)$. 
This requirement is in particular satisfied
for the geodesic distance of a connected closed Riemannian manifold.

\begin{definition}[Uniform manifold] \label{Def-UM}
Let $\UM$ be a manifold with a metric $d$ which turns $\UM$ into a geodesic
space.
We say $(\UM,d)$ has a uniform $C^k$ structure (resp.\ uniform
analytic structure) if there exist two constants $R>0$ and
$C>0$ and coordinate maps $\chi_x: B_R(x) \to \Rn$ for each $x\in\UM$
such that
\begin{equation}\label{bilipschitz}
C^{-1} d(y,z) \le |\chi_x(y) - \chi_x(z)| \le C d(y,z)
\quad \text{ for all } x \in \UM,\; y,z \in B_R(x)
\end{equation}
with $C^k$ (resp analytic) coordinate changes
$\chi_y \circ \chi_x^{-1}$ satisfying
\begin{equation} \label{analytic}
|\d^\beta (\chi_y \circ \chi_x^{-1})|
\le
C^{|\beta|+1} \beta! \qquad \text{ in  }
\chi_x ( B_R(x) \cap B_R(y))
\end{equation}
for any multi-index~$\beta$ of length $|\beta|\le k$ (resp.\ for any
multi-index). We call $R$ (and $C$) a radius (and constant) of
uniformity for~$\UM$.
\end{definition}

The first condition says that balls of radius $B_R$ are uniformly
bilipschitz equivalent to subsets of $\Rn$. The second
condition \eqref{analytic} implies that the coordinate maps are
uniformly $C^k$ (resp.\ analytic).

For our purposes, $C^3$ would sufficient regularity, but
the cigar manifold is even analytic; and
it is not hard to see that the exponential map provides such
coordinate maps for the cigar manifold with $R=1$ and a suitable
constant~$C$.  Observe that we do not require a Riemannian structure
on the manifold. But if the manifold carries  a Riemannian metric
then we use the associated metric, which  turns the manifold into a geodesic
space.

The following lemma implies that the volume of balls is at most exponential in
the radius.
Given an open subset $A$ of $\UM$ and $r>0$, we denote by $N(A,r)$ the maximal
number of disjoint balls of radius $r$ in $A$ (if such a maximum exists; else
$N(A,r):=\infty$).

\begin{lemma}[Packing $r$-balls into $\rho$-balls] \label{Lem-ballpacking}
Suppose that $\UM$ is a uniform manifold of dimension~$n$ with a radius and
constant of uniformity~$R$ and~$C$.
Then  the following bound  holds
$$
N(B_\rho(x_0),r) \le \left\{
\begin{array}{ll}
\left( \frac{C^2\rho}{r}\right)^n
                & \text{ if } r < \rho \le R \,,
\\
 ( \max\{C^2R/r,1\} )^n     (5C^2)^{5n\rho/R}
                & \text{ if } \max\{r,R\} <\rho \;..
\end{array} \right.
$$
\end{lemma}
\begin{proof}

Suppose first that $r<\rho\le R$ and consider a ball
$B_r(x_1)$ contained in  $B_\rho(x_0) \subset B_R(x_0)$. Then
$$
B_{r/C}( \chi_{x_0}(x_1)) \subset \chi_{x_0}(B_r(x_1)) \subset \chi_{x_0}(B_\rho(x_0))
\subset B_{C\rho}(\chi_{x_0}(x_0))
\;.
$$
Comparing the volumes of the images we see that there can be at most
$(C^2 \rho/r)^n$
such disjoint balls in $B_\rho(x_0)$.

Secondly,  let's suppose $r=\frac{R}5$, and $\rho$ arbitrary. The previous case
$\rho\le R$ will serve as the beginning of an induction with steps from $\rho$
to $\rho+r$. So suppose $\{B_r(y_j)\mid j=1,\ldots,N\}$ is a maximal packing of
$r$-balls in $B_\rho(x_0)$. Then the $B_{2r}(y_j)$ will cover $B_{\rho-r}(x_0)$
as a consequence of the maximality of the packing. But then the $B_{4r}(y_j)$
will cover $B_{\rho+r}(x_0)$.
This is because for any $z\in B_{\rho+r}(x_0)$, we can find a $z'\in
B_{\rho-r}(x_0)$ that has distance  $2r$ from $z$; and then
$z\in B_{4r}(y_j)$ whenever $z'\in B_{2r}(y_j)$. Having constructed this
covering, we now take a family of disjoint balls $\{B_r(z_i)\mid i\in I\}$ in
$B_{\rho+r}(x_0)$. As each $z_i$ must be in some $B_{4r}(y_j)$, the $B_r(z_i)$
lies in $B_{5r}(y_j)$. As there are at most $(5C^2)^n$ such $B_r(z_i)$ in each
given $B_{5r}(y_j)$, and there are only $N$ many $y_j$'s, we conclude
$N(B_{\rho+r}(x_0), r) \le (5C^2)^n N(B_\rho(x_0),r)$.

We have thus proved inductively that
$N(B_\rho(x_0),R/5) \le (5C^2)^{5n\rho/R}$.

Now we turn to the general case: If $r>\frac{R}{5}$, the estimate for
$r=\frac{R}{5}$ is still valid trivially. If $r<\frac{R}{5}$,
suppose we have a maximal collection of $B_r(x_i)$ in $B_\rho(x_0)$.
We also
place a maximal collection of $B_{R/5}(y_j)$ in $B_\rho(x_0)$. As before,
the maximality then ensures that the collection of $B_{3R/5}(y_j)$ covers
all of $B_\rho(x_0)$ and each $B_r(x_i)$ is contained in some
$B_{r+3R/5}(y_j)\subset B_{4R/5}(y_j)$.

By our second estimate, there are no more than $(5C^2)^{5n\rho/R}$ of the
$B_{R/5}(y_j)$ fitting in $B_{\rho}(x_0)$, thus bounding the number of $y_j$,
and by the first estimate no more than $(C^2R/r)^n$ of the $B_r(x_i)$ fitting
in each $B_{4R/5}(y_j)$.
\end{proof}

As a consequence, given $r<R/3$ there is a sequence of points $(x_j)$
with distance at least $r$ such that $\UM$ is covered
by the balls $B_{3r}(x_j)$.
There  are coordinate maps $\chi_{x_j}$ defined on these balls since $r<R/3$.
No point lies in more than $(3C^2)^n$ of the balls $B_{3r}(x_j)$.
We fix $r$ and the points $x_j$ in the sequel. There is a partition of
unity $(\eta_l^2)$ subordinate to this covering with uniform bounds on
derivatives (up to fixed order) of $\eta_l \circ \chi_{x}^{-1}$
on $\chi_{x_l}(B_R(x_l))$; we refer to this property as
`uniform smoothness'.

We may choose such a uniformly smooth partition of unity.
The structure of uniform manifolds  allows to construct useful functions,
specifically a smooth approximation to the radial coordinate $d(\cdot,x_0)$:

\begin{lemma}[Approximate radial coordinate on a uniform manifold] \label{growth}
Choose $x_0 \in \UM$. There exists a function $\rho:\UM\to\R$
 with bounded derivatives such that $ \rho  (x) -d(x,x_0)$ is bounded.
\end{lemma}
\begin{proof}
We define the functions $\rho_l$ on the $l^{\rm th}$ coordinate chart
$\chi_{x_l}(B_R(x_l))$  as regularizations of
$(\eta_l \circ\chi_{x_l}^{-1}) \cdot
(d(\chi_{x_l}^{-1}(\cdot),x_0)-d(x_l,x_0))$,
and then we can take
${ \rho = \sum \eta_l \cdot (\rho_l \circ \chi_{x_l} +  d(x_l,x_0) ) }$.
\end{proof}

Given a uniformly smooth partition of unity $\eta_l^2$ (as explained before
the lemma) on the uniform manifold,
we define H\"older norms
$$
\| f \|_{C^{\halpha}(\UM)} := \sup_l
 \| (\eta_l f) \circ \chi_{x_l}^{-1} \Vert_{C^\halpha
   (\chi_{x_l}(B_R(x_l)))}
\,,
$$
where on subsets $X \subset \Rn$,
\begin{equation}\label{Hoelder norm}
\| f \|_{C^\halpha(X)} := \max\{ [f]_\halpha, \|f\|_{L^\infty} \}
\end{equation}
with
$$
[f]_\halpha:=
\sup_{0<|x'-x|\le 1}
\frac{ |f(x')-f(x)|}%
     {|x'-x|^\halpha}  
\;.
$$
This definition depends on the choice of the
partition of unity, but different choices lead to equivalent norms.
The presence of the $L^\infty$-norm in~\eqref{Hoelder norm} permits to drop the
constraint $|x'-x|\le1$ altogether in the definition of~$[f]_{\halpha}$, or
else to replace it with a different bound $|x'-x|\le R$, again providing
equivalent norms.
Likewise, it is easy to see that an equivalent norm can be defined without
reference to a partition of unity by
\begin{equation}\label{Hoelder-norm-global}
\|f\|^\circ_{C^\halpha(\UM)} := \max\{ \|f\|_{L^\infty(\UM)} , [f]^\circ_\halpha \}
\end{equation}
with
$$
[f]^\circ_\halpha := \sup_{0<d(x',x)}
\frac{ |f(x')-f(x)|}%
     {d(x',x)^\halpha}  
\;.
$$
We will use the variant \eqref{Hoelder norm} for the general proof
of the Schauder estimates, but the equivalent variant
\eqref{Hoelder-norm-global}, when dealing with $\M$ specifically. In this case,
$d$ obviously refers to the cigar metric on~$\M$, not the euclidean metric
on~$\Rn$.

We can similarly define the spaces $C^{k,\halpha}$ by the maximum of the
$C^\halpha$ norms \eqref{Hoelder norm} of all the derivatives up to order~$k$.

H\"older spaces $C^\halpha([0,\infty)\times \UM)$  on
space-time cylinders can be understood  using the parabolic metric
$$
d_P \bigl( (t_1,x_1)\,;\;(t_2,x_2)\bigr) =
\max \bigl\{ |t_2-t_1|^{\frac12} ,
d(x_1, x_2) \bigr\}
\;.
$$
This leads to
$$
\| f \|_{C^{\halpha}([0,T]\times\UM)} := \sup_l
 \| (\eta_l f) \circ \chi_{x_l}^{-1} 
             \|_{C^\halpha([0,T]\times\chi_{x_l}(B_R(x_l)))}
\,,
$$
where on subsets $X \subset \Rn$\,
\begin{equation}\label{xtHoelder norm}
\| f \|_{C^\halpha([0,T]\times X)} :=
\max\{ [f]_{x;\halpha}, [f]_{t;\halpha/2},  \|f\|_{L^\infty} \}
\end{equation}
with
$$
\begin{array}{l}
[f]_{x;\halpha} := \sup_t
\sup_{0<|x'-x|\le 1}
\frac{ |f(t,x')-f(t,x)|}%
     {|x'-x|^\halpha}
\;,
\\[1ex]
[f]_{t;\halpha/2} := \sup_x
\sup_{0<|t'-t|\le 1}
\frac{ |f(t',x)-f(t,x)|}%
     {|t'-t|^{\halpha/2}}
\;.
\end{array}
$$
Similar comments about equivalent norms in the style of
\eqref{Hoelder-norm-global} apply in an obvious manner.
In the literature, these spaces are sometimes  denoted as
$C^{\halpha/2,\halpha}$, and slight differences in the definitions lead to
different, but trivially equivalent, norms.

We will henceforth find the abbreviations
\begin{equation} \label{defMT,UMT}
\UMT := [0,T]\times\UM  \;,\quad  \MT := [0,T]\times\M
\;,\quad \RnT := [0,T]\times\Rn
\end{equation}
useful.

Using multi-index
notation $\beta:= (\beta_0,\beta_1,\ldots,\beta_n)$ with the parabolic weight
$|\beta|:= 2\beta_0 + \beta_1+\ldots+\beta_n$ and
$\d^\beta:=
\d_t^{\beta_0} X_1^{\beta_1} \cdots X_n^{\beta_n}$,
we let, for functions on~$\Rn$,
$$
\|f\|_{C^{k,\halpha}} := \max \left\{
\sup_{|\beta|\le k} \|\d^\beta f\|_{C^\halpha} ,
\sup_{|\beta|\le k-1} \sup_{x, |t_1-t_2|\le 1}
   [\d^\beta f]_{t;(1+\halpha)/2}
\right\}
$$
and can obtain $C^{k,\halpha}$ norms on~$\UMT$ or on~$[0,\infty`[\times\UM$ by
partitions of unity as before.

We note that not all authors include the term
$[\d^\beta f]_{t;(1+\halpha)/2}$
in their definition of the parabolic H\"older
norm (e.g., Krylov \cite{MR1406091} and Friedman \cite{MR0181836} don't, but
Ladyzhenskaya et al.~\cite{MR0241822} do). For $k=2$, this term is
indeed controlled by the other terms, as can be seen in Krylov's book
from his Ex.~8.8.6 in connection with the proof of Thm.~8.8.1.
The argument will generalize for even~$k$. The non-equivalence of the two
styles of parabolic H\"older norms for odd~$k$ will not be an issue for our
purposes.
We find it convenient to repeat the simple norm equivalence argument here,
for easy reference.
It suffices to do the argument, which is
local, in flat space because of the local equivalence of the metrics.
For the rest of the section, we simplify notation by giving the arguments for
one space dimension (writing a scalar $x$), with the generalization to higher
dimensions being obvious. We estimate
$$
\begin{array}{l}\Dst
|w_x(t,x)-w_x(s,x)| \le \mbox{}
\\[1.5ex]\Dst \kern 2em \le
\left| w_x(t,x) - \frac{w(t,x+h)-w(t,x)}{h} -
       w_x(s,x) + \frac{w(s,x+h)-w(s,x)}{h} \right|
\\[1.5ex]\Dst \kern 3em
+
\left| \frac{w(t,x+h)-w(t,x)}{h} - \frac{w(s,x+h)-w(s,x)}{h} \right|
\\[1.5ex]\Dst \kern 2em
\le
h \int_0^1\int_0^1 |w_{xx}(t,y+lrh)-w_{xx}(s,y+lrh)|\,l\,dl\,dr
\\[1.5ex]\Dst \kern 3em
+
|t-s| \, \left| \frac{w_t(\theta,x+h) - w_t(\theta,x)}{h} \right|
\,,
\end{array}
$$
where in the second term the mean value theorem was applied to the function
$g(t):= \frac{w(t,x+h)-w(t,x)}{h}$.
Using the H\"older estimates for $w_{xx}$ and $w_t$, we conclude
$$
|w_x(t,x)-w_x(s,x)| \le
h|t-s|^{\halpha/2}\, [w_{xx}]_{t; \halpha/2} +
|t-s| h^{\halpha-1}\, [w_t]_{x;\halpha}
\;.
$$
We now obtain
$
|w_x(t,x)-w_x(s,x)| \le  |t-s|^{(1+\halpha)/2}
([w_{xx}]_{t; \halpha/2} +  [w_t]_{x;\halpha})
$
by letting $h:= |t-s|^{1/2}$.

We will heavily rely on spaces of H\"older continuous functions. They
have nice algebraic properties: products and composition have good
properties in these spaces and there are optimal estimates for linear
parabolic equations in these function spaces. We begin with the
algebraic side, for use in the next section.

\begin{lemma}[Estimates for some nonlinearities in H\"older spaces]
  \label{calculus}
  For $D\subset\R^2$ open and convex, $h\in C^{k+1}(D)$ and a pair of functions
  $(f,g)$ with range in a compact
  subset of~$D$, the following estimates hold on $\M$ and on $\MT$:
  \begin{equation}\label{c1}
  \| h\circ (f,g) \|_{C^{k,\halpha}} \le
  C\left( \| f \|_{C^{k,\halpha}} , \| g \|_{C^{k,\halpha}}
  \right)
  \end{equation}
  where bounds on $h$ and its derivatives on the convex hull of the range of
  $f$ and $g$
  enter into the dependence of~$C$ on the norms $\|f\|$ and $\|g\|$;
  \begin{equation}\label{c2}
  \| fg \|_{C^{k,\halpha}} \le c
   \| f \|_{L^\infty} \| g \|_{C^{k,\halpha}} +
  c\| g\|_{L^\infty} \| f \|_{C^{k,\halpha}}
  \;;
  \end{equation}
  \begin{equation}\label{c3}
  \| h\circ (f,g) g^2 \|_{C^{k,\halpha}} \le
  C( \| f\|_{C^{k,\halpha}}, \| g \|_{C^{k,\halpha}})
  \| g \|_{L^\infty} \| g \|_{C^{k,\halpha}}
  \;.
  \end{equation}
  If $h$ is analytic, then the map
  \begin{equation} \label{c4}
     C^{k,\halpha}\times C^{k,\halpha} \ni (f, g) \longmapsto h(f,g) g^2
     \in C^{k,\halpha}
  \end{equation}
  is an  analytic map of Banach spaces.
\end{lemma}

\begin{remark}
The lemma  will be applied in a slightly more general situation with $h$
depending also on $x$ and $t$. This however follows from the lemma above
because the statement remains correct if we consider vector valued functions
$f$ and $g$, and we may replace $f$ by the vector $(t,x,f)$ for bounded
functions $x$ and $t$. Typically we will consider local coordinates with
uniform estimates for all quantities which occur.
\end{remark}

\begin{remark}[Notions of Analyticity] 
  There are different equivalent notions of analyticity
  of maps defined on open sets in real or complex Banach spaces. We
  refer to section 15.1 in chapter 4 of \cite{MR787404} for a detailed
  exposition of the facts stated below.  Let $X$ and $Y$ be complex
  Banach spaces and $U \subset X$ an open subset. A map $f: U \to Y$
  is called weakly holomorphic if it is locally bounded and if for every
  one dimensional affine subspace $L \subset X$ and every $l \in Y^*$
  the map $ L \cap U \ni x \longmapsto l(f(x)) \in \C$ is complex
  differentiable. It is an immediate consequence (see exercise 2 in
  section 15 of \cite{MR787404}) of the Cauchy integral formula that a
  weakly holomorphic function is continuously differentiable and
  complex differentiable. Locally bounded complex differentiable
  functions are analytic in the sense that the Taylor expansion
  converges in a ball. Real analyticity can be defined in terms of
  convergent Taylor expansions. Real analytic functions can be
  extended to complex analytic functions, and conversely, any
  restriction to a real subspace (if $X$ is the complexification of a
  real Banach space) of the real part of a holomorphic function
  (assuming that $Y$ is the complexification of a real Banach space)
  is analytic.  Moreover both the real analytic and complex inverse
  function theorem holds. In the sequel we do not distinguish between
  different definitions of analyticity and use whatever is convenient.
  The results of this paper do not depend heavily on the notion of
  analyticity. Its main purpose is a nontechnical approach to
  differentiability and higher differentiability in the previous
  lemma.
\end{remark}

\begin{proof}
The first inequality \eqref{c1} requires a standard
calculation.
We may restrict ourselves to $h\circ f$ and verify
$$
\| h\circ f \|_{C^{k,\halpha}} \le C (\| f \|_{C^{k,\halpha}})
$$
by allowing vector valued
functions $f$. It is a routine induction over $|\beta|$
that for any multiindex $\beta$
with $|\beta |\le k$ all terms of $\d^\beta h(f)$ can be estimated
in the desired way.

We  get the $k=0$ case of~\eqref{c2} (which is the most important
for us) from
$$
f(x)g(x)-f(x')g(x') = f(x) (g(x)-g(x')) + (f(x)-f(x')) g(x') \;.
$$
The case $k\ge1$ of \eqref{c2} is more subtle
and is of an interpolation type, since naturally occurring intermediate
norms as, e.g., $\|f\|_{C^1}\|g\|_{C^{0,\halpha}}$ must implicitly be controlled
by the extreme norms on the right hand side. It suffices to prove \eqref{c2}
for fixed compact support, since the general case can be reduced to this
special one by means of a locally finite partition of unity.
In that case however, due to the local equivalence of metrics, one can prove
the estimate in flat space. A tedious, but elementary, proof can be given
inductively. For instance, to estimate a term like
$\|f'\|_{L^\infty}[g]_{\halpha}$, one can use interpolation inequalities as
Krylov's \cite{MR1406091} Thm.~3.2.1 in the elliptic case, or Thm.~8.8.1 in the
parabolic case;  the scaling weight~$\eps$ contained therein will be chosen
in terms
of the norms, such as to remove one higher norm from products like
$\eps^{1+\halpha}\|f\|_{C^{1,\halpha}}\|g\|_{C^{1,\halpha}}$.

A less tedious,  but not as elementary, proof in the
case of $\Rn$ can be found in the book of  Runst and Sickel \cite{MR1419319},
Sec.~4.6.4, and the argument there generalizes to our setting.
(Note that their setting in Besov spaces includes H\"older spaces
$C^{k,\halpha} = B^{k+\halpha}_{\infty,\infty}$.)
The Littlewood-Paley decomposition (dyadic partition of unity in Fourier space)
that enters implicitly already in the definition of Besov spaces (pg.~8
of~\cite{MR1419319})  needs to be made anisotropic to reflect
the parabolic scaling.
We omit tedious details, which would only distract from the ideas of the
present paper.

Inequality \eqref{c3} is a consequence of the previous bounds:
$$
\| h\circ(f,g) g^2 \|_{C^{k,\halpha}}
\le
c \| h\circ(f,g) \|_{L^\infty} \| g \|_{L^\infty} \| g \|_{C^{k,\halpha}}  +
c \|  h\circ(f,g) \|_{C^{k,\halpha}} \| g \|_{L^\infty}^2
\;.
$$

Now suppose that $h$ is analytic. It can be extended to a holomorphic
complex  map.  We complexify the H\"older spaces by considering functions with
complex values.
Left composition with $h$ is clearly a weakly holomorphic map, hence
holomorphic. The restriction to real valued functions is therefore real
analytic.
\end{proof}

\begin{remark}
The reason why we insist on~\eqref{c2} rather than the simpler algebra property
$\|fg\| \le c\|f\|\,\|g\|$ in $C^{k,\halpha}$ is the following: While the
H\"older spaces are needed for a clean PDE theory,  we want to prove a
theorem whose natural hypothesis is about the $L^\infty$ norm, rather than
having to introduce an unnatural smallness condition in a H\"older norm (see
the proof of Thm.~\ref{t-3}). In particular, when estimating the nonlinearity,
we need to use a~priori smallness in the $L^\infty$ norm to avoid the
singularity at $u=0$, but may need to work in a
subset of $C^\halpha$ that does not impose smallness of the $C^\halpha$ norm.
\end{remark}

We will also use a variant of Lemma~\ref{calculus} that deals with certain
weight functions:
\begin{lemma}[Weighted product estimate]\label{weighted-algebra}
  Let $f\in C^1(`]-a,a`[\to\R)$ with $f(0)=0$ and
  let $0<h\in C^1(\Rn\to\R)$ be a
  positive weight function with $|\nabla\ln h|$ bounded. Then
  $$
  \|(f\circ w)wh\|_{C^\halpha(\Rn)} \le
     C 
       \|w\|_{L^\infty(\Rn)} \|wh\|_{C^\halpha(\Rn)}
  $$
  where the constant $C$ depends on the $C^1$ norm of~$f$ and $|\nabla\ln h|$.
  The same estimate holds with space-time norms over $\RnT$, provided
  $|\d_t\ln h|$ is also bounded.
\end{lemma}
The crucial issue here is that there is no second term in which the H\"older
norm would apply to  the unweighted function.
Specifically, we are after the following
\begin{cor}\label{weighted-algebra-cigar}
  Let $h(x):= (\cosh s)^{-\eta}$ on the cigar manifold $\M$. Under the same
  hypotheses as in Lemma~\ref{weighted-algebra}, we have
  $$
    \|(f\circ w)wh\|_{C^\halpha(\MT)} \le
       C
         \|w\|_{L^\infty(\MT)} \|wh\|_{C^\halpha(\MT)} \;.
  $$
\end{cor}

\begin{proof}[Proof of Lemma~\ref{weighted-algebra} and
    Cor.~\ref{weighted-algebra-cigar}]
  Clearly $\|(f\circ w)wh\|_{L^\infty}$ is estimated in the claimed way; so we
  only need to estimate the H\"older quotient.
  We estimate, abbreviating $w(x_i)=: w_i$ and $h(x_i)=:h_i$,
  $$
  \begin{array}{l}
  (w_1-w_2)(f(w_1)w_1h_1 - f(w_2)w_2h_2) =  \mbox{}
  \\[1ex] \kern3em \mbox{}
  (f(w_1)w_1-f(w_2)w_2)(w_1h_1-w_2h_2)
  + (h_2-h_1)w_1w_2(f(w_1)-f(w_2))
  \;.
  \end{array}
  $$
  With no loss of generality we assume $h_2\le h_1$. At least if $w_1\neq w_2$,
  we can obtain
  $$
  \begin{array}{l}\Dst
  \frac{f(w_1)w_1h_1 - f(w_2)w_2h_2}{|x_1-x_2|^\halpha} =  \mbox{}
  \\[2ex] \Dst\kern3em \mbox{}
  \frac{f(w_1)w_1-f(w_2)w_2}{w_1-w_2}\,\frac{w_1h_1-w_2h_2}{|x_1-x_2|^\halpha}
  \\[2ex]\Dst \kern4em \mbox{}
   + h_1w_1 w_2 \frac{f(w_1)-f(w_2)}{w_1-w_2}\,
     \frac{\exp[\ln h_2-\ln h_1]-1}{|x_1-x_2|^\halpha}
  \;.
  \end{array}
  $$
  We then get the estimate
  $$
  \begin{array}{l}\Dst
  \frac{|f(w_1)w_1h_1 - f(w_2)w_2h_2|}{|x_1-x_2|^\halpha} \le  \mbox{}
  \\[2ex] \Dst\kern5em \mbox{}
  \left(\sup_{|r|\le\|w\|_{L^\infty}}\Bigl|\frac{d}{dr}(f(r)r)\Bigr| \right)\,
  \frac{|w_1h_1-w_2h_2|}{|x_1-x_2|^\halpha}
  \\[2ex]\Dst \kern6.5em \mbox{}
   + \|wh\|_{L^\infty}\|w\|_{L^\infty}
      \left(\sup_{|r|\le\|w\|_{L^\infty}}|f'(r)|\right)\,
     \frac{\min\{1, |\ln h_2 - \ln h_1|\}}{|x_1-x_2|^\halpha}
  \\[2ex] \Dst\kern3em \mbox{}
  \le
  C \|w\|_{L^\infty}\, \|wh\|_{C^\halpha} +
  Ca \|w\|_{L^\infty}\, \|wh\|_{L^\infty} \, \min_{d>0}\Bigl\{d^{-\halpha},
  \|\nabla\ln h\|_{L^\infty} d^{1-\halpha}\Bigr\}
  \;.
  \end{array}
  $$
  The estimate persists by continuity if $w_1=w_2$.
  If space-time norms are desired, the same estimate applies to time H\"older
  quotients. This proves the lemma.

  For the corollary on the cigar manifold, we only have to replace $|x_1-x_2|$
  with $d(x_1,x_2)\ge |s_1-s_2|$, according
  to~\eqref{cigarmetric}.
\end{proof}

\section{Schauder estimates for the heat equation}
\label{SecHeatEq}

The results and methods in this chapter are basically well-known. However, we
need to elaborate on some details since we are also relying on some less
familiar  weighted norms.

Also the impact of low regularity inhomogeneities
needs to be taken into account cleanly to cope with applications to the
quasilinear case.
We will use these estimates in the next section to get local existence for
parabolic equations on uniform manifolds. Again, this will be done along
traditional lines; however to control the noncompactness of the domain by means
of uniform estimates requires recalling  the proof details from the
traditional case for inspection and reference.

Some comments concerning the semigroup point of view and the issue of
big vs.\ little H\"older spaces will be given in
Rmk.~\ref{analytic-semigroup} below; suffice it to say here that these
issues are not of any real significance for our purposes.

So we begin by studying the heat equation
\begin{equation}   \label{heateq}
 v_t - \Laplace v = \d_{ij}^2 f^{ij}(t,x) + \d_i b^i(t,x) + c(t,x), \qquad v(0,x)=
  v_0(x)
\end{equation}
and use the following lemma, which is basically standard; but need to implement
some modifications to the usual versions. Namely, in order to prove
a regularization estimate,  we will also  work with  space-time
norms that deteriorate as $t\to0$: We define the space $C^\halpha_*(\RnT)$
to consist of those functions for which the norm
$$
\|u\|_{C^\halpha(\RnT)}^* := \|u\|_{C^\halpha_*(\RnT)} := \max\Bigl\{
\sup_{0<t\le T} t^{\halpha/2} \|u\|_{C^\halpha([t,T]\times\Rn)},
\|u\|_{L^\infty(\RnT)} \Bigr\}
$$
is finite. We prefer the first notation for typographical clarity even
when referring to functions in the larger space $C^\halpha_*(\RnT)$.
These norms reflect the short-time scaling $\|v(t)\|_{C^\halpha(\Rn)}\le
Ct^{-\halpha/2}\|v_0\|_{L^\infty(\Rn)}$, which is satisfied by the homogeneous
heat equation. The $L^\infty$ norm is included in the definition in order to
salvage
the algebra properties $\|uv\|_{C^\halpha(\RnT)}^* \le C
\|u\|_{C^\halpha(\RnT)}^* \|v\|_{C^\halpha(\RnT)}^*$ and also the stronger
\begin{equation}\label{strong-alg-prop*}
\|uv\|_{C^\halpha(\RnT)}^* \le C \left(
\|u\|_{C^\halpha(\RnT)}^* \|v\|_{L^\infty} +
\|u\|_{L^\infty} \|v\|_{C^\halpha(\RnT)}^* \right)
\;,
\end{equation}
which is a straightforward consequence of~\eqref{c2}.
The first
component alone
would not control the $L^\infty$ norm as the function $u(t,x)=\ln t$ shows.
\begin{lemma}[H\"older smoothing of rough sources by the heat flow]
  \label{constant}
There exists a unique solution to the heat equation \eqref{heateq}
in the space~$C^\halpha(\RnT)$.
It satisfies
\begin{equation}  \label{HE-est1}
\begin{array}{l} \Dst
\| v \|_{C^\halpha(\RnT)}
\le  C \left(
   \| v_0 \|_{C^\halpha(\Rn)}
 + \| f \|_{C^\halpha(\RnT)}  \phantom{T^{\frac12}}\right.
\\[1ex]\kern3.8cm \Dst \left. \mbox{}
 + T^\frac12 \| b \|_{C^\halpha(\RnT)}
 + T^{1-\frac12\halpha} \| c \|_{C^\halpha(\RnT)}
\right)
\;,
\end{array}
\end{equation}
and
\begin{equation} \label{HE-est2}
\begin{array}{l} \Dst
 \| v \|_{C^\halpha(\RnT)}
\le  C \left(
   \| v_0 \|_{C^\halpha(\Rn)}
 + \| f \|_{C^\halpha(\RnT)}  \phantom{T^{\frac12}}\right.
\\[1ex]\kern5mm \Dst \left. \mbox{}
 + T^{\frac{1-\halpha}2} \| b \|_{L^\infty(\RnT)}
 + T^{1-\frac12\halpha} \| c \|_{L^\infty(\RnT)}
 \right)
\,,
\end{array}
\end{equation}
and the regularization estimate
\begin{equation}  \label{HE-est3}
\begin{array}{l} \Dst
\| v \|_{C^\halpha(\RnT)}^*
\le  C \left(
   \| v_0 \|_{L^\infty(\Rn)}
 + \| f \|_{C^\halpha(\RnT)}^*  \phantom{T^{\frac12}}\right.
\\[1ex]\kern3.8cm \Dst \left. \mbox{}
 + T^\frac12 \| b \|_{C^\halpha(\RnT)}^*
 + T^{1-\frac12\halpha} \| c \|_{C^\halpha(\RnT)}^*
\right)
\;,
\end{array}
\end{equation}
where the constant $C$ may depend on an upper bound for~$T$.
\end{lemma}

Going through the proofs, we will also see:
\begin{cor}\label{DBCconstant}
  The very same estimates as in Lemma~\ref{constant} hold
  in the half space with Dirichlet (or Neumann) boundary conditions.
\end{cor}

\begin{proof}[Proof of Lemma~\ref{constant}]
By superposition, the estimates are assembled from four cases, in each of which
exactly one of the quantities $f,b,c,v_0$ is non-zero. The
estimate~\eqref{HE-est1}
for $v_0$ is classical, see for instance \cite{MR0241822}, IV, (2.2), with the
critical part being (2.13) there.
Likewise we can refer to (\cite{MR0241822},
IV, (2.1), and (2.8),(2.9) for the estimating $w^{ij}_t-\Laplace
w^{ij}=f^{ij}$, and obtain our case with $v=\d_{ij}^2w^{ij}$.

We also note that the arguments for $b^i$ and~$c$ in~\eqref{HE-est1}
can be given as (simpler) modifications of the estimates for the
$f^{ij}$, thus concluding the proof of~\eqref{HE-est1}. However, as we
will use the strengthened version~\eqref{HE-est2}, we will give these
estimates explicitly in a moment. At the same time, we will later
refer to~\eqref{HE-est1} in order to derive~\eqref{HE-est3} from it.
For the convenience of the reader, we will also redo the detailed
calculations pertaining to \eqref{HE-est1}--\eqref{HE-est3} in a
self-contained manner in an appendix, with no-pretense to novelty.

By~$\Gamma$, we denote the heat kernel:
$\Gamma(t,x):= (4\pi t)^{-n/2} \exp[-|x|^2/4t]$.

We estimate \underline{the contributions of $\|b\|_{L^\infty}$ to the alternate
  bound~\eqref{HE-est2}:}

In terms of $\|b\|_{L^\infty}$, we can estimate the spatial
H\"older quotient as
$$
\begin{array}{l} \Dst
\frac{|v(t,x')-v(t,x)|}{|x'-x|^\halpha}
\le
  \int_0^{t} \int_{\Rn}
\frac{|\nabla\Gamma(\tau,x'-y)-\nabla\Gamma(\tau,x-y)|}%
     {|x'-x|^\halpha} \|b\|_{L^\infty}
\,dy\,d\tau
\;.
\end{array}
$$
With $z:= (x-y)/\sqrt{4\tau}$ and
$h(z):= -2z \exp(-|z|^2) = \nabla\Gamma(\frac14,z)$,
this becomes
$$
\begin{array}{l} \Dst
\frac{|v(t,x')-v(t,x)|}{|x'-x|^\halpha}
\le c
  \int_0^{t} \tau^{-(1+\halpha)/2} \int_{\Rn}
\frac{|h(z+\zeta)-h(z)|}{|\zeta|^\halpha}
\,dz\,d\tau\, \|b\|_{L^\infty}
\;,
\end{array}
$$
where $\zeta=(x'-x)/\sqrt{4\tau}$.

Now for $|s|\le1$, we can estimate the inner integrand by the integrable
quantity $\sup_{B_1(z)} |Dh|$; for $|s|\ge1$, we can estimate the inner
integral by $2\int|h(z)|\,dz$, and combining the two cases, we get
$[v]_{x,\halpha}\le C T^{(1-\halpha)/2}\|b\|_{L^\infty(\RnT)}$.  The estimate
of $\|v\|_{L^\infty}$ is
similar, but simpler, providing a $T^{1/2}$ factor.

For the time H\"older quotient, we estimate
$$
\begin{array}{l} \Dst
|v(t',x)-v(t,x)| \le \int_t^{t'}\int_{\Rn}|\nabla\Gamma(t'-\tau,y)|\,
|b(\tau,x-y)|\,dy\,d\tau
\\[2ex]\kern4em\Dst \mbox{}
+
\int_0^t \int_{\Rn} |\nabla\Gamma(t'-\tau,y) - \nabla\Gamma(t-\tau,y)|\,
|b(\tau,x-y)|\,dy\,d\tau
\\[2ex]\Dst\mbox{}
\le C(\sqrt{t'}-\sqrt{t})\|b\|_{L^\infty} +
\|b\|_{L^\infty}
\int_0^t \int_{\Rn}\Bigl|\int_{t-\tau}^{t'-\tau} \nabla\Gamma_s(s,y)
\, ds\Bigr|\,dy\,d\tau
\,.
\end{array}
$$
We only have to estimate the last integral yet. Carrying out the $y$
integration first gives $O(s^{-3/2})$, then the $s$ integration yields
$$
C\int_0^t [ (t'-\tau)^{-1/2} - (t-\tau)^{-1/2} ]\, d\tau = C(t'-t)^{1/2}
\;.
$$
Hence $[v]_{t;\halpha/2} \le C T^{(1-\halpha)/2}\|b\|_{L^\infty}$.

We estimate \underline{the contributions of $\|c\|_{L^\infty}$ to the alternate
  bound~\eqref{HE-est2}:}

Finally, the $c$-term in~\eqref{HE-est2},  can be estimated in a similar way,
only with $\Gamma$ instead of $\nabla\Gamma$. For the spatial H\"older quotient
we obtain $O(T^{1-\frac\halpha2})$.
The $L^\infty$ estimate gives a factor~$T$ in a straightforward way.

For the time H\"older quotient we can easily estimate
$$
|v(t',x)-v(t,x)|/\|c\|_{L^\infty} \le
C\int_0^t[\ln(t'-\tau)-\ln(t-\tau)]\,d\tau
\,,
$$
which gives already the desired result, except for an extra logarithmic term;
and this result would be sufficient for our purposes. But for good measure,
let's prove the claimed, optimal, estimate:
$$
\begin{array}{l} \Dst
|v(t',x)-v(t,x)| \le \int_t^{t'}\int_{\Rn}\Gamma(t'-\tau,y)\,
|c(\tau,x-y)|\,dy\,d\tau
\\[2ex]\kern4em\Dst \mbox{}
+
\int_0^t \int_{\Rn} |\Gamma(t'-\tau,y) - \Gamma(t-\tau,y)|\,
|c(\tau,x-y)|\,dy\,d\tau
\\[2ex]\Dst\mbox{}
\le (t'-t+I)\|c\|_{L^\infty}
\end{array}
$$
with
\begin{equation}\label{Hkerndiffint}
I:= \int_0^t\int_\Rn |\Gamma(t'-\tau,y)-\Gamma(t-\tau,y)|\,dy\,d\tau
\;.
\end{equation}
Now for $t_1>t_0$, we have
$$
\Gamma(t_1,y) \glresp \Gamma(t_0,y)
\iff
|y|^2 \glresp r_*^2 := \frac{2n(\ln t_1-\ln t_0)}{t_1-t_0} t_0 t_1
\;.
$$
Hence
$$
\begin{array}{l} \Dst
\int_\Rn |\Gamma(t_1,y)-\Gamma(t_0,y)|\,dy
\\[2ex]\Dst\kern3em\mbox{}
=
\frac{|S^{n-1}|}{\pi^{n/2}} \,
2\int_0^{r_*} \left((4t_0)^{-n/2}e^{-r^2/4t_0} -
                    (4t_1)^{-n/2}e^{-r^2/4t_1}\right) \,
r^{n-1}\,dr
\\[2ex]\Dst\kern3em\mbox{}
=
\frac{|S^{n-1}|}{\pi^{n/2}} \,
\left(\gamma\bigl(\frac n2 \,\frac{\ln t_1 - \ln t_0}{t_1-t_0}\, t_1 \bigr)
    - \gamma\bigl(\frac n2 \,\frac{\ln t_1 - \ln t_0}{t_1-t_0}\, t_0 \bigr)
\right)
\end{array}
$$
where $\gamma(u):= \int_0^u e^{-s} s^{n/2-1}\,ds$ is the incomplete gamma
function, whose second argument $\frac n2$ we suppress, and of which we only
need that it is increasing with finite limit at $+\infty$, and with bounded
derivative.
With $t_1:= t'-\tau$, $t_0:= t-\tau$, $t_1-t_0=t'-t=:d$, and
$(t-\tau)/d=:\sigma$ as a new integration variable, we get
from~\eqref{Hkerndiffint}
$$
I= Cd \int_0^{t/d} \Bigl(
  \gamma\bigl({\Tst \frac n2 (s+1)\ln\frac{s+1}{s}}\bigr)
  -
  \gamma\bigl({\Tst \frac n2 s \ln\frac{s+1}{s}}\bigr) \Bigr) \,ds
\;.
$$
As $\sup|\gamma'|$ is finite, the integrand is bounded by $C
\ln\frac{s+1}{s}=O(\frac1s)$ as $s\to\infty$.
Therefore $I\le Cd(1+\ln(1+\frac td))\le Cd + Cd\ln(1+\frac Td)$.
From this, it follows
$$
[v]_{t;\halpha/2} \le \Bigl(C T^{1-\halpha/2} + C \sup_{d>0}
d^{1-\halpha/2}\ln(1+{\textstyle\frac Td}) \Bigr)
\|c\|_{L^\infty(\RnT)}
\;.
$$
The sup is taken on when $d=c_0T$ with $c_0$ determined by a transcendental
equation, and the value of the sup is therefore $C T^{1-\halpha/2}$, which
proves the estimate for the time H\"older quotient.

We have thus concluded the proof of \eqref{HE-est2} and now turn to
\eqref{HE-est3}.

For the homogeneous heat equation with initial data $v_0$, we have
first
\begin{equation}\label{vmax}
  |v(t,x)|\le \|v_0\|_{L^\infty}
\;.
\end{equation}
Then we have
\begin{equation}\label{vxalpha}
  |v(t,x')-v(t,x)| \le C \frac{|x'-x|}{t^{1/2}} \|v_0\|_{L^\infty}
\end{equation}
because of the estimate
$$
\begin{array}{l} \Dst
  v(t,x')-v(t,x) =
  \int_{\Rn}(\Gamma(t,x'-y)-\Gamma(t,x-y))\,v_0(y)\,dy
\\[1.5ex]\Dst \phantom{ v(t,x')-v(t,x) }
  = \int_{\Rn}\int_0^1
  (x'-x)\cdot\nabla\Gamma(t,x-y+s(x'-x))\, v_0(y)\,ds\,dy
  \,,
\\[2ex]\Dst
  |v(t,x')-v(t,x)| \le |x'-x|\,\|v_0\|_{L^\infty}
  \int_{\Rn}|\nabla\Gamma(t,z)|\,dz
\;.
\end{array}
$$
\eqref{vxalpha} and \eqref{vmax} combined provide
for $t\ge \tau$,
$$
\tau^{\halpha/2}\frac{|v(t,x')-v(t,x)|}{|x'-x|^\halpha}
\le \min\left\{ C\Bigl(\frac{|x'-x|}{\tau^{1/2}}\Bigr)^{1-\halpha} \,,\;
                2 \frac{\tau^{\halpha/2}}{|x'-x|^\halpha} \right\}
    \, \|v_0\|_{L^\infty}
\le
C'\|v_0\|_{L^\infty}
\;.
$$
Similarly we can prove, for $t'>t\ge\tau$, that
\begin{equation}\label{vtalpha}
    |v(t',x)-v(t,x)| \le C \frac{t'-t}{t} \|v_0\|_{L^\infty}
\end{equation}
and combine it with \eqref{vmax} to estimate the time H\"older
quotient in the same manner.

Taking the supremum over $\tau\in`]0,T]$, we conclude that
$\|v\|^*_{C^\halpha(\RnT)} \le \|v_0\|_{L^\infty}$.

Now let's look at the estimate of $(\d_t-\Laplace)v = \d_{ij}^2
f^{ij}$ with initial data~0.
We decompose $v=v_1+v_2$ where
$$
\begin{array}{ll}\Dst
(\d_t-\Laplace) v_1 = \chi_{[0,\tau/2]}(t)\, \d_{ij}^2 f^{ij}
 & v_1|_{t=0}=0
\,,
\\[1.5ex]
(\d_t-\Laplace) v_2 = \chi_{[\tau/2,T]}(t)\, \d_{ij}^2 f^{ij}
 & v_2|_{t=0} = 0 = v_2|_{t=\tau/2}
\;.
\end{array}
$$
For $v_2$, the classical Schauder estimate \eqref{HE-est1} applies on
the time interval $[\tau/2,T]$  and
gives
$$
\tau^{\halpha/2}\|v_2\|_{C^\halpha([\tau,T]\times\Rn)} \le
\tau^{\halpha/2}\|v_2\|_{C^\halpha([\tau/2,T]\times\Rn)} \le
C\tau^{\halpha/2} \|f\|_{C^\halpha([\tau/2,T]\times\Rn) }
\;.
$$
For $v_1$, we note that
$$
v_1({\Tst\frac34}\tau,x) = \int_0^{\tau/2} \int_{\Rn}
\d_{ij}^2\Gamma({\Tst\frac34}\tau-s,x-y) f^{ij}(s,y)\,dy\,ds
\;.
$$
Hence
$$
\begin{array}{l}\Dst
({\Tst\frac14}\tau)^{\halpha/2} \|v_1\|_{C^\halpha([\tau,T]\times\Rn)}
\le
C \|v_1({\Tst\frac34}\tau,\cdot)\|_{L^\infty(\Rn)}
\le
\int_0^{\tau/2}
({\Tst\frac34}\tau-s)^{-1}\|f(s,\cdot)\|_{L^\infty(\Rn)}\,ds

\\[1.5ex]\Dst \kern3em
\le
C \|f\|_{L^\infty([0,\tau/2]\times\Rn)}
\;.
\end{array}
$$
In conclusion
$$
\tau^{\halpha/2}\|v\|_{C^\halpha([\tau,T]\times\Rn)} \le
C \Bigl(\tau^{\halpha/2}\|f\|_{C^\halpha([\tau/2,T]\times\Rn)}
    + \|f\|_{L^\infty(\RnT)} \Bigr)
$$
If we can also estimate $\|v\|_{L^\infty}$  in like manner, taking the
supremum over $\tau$ establishes the $f$ part of~\eqref{HE-est3}
immediately. And indeed, we can estimate
$$
\begin{array}{l}\Dst
v(t,x) = \int_0^t\int_\Rn \d_{ij}^2\Gamma(t-s,x-y) f^{ij}(s,y)\,dy\,ds
\\[1.5ex]\Dst \phantom{v(t,x)}
= \int_0^t\int_\Rn \d_{ij}^2\Gamma(t-s,x-y)
                    (f^{ij}(s,y) - f^{ij}(s,x)) \,dy\,ds
\,,
\\[2ex]\Dst
|v(t,x)| \le \int_0^t \int_{\Rn}
    \frac{C}{(t-s)^{n/2+1}} e^{-|x-y|^2/5(t-s)} s^{-\halpha/2}
    |x-y|^\halpha \|f\|_{C^\halpha}^* \, dy\, ds
\\[1.5ex]\Dst \phantom{|v(tx)|}
\le C\int_0^t s^{-\halpha/2}(t-s)^{\halpha/2-1}\,ds\,
   \|f\|_{C^\halpha}^*
\le C \|f\|_{C^\halpha}^*
\;.
\end{array}
$$
This concludes the proof of the $f$ part of~\eqref{HE-est3}. The
impact of $b$ and $c$ is obtained in a similar manner.
\end{proof}

\begin{proof}[Proof of Cor.~\ref{DBCconstant}]
The proof for Dirichlet boundary conditions is essentially the same:
In~\eqref{HQx:f} and similar estimates, we
merely have to replace $\Gamma(t-\tau,x-y)$ with
$\Gamma(t-\tau,x-y)-\Gamma(t-\tau,x-y^*)$, where $y^*$ is the reflection of $y$
at the hyperplane bounding the half space; and of course we integrate
only over the
half space instead of $\Rn$. This half-space heat kernel obeys the same
Gaussian estimates, and no further changes are needed. In \eqref{HQx:b} and
similar equations, where we had conveniently put the $y$ into $\Gamma$ and the
$x-y$ into the right hand side term, we can again swap them and now get
expressions under the integral like $\bigl(\d_i\Gamma(t-\tau,x-y) -
\d_i\Gamma(t-\tau,x-y^*)\bigr) (b^i(\tau,y)-b^i(\tau,y+x'-x))$, supporting the
same estimates. The integrand within \eqref{Hkerndiffint} now becomes
$$
\begin{array}{l}
|\Gamma(t'-\tau,x-y)-\Gamma(t'-\tau,x-y^*)
-\Gamma(t-\tau,x-y)+\Gamma(t-\tau,x-y^*)|
\\[1ex]\kern1em \mbox{}
\le
|\Gamma(t'-\tau,x-y)-\Gamma(t-\tau,x-y)|
+|\Gamma(t'-\tau,x-y^*)-\Gamma(t-\tau,x-y^*)|
\end{array}
$$
and we may enlarge the integration over all of $\Rn$ and re-use the old
estimates.

The same applies to Neumann boundary conditions.
\end{proof}

\section{Quantitative global well-posedness of the linear and
  nonlinear equations in H\"older spaces}
\label{SecGlobWP}

Consider a smooth function on a uniform manifold $\UM$. Its derivative is a
section in the cotangent bundle and its second derivative is a section in a
suitable vector bundle. We consider differential equations
\begin{equation} \label{para}
   u_t = F(t,x,u, Du, D^2u)
\end{equation}
governing functions $u$ on~$\UM$.
The precise structure of $F$ is not of importance in this section.
We only need that there is such a formulation.
In local coordinates, it reduces to an equation of the same type,
but now with standard derivatives in $\Rn$.

We assume that the equation can be
written in local coordinates (and with the Einstein summation convention being
in effect) as
\begin{equation} \label{adjoint}
  u_t = \d_{ij}^2 ( f^{ij} (x,u)) +  \d_i (b^i(x,u))
  + c(x,u)
\end{equation}
with nonlinearities defined for $u$ in an open interval $U\subset\R$.

It is not difficult to check that this property is independent of our choice of
local coordinates and the partition of unity. We say that the equation is
uniformly parabolic if there exist positive constants $\lambda, \Lambda$ such
that
\begin{equation}
\lambda |\xi|^2 \le \d_u  f^{ij} (x,u) \xi_i \xi_j \le \Lambda |\xi|^2
\end{equation}
holds in such local coordinates as in Def.~\ref{Def-UM},
for $u$ in a compact subinterval $V\subset\subset U$, where $\lambda,\Lambda$
may depend on $V$. We call the equation
uniformly analytic  if in addition for any compact $V\subset U$,
all $k \ge 0$ and all multiindices~$\beta$
\begin{equation}
| \d^k_u \d^\beta f^{ij} (x,u)| + | \d^k_u \d^\beta  b^i(x,u)| + |\d^k_u
\d^\beta c(x,u)| \le C r^{k+|\beta|} (k+|\beta|)^{k+|\beta|}
\end{equation}
with $r$ depending on $V$. Again this property is independent of the chosen
points $x_j$ and $R$ -- up to changing $C$ and $r$.

To address~\eqref{adjoint}, we will also be interested in linear equations,
which can be written in local coordinates as
\begin{equation} \label{linear}
u_t - \d_{ij}^2 (a^{ij} u) - \d_i (b^i u) - c u = f + \d_i g^i
\;,\qquad u(x,0) = u_0(x)
\end{equation}
assuming that the coefficients $a^{ij}$ are uniformly elliptic and continuous.

Let us tackle the linear case first:

\begin{theorem}[Inhomogeneous linear parabolic flows on uniform manifolds]
\label{linearth}\mbox{} 
  Suppose that $u_0 \in C^\halpha(\UM)$, $0<\halpha<1$. Then
  there exists a unique weak solution $u \in C^\halpha([0,T]\times \UM)$ to
  \eqref{linear}, and it satisfies
  \begin{equation}\label{ball-estimate}
   \| u \|_{C^\halpha(\UMT)} \le C
    \left( \| u_0 \|_{C^\halpha(\UM)} + \| f \|_{L^\infty(\UMT)} +
    \| g \|_{L^\infty(\UMT)} \right).
  \end{equation}
  where the constant $C$ depends on~$T, \lambda, \Lambda, n ,
   \| a^{ij} \|_{C^\halpha}, \| b^i \|_{L^\infty},
   \| c \|_{L^\infty(\UMT)}$.

  Likewise, a unique weak solution to~\eqref{linear} exists in
  $C^\halpha_*(\UMT)$  for $u_0\in   L^\infty(\UM)$, and it satisfies the
  similar estimate \eqref{ball-estimate*} below.

  In the homogeneous case ($f=0$, $g=0$), the equation satisfies a comparison
  principle ($u \ge 0$ if $u_0\ge0$)   and,
  provided the $a^{ij},b^i,c$ are independent of time,
  defines an analytic semigroup on $C^\halpha(\UM)$
  in a sense made  precise in the following
  Remark~\ref{analytic-semigroup}.
  The equation also defines an analytic (in the sense of the
  following remark) semigroup on $\BC(\UM)$ (the space of bounded continuous 
  functions with the $L^\infty(\UM)$ norm)
  as well,  and it satisfies, for $0<t\le T$, the regularization estimate
  $\|u(t)\|_{C^\halpha(\UM)}  \le C t^{-\halpha/2} \|u_0\|_{L^\infty(\UM)}$,
  where $C$ may depend on~$T$.
\end{theorem}

\begin{remark}[Analytic Semigroups] \label{analytic-semigroup}
(a) By an analytic semigroup in $C^\halpha$, 
we mean a family $t\in [0,\infty`[ \longrightarrow S(t)$
of linear operators on $C^\halpha$ 
satisfying $S(t_1+t_2) = S(t_1)S(t_2)$ for all $t_1,t_2 \ge 0$,
which extends to a complex sector and is holomorphic in the interior
of that sector, satisfies a  bound  $\|S(t)\|\le C e^{C \re t}$,
and for which $(t,x) \mapsto S(t)u_0(x)$
is in $C^\halpha$ with respect to $t$ and $x$, including at $t=0$.
(As $t \to 0$, this implies $S(t)u_0 \to u_0$ uniformly, but not necessarily
in the $C^\halpha$ norm as would be required for the standard
definition of analytic semigroup, which includes strong continuity at $t=0$.)

The analogous definition applies for an analytic semigroup in~$\BC$;
the continuity requirement in this case is that the mapping $(t,x)\mapsto
S(t)u_0(x)$ is in $\BC$ with respect to both variables, including $t=0$.
This still implies $S(t)u_0\to u_0$ pointwise (but not necessarily
uniformly) as $t\to0+$.

(b) Indeed, strong continuity at $t=0$ cannot hold in the space $C^\halpha$
because, for $u_0$ {\em not\/} in the closure of $C^\infty$ within $C^\halpha$,
smoother functions like $S(t)u_0$ cannot converge to $u_0$ in the
$C^\halpha$-norm.  Our estimates do imply that
$\|S(t)u_0-u_0\|_{L^\infty}\to0$ as $t\to0$. Moreover, if $u_0\in C^\hbeta$
with $\hbeta>\halpha$, then $\|S(t)u_0-u_0\|_{C^\halpha}\to0$ as~$t\to0$,
because by two independent easy estimates,
$$
\frac{|u(t,x')-u_0(x') - u(t,x)+u_0(x)|}{|x'-x|^\halpha} \le \min\left\{
\begin{array}{c}
  ([u(t)]_{x;\hbeta} +[u_0]_{x;\hbeta}) |x'-x|^{\hbeta-\halpha} \,,
\\[1ex]
   2\|u(t)-u_0\|_{L^\infty}/|x'-x|^\halpha
\end{array}
\right\}
$$
and this minimum tends to~0 uniformly in space as $t\to0$.

(c) Alternatively we could work in the `little H\"older spaces'~$\oC^\halpha$,
the closure of $C^\infty$ in $C^\halpha$ instead of $C^\halpha$ itself. The
estimates are the same, and the changes in the proofs are marginal.
In $\oC^\halpha$, the strong continuity of the semigroup is restored, because
$u_0\in\oC^\halpha$ can be approximated by $\phi\in C^\hbeta$ for arbitrary
$\hbeta\in`]\halpha,1`[$, and using the standard estimate
$$
\|S(t)u_0-u_0\|_{C^\halpha} \le
\|S(t)\phi-\phi\|_{C^\halpha} + \|S(t)(u_0-\phi)\|_{C^\halpha}  +
\|u_0-\phi\|_{C^\halpha}
\;.
$$

(d) These observations ensure that solutions are continuous up to $t=0$ 
and the initial condition is understood in the obvious sense, albeit 
(for the classical H\"older spaces) not in the functional analytic sense
of norm continuity. Basically, these
subtleties at $t=0$ are of no concern to us as we are interested in
asymptotics for $t\to\infty$; on a technical level, this is borne out
by taking iterative discrete time steps, and the finite time flow map
has all the continuity, and even smoothness, we want, as elaborated
upon at the end of the proof of~Thm.~\ref{linearth}.
\end{remark}

\begin{proof}[Proof of Thm.~\ref{linearth}]
We turn to Equation~\eqref{linear} on the uniform manifold, i.e.,
$$
u_t - \d_{ij}^2 (a^{ij} u) - \d_i (b^i u) - c u = f +\d_i g^i
\;,\qquad u(x,0) = u_0(x)
\;,
$$
where the data $a^{ij},b^i,c,f,u_0$ are H\"older continuous,
and introduce  a
small parameter $\gamma$ to be chosen later. By the
H\"older continuity of
the inital data $u_0$ there exists $\delta$ such that
$|a^{ij}(x)-a^{ij}(y)|\le\gamma$
if $d(x,y) < \delta$.
We may decrease $\delta$ such that $3\delta$ is still a radius of uniformity.

We use a partition of unity $\sum\vareta_l^2$ subordinate to a locally finite
covering with balls of diameter $\delta$, with local coordinate maps $\chi_l$
from these balls into $\Rn$.
As the $a^{ij}$ are almost constant on these balls, we can, after an affine
change of coordinates, achieve that $a^{ij}-\delta^{ij}$ is $L^\infty$-small on
these balls.  This introduces constants that can be bounded in terms of the
ellipticity parameters $\lambda,\Lambda$.
We decompose~$u$ by
$$
\Tst
v_l := (\vareta_l u)\circ\chi_l^{-1}, \qquad
u = \sum_l \vareta_l (v_l\circ\chi_l)   \;.
$$
In order to avoid clogging up the calculation with coordinate maps we define
$\eta_l:=\vareta_l \circ\chi_l^{-1}$ and $u_l:= u\circ\chi_l^{-1}$ for the
pullbacks of the partition of unity and of the function $u$ to the respective
coordinate patches in~$\Rn$. This way, in the equations below, the coefficient
functions as well as the quantities indexed with $l$ live on $\RnT$ (or, in the
case of $u_l$, a subset thereof), whereas the function $u$ lives on $\UMT$.

Then $v_l$ satisfies the initial condition $v_l(0)=\eta_l u_{0,l}$ and
\begin{equation}\label{BFP-localization}
\begin{array}{r@{}l}  \Dst
\d_t v_l - \Laplace v_l  &\Dst\mbox{} = \eta_l \d_t u_l -\Laplace(\eta_l u_l)
\\[1ex]&\Dst\mbox{} =
\d_{ij}^2 \bigl[ (a^{ij}- \delta^{ij})  \eta_l  u_l   \bigr] +
\d_i\bigl[ b^i\eta_l u_l - 2 a^{ij}(\d_j\eta_l) u_l + g^i\eta_l \bigr]
\\[1ex]&\Dst\kern1.2em\mbox{}
+ c\eta_l u_l + f \eta_l - g^i (\d_i\eta_l)
+ \bigl[a^{ij} \d_{ij}^2 \eta_l - b^i (\d_i\eta_l) \bigr] u_l
\;.
\end{array}
\end{equation}
On the other hand, given $u$, we can use this equation to define the $v_l$, and
hence a mapping $\mapT: u\mapsto \sum(\eta_l v_l)\circ \chi_l$, for which we
want to find a fixed point by means of the  Banach fixed point theorem (i.e.,
contraction mapping principle)
and Lemma \ref{constant}. On the support of $\eta_l$, $a^{ij}-\delta^{ij}$ is
small in the $L^\infty$ norm, and otherwise all coefficients are bounded in
$C^\halpha$. The $\eta_l$ are bounded in the $L^\infty$ norm, but their
$C^\halpha$ norm will be large for small~$\delta$.
Given $u$,  \eqref{HE-est1} yields, for each~$l$,
\begin{equation}\label{pre-contraction}
\begin{array}{l}
\|v_l\|_{C^\halpha(\RnT)}
\le C \Bigl(
\| \eta_l u_{0,l} \|_{C^\halpha} +
\gamma \|\eta_lu_l\|_{C^\halpha}
+\|a^{ij}-\delta^{ij}\|_{C^\halpha}\|\eta_lu_l\|_{L^\infty}
\\[1.5ex] \kern7em \mbox{}
+T^{1/2} \|\eta_l\|_{C^{2,\halpha}}
    \|(a^{ij},b^i,c)\|_{C^\halpha} \|u_l\|_{C^\halpha}
\\[1.ex] \kern7em \mbox{}
+ T^{1-\frac{\halpha}2}\|\eta_l\|_{C^\halpha}\|f\|_{C^\halpha}
+ T^{\frac12}\|\eta_l\|_{C^{1,\halpha}}\|g^i\|_{C^\halpha}
\Bigr)
\,,
\end{array}
\end{equation}
where we have used the fine algebra estimate \eqref{c2} on the term with the
highest derivatives, but the simpler estimate $\|uw\|_{C^\halpha}\le
C\|u\|_{C^\halpha}\,\|w\|_{C^\halpha}$ suffices for the other terms
since their coefficients can be made small by choosing $T$ small.
Since $u_l$ is the restriction of $\sum_k \eta_k^2 u_k^{}$ to the $l^{th}$
coordinate patch, and at most a fixed number $N$ of terms (given by the ball
packing Lemma~\ref{Lem-ballpacking}) contributes to this sum for each~$l$, we can
estimate $\|u_l\|_{C^\halpha}\le NC_\delta\|u\|_{C^\halpha(\UMT)}$.

We want this estimate \eqref{pre-contraction} for later reference, but we can
strengthen it immediately: namely, a similar estimate based on~\eqref{HE-est2}
is possible with  only  the $L^\infty$ norms of
$b,c$ instead of the H\"older norm; and also only the $L^\infty$ norms
of~$f,g$, and correspondingly smaller exponents for $T$.

Similarly, given two sources $u$ and $\bar u$ for~\eqref{BFP-localization},
with the same initial trace $u_0$, we get for the difference $v_l-\bar v_l$ of
the solutions
the same estimate with $u$ replaced by $u - \bar u$ on the right hand side
and $u_0=0$, $f=0=g$.  Thus
\begin{equation}\label{aux-contraction}
\begin{array}{r@{}l}
\|\eta_l (v_l-\bar v_l)\|_{C^\halpha(\RnT)}
&\mbox{}\le
\|v_l-\bar v_l\|_{C^\halpha(\RnT)} +
\|\eta_l\|_{C^\halpha(\Rn)} \|v_l-\bar v_l\|_{L^\infty(\RnT)}
\\[1ex]&\mbox{}\le
(1+T^{\halpha/2}\|\eta_l\|_{C^\halpha(\Rn)})
  \|v_l-\bar v_l\|_{C^\halpha(\RnT)}
\end{array}
\end{equation}
and similar with products $\eta_l\eta_k$ of cutoff functions.
Recall that $\|\mapT u\|_{C^\halpha(\UMT)} := \sup_k
\|\sum_l\eta_lv_l\eta_k\|_{C^\halpha(\RnT)}$. Here, for each $k$, the sum
consists of at most $N$ terms, where $N$ describes  the {\em uniform\/} local
finiteness of the partition $\{\vareta_l^2\}$ as given by
Lemma~\ref{Lem-ballpacking}.

Therefore, in order to show that the map $u \mapsto \mapT u:= \sum (\eta_l
v_l)\circ\chi_l$ is a contraction, with contraction constant $\vartheta<1$,  we
need to prove a similar contraction estimate (with smaller contraction
constant~$\vartheta/N$) for the maps $u\mapsto v_l$.
Choosing  $\gamma=1/(2CN)$ we determine the necessary $\delta$ and obtain
a (large but fixed) bound for $\sup_l\|\eta_l\|_{C^{k,\halpha}}$.
All terms in \eqref{pre-contraction} that contain potentially large norms of
$\eta_l$ are multiplied by a power of $T$, which we choose small  to compensate
for the norm of the cutoff ($u_0$ doesn't occur in the difference, and we again
note that
$\|\bar u-u\|_{L^\infty(\UMT)} \le T^{\halpha/2}\|\bar u-u\|_{C^\halpha(\UMT)}$
because $(\bar u - u)|_{t=0}=0$).
So we conclude, for $T$ sufficiently small
(independent of~$u_0$), that $\mapT$ is
a contraction (globally on all of~$C^\halpha(\UMT)$).
An iterative procedure gives existence and uniqueness.

An estimate
\begin{equation} \label{ball-estimate2}
   \| u \|_{C^\halpha(\UMT)} \le C
    \left( \| u_0 \|_{C^\halpha(\UM)} + \| f \|_{C^\halpha(\UMT)} +
    \| g \|_{C^\halpha(\UMT)} \right)
\end{equation}
similar to \eqref{ball-estimate} in
Theorem \ref{linearth} holds because
Equation~\eqref{pre-contraction} implies
\begin{equation}\label{Hoelder-ball-est}
\begin{array}{l}
\|\mapT u\|_{C^\halpha(\UMT)}
\le
\frac12 \|u\|_{C^\halpha(\UMT)}
 + CN  T^{1-\frac\halpha2} C_\delta\|f\|_{C^\halpha}
 + CN  T^{\frac12} C_\delta\|g\|_{C^\halpha} +
\\[1ex]\kern6em\mbox{}
 + CN \bigl( C_\delta \|u_0\|_{C^\halpha} + C_\delta\|u\|_{L^\infty} +
             T^{1/2}C_\delta \|u\|_{C^\halpha} \bigr)
\,,
\end{array}
\end{equation}
with constants now depending on norms of $a^{ij},b^i,c$. We again estimate
$\|u\|_{L^\infty} \le \|u_0\|_{L^\infty} + T^{\halpha/2}
\|u-u_0\|_{C^\halpha}\le (1+T^{\halpha/2})\|u_0\|_{C^\halpha} +
T^{\halpha/2}\|u\|_{C^\halpha}$ and solve for $\|\mapT
u\|_{C^\halpha}=\|u\|_{C^\halpha}$.
Likewise, by using the strengthened version of~\eqref{pre-contraction}, based
on~\eqref{HE-est2}, we get \eqref{ball-estimate}.

Next, we prove the regularization estimate.
Using~\eqref{HE-est3} instead of~\eqref{HE-est1}, we can construct a local
solution by Banach's fixed point theorem as before, with the analog
of~\eqref{pre-contraction} involving
$\|\eta_lu_0\|_{L^\infty(\UM)}\le\|u_0\|_{L^\infty(\UM)}$  instead of
$\|\eta_lu_0\|_{C^\halpha(\UM)}$, and otherwise $\|\cdot\|_{C^\halpha(\UMT)}^*$
instead of $\|\cdot\|_{C^\halpha(\UMT)}$.
A slight modification of~\eqref{aux-contraction} is needed: we use that
$\|\eta_l\|_{C^\halpha(\RnT)}^*\le
\max\{1, T^{\halpha/2}\|\eta_l\|_{C^\halpha(\Rn)}\}$.
Similarly, in the analog of~\eqref{pre-contraction}, the term
$\|a^{ij}-\delta^{ij}\|_{C^\halpha(\RnT)}^*\|\eta_lu_l\|_{L^\infty}$
can be estimated by
$\max\{\gamma, T^{\halpha/2}\|a^{ij}-\delta^{ij}\|_{C^\halpha(\Rn)}\}\,
\|\eta_lu_l\|_{C^\halpha(\RnT)}^*$.

This argument proves
the short term existence and uniqueness for $L^\infty$ initial
data; it also provides, as in
\eqref{Hoelder-ball-est}, an analog to \eqref{ball-estimate}:
\begin{equation}\label{ball-estimate*}
\| u \|_{C^\halpha(\UMT)}^* \le C
    \left( \| u_0 \|_{L^\infty(\UM)} +
    \| f \|_{C^\halpha(\UMT)}^* +     \| g \|_{C^\halpha(\UMT)}^*
\right)
\,,
\end{equation}
from which
$\| u(t) \|_{C^\halpha(\UM)} \le C t^{-\halpha/2} \| u_0 \|_{L^\infty(\UM)}$
follows immediately when $f=0$, $g=0$.

Let us briefly note that these same estimates can also be used for
differences $u_n- u_m$ of solutions with, say, the same
initial data $u_0$ and with $f=0=g$, but different coeeficients
$a^{ij}_n, b^i_n, c_n$  and
$a^{ij}_m, b^i_m, c_m$  respectively. They will
then provide in analogy to \eqref{pre-contraction} that
$$
\begin{array}{l}
\|v_{ln}-v_{lm}\|_{C^\halpha(\RnT)} \le \mbox{}
\\[1.5ex] \kern4em \mbox{}
\le C \Bigl(
\gamma \|\eta_l(u_{ln}-u_{lm})\|_{C^\halpha}
+\|a^{ij}_n-\delta^{ij}\|_{C^\halpha}\|\eta_l(u_{ln}-u_{lm})\|_{L^\infty}
\\[1.5ex] \kern7em \mbox{}
+T^{1/2} \|\eta_l\|_{C^{2,\halpha}}
    \|(a^{ij}_n,b^i_n,c_n)\|_{C^\halpha} \|u_{ln}-u_{lm}\|_{C^\halpha}
\\[1.5ex] \kern7em \mbox{}
+ 2\gamma \|\eta_l u_{lm}\|_{C^\halpha}
+\|a^{ij}_n-a^{ij}_m\|_{C^\halpha}\|\eta_l u_{lm}\|_{L^\infty}
\\[1.5ex] \kern7em \mbox{}
+T^{1/2} \|\eta_l\|_{C^{2,\halpha}}
    \|(a^{ij}_n-a^{ij}_m,b^i_n-b^i_m,c_n-c_m)\|_{C^\halpha}
    \|u_{lm}\|_{C^\halpha}
\Bigr)
\;.
\end{array}
$$
Then, our estimate, applied to the respective fixed points $u_n$ and $u_m$,
provides continuous dependence of the solution on the coefficients:
\begin{equation}\label{Lip-dependence}
\|u_n-u_m\|_{C^\halpha(\RnT)} \le C \|u_m\|_{C^\halpha(\RnT)}
\,
    \|(a^{ij}_n-a^{ij}_m,b^i_n-b^i_m,c_n-c_m)\|_{C^\halpha}
\;.
\end{equation}

The comparison statement (nonnegative initial data give nonnegative
solutions) would follow by commuting multiplication and
differentiation and classical maximum principles if the coefficients
were twice continuously differentiable, and if the manifold were
compact. The first obstacle can easily be dealt with by
regularization: Approximate the $a^{ij},b^i$ uniformly by smooth
$\tilde{a}^{ij}_n, \tilde{b}^i_n$
with bounded $C^\halpha$ norms, and such that the
$\tilde{a}^{ij}_n, \tilde{b}^i_n$ converge to $a^{ij},b^i$ in a
(weaker) $C^\beta$ norm, with $0<\beta<\halpha$. The corresponding solutions
$\tilde u_n$ then converge to $u$ in the $C^\beta$ norm by the continuous
dependence estimate \eqref{Lip-dependence}.

So we now write, for smooth coefficients in local coordinates
$$
u_t -  a^{ij} \d_{ij}^2 u - \tilde b^i \d_i u - \tilde c u = 0
\;.
$$
Replacing $u$ by $e^{-\lambda t} u  $ if necessary, we may assume
that $\tilde c \le 0$.  Let $\rho$ be the radius function constructed in
Lemma~\ref{growth}. Then  $- \eps (\mu t + \rho )$
is a nonpositive subsolution, provided
$\mu$ is sufficiently big. Hence, if $u_0$ is nonnegative and $f=0$, $g=0$,
then $u + \eps (\mu t +\rho) $ is a supersolution which tends to
$\infty$  as $ x \to \infty$ and which is nonnegative at $t=0$.  By
the maximum principle it does not assume its infimum for $0 < t \le
T$ and hence it is nonnegative. We let $\eps \to 0$ to arrive at the
desired conclusion.

The construction for short time~$T$ can be iterated to get global existence in
time, and hence a semigroup.
The fact that we have an {\em analytic\/} semigroup follows from a 
standard argument: We introduce 
a complex parameter $\tau$ and study, in a space of complex functions,
\begin{equation} \label{lineartau}
u_t = \tau \Bigl(    \d_{ij}^2 (a^{ij} u) + \d_i
  (b^i u) + c u \Bigr)  \;, \qquad u(x,0) = u_0(x)
\;.
\end{equation}
The previous arguments work for all $\tau>0$. We rewrite the equation as
$$
u_t -    \tau_0\bigl(\d_{ij}^2 (a^{ij} u)  - \d_i  (b^i u) - c u   \bigr)
=
(\tau-\tau_0) \bigl(   \d_{ij}^2 (a^{ij} u) - \d_i  (b^i u) - c u
\bigr)
\,.
$$
The right hand side is analytic in $\tau$.
We obtain an analytic dependence on $\tau$, which is equivalent to having 
an analytic semigroup -- as here we don't require norm continuity at time $t=0$.
Specifically, the same estimates we have made for $\tau=1$ apply to the present
case, for $\tau\in\C$, as long as $|\tau-\tau_0|/\tau_0$ is sufficiently small
(with the right hand side not going into~$f$, but distributing over all the
other terms in~\eqref{pre-contraction}).  So our estimates extend to a small
sector in the complex time-plane.

Our estimates using the space $C^\halpha_*$ for $L^\infty$ initial data 
will swiftly imply 
that  the solution of \eqref{linear} defines an analytic semigroup
on~$\BC(\UM)$, again in the sense of~Remark~\ref{analytic-semigroup}. 
(We make the statement for $\BC$ rather than $L^\infty$ to avoid 
technicalities in which sense initial data are taken on, technicalities that 
would not pertain to the focus of our problem.)

To see this, let us first assume that our initial data $u_0$ are 
{\em uniformly\/} continuous. Then they can be approximated uniformly by 
$C^\halpha$ initial data $u_0^{[k]}$. Using the 
estimate~\eqref{ball-estimate*}, we infer that the corresponding 
solutions~$u^{[k]}$ converge in~$C^\halpha_*$ and hence uniformly to 
a (the) solution $u$ for initial dat $u_0$. But the $u^{[k]}$ themselves were
continuous (actually they are in the unmodified $C^\halpha(\UMT)$), so their 
limit is (uniformly) continuous. 

Non-uniform continuity of $u_0$ can be handled by noting that for any 
continuous weight function $w$ vanishing at infinity, $wu_0$ will be 
{\em uniformly\/} continuous if $u_0\in\BC(\UM)$ is continuous.
Such a weight function can be constructed in a smooth manner based 
on Lemma~\ref{growth}, and then rewriting~\eqref{linear} for the corresponding 
$wu_0$ results in an equation of the same type, so that the result with 
uniformly continuous data for the conjugated equation gives the continuity  
result for the original equation with (not necessarily uniformly) continuous
data~$u_0$. The boundedness is of course maintained from the argument with the 
unconjugated equation. (In the specific situation $\UM=\M$ that interests us 
for the fast diffusion equation, we can take $w=(\cosh s)^{-\eta}$.) 
\end{proof}

Basically the same proof yields
\begin{cor}[Variant for Dirichlet boundary data]\label{linearthDBC}
For equation~\eqref{linear} on a smoothly bounded domain on~$\UM$, with
Dirichlet boundary conditions $u=0$ and $C^\halpha$ initial data, there exists
a unique solution in~$C^\halpha$, and estimate~\eqref{ball-estimate} continues
to apply, as well as all the other conclusions from Thm.~\ref{linearth}.
\end{cor}

Now we can tackle the {\em short-term\/} estimates for the nonlinear equation
by means of a slight modification of the proof of the linear theorem:

\begin{theorem}[Short time smooth dependence on data] \label{existence}
  Suppose equation \eqref{para} can be written in local
  coordinates as \eqref{adjoint}, namely
  $$
    u_t = \d_{ij}^2 ( f^{ij} (x,u)) +  \d_i (b^i(x,u))  + c(x,u)
  $$
  with nonlinearities defined for $u$ in an open interval $U\subset\R$,
  and that this equation  is uniformly both parabolic
  and analytic.  Let $0< \halpha <1$, 
  and choose two subintervals
  $V \subset \subset  W \subset \subset U$.
  There exists $T>0$ depending  on $V$ and $W$, $\|u_0\|_{C^\halpha(\UM)}$, and
  the nonlinearities of the operator,
  such that for $u_0 \in C^\halpha(\UM)$ with $u_0(\UM) \subset V$ there exists
  a unique weak solution $u \in C^\halpha(\UMT)$
  to \eqref{adjoint}. Its values are in $W$.
  For every $u_0\in C^\halpha(\UM)$ with range in $V$, there exists an 
  $R$-ball in $C^\halpha(\UM)$ and a $T$ such that the map
  $$
  C^\halpha(\UM)\supset B_R(u_0) \ni u_0 \longmapsto u \in C^\halpha(\UMT)
  $$
  is analytic.
\end{theorem}
As stated, as of the present theorem, the existence time might depend on the
H\"older norm of the initial data. However, we will see soon that this is
not the case and that the existence time is influenced by $u_0$ only
through its supremum norm.

\begin{proof}
The first step for the nonlinear problem, namely Equation~\eqref{adjoint},
$$
  u_t = \d_{ij}^2 ( f^{ij} (x,u)) +  \d_i (b^i(x,u))
  + c(x,u)
$$
requires few changes.  We rewrite
the equation again as fixed point problem, this time in an $L^\infty$ ball
about $u_0$, within $C^\halpha$. We choose $\eps_0$ such that a $2\eps_0$
neighbourhood of $V$ still lies in~$W$.
We search a small solution in the form
$u=u_0 + \tilde u$, where $\tilde u$ satisfies an equation of the same type;
$\|\tilde{u}\|_{L^\infty}$ will be assumed to be $\le 2\eps_0$.

Just as in the linear case, the equation for $\tilde u$ can be localized
using the partition of unity  with
$$\Tst
v_l = (\vareta_l  \tilde u)\circ\chi_l^{-1}\;,
\qquad
\tilde u = \sum_l \eta_l (v_l\circ \chi_l)\;.
$$
Again $\tilde u_l:= u\circ\chi_l^{-1}$ and
$\eta_l:=\vareta_l\circ\chi_l^{-1}$.
One easily computes in local coordinates
\begin{equation} \label{local}
   \d_t v_{l} - \Laplace v_l =
   \d_{ij}^2 \tilde f^{ij}_l + \d_i \tilde b^i_l + \tilde c_l
\end{equation}
(with initial data 0) where, by a slight abuse of notation,
$$
\begin{array}{r@{\mbox{}=\mbox{}}l}
\tilde f^{ij}_l  &
    \eta_l \,\bigl(
           f^{ij}(x,u_0 + \tilde u) -
           \delta^{ij} \tilde{u} \bigr)
\,,
\\[1ex]
\tilde b^i_l   &
    \eta_l b^i(x,u_0+\tilde u)
    - 2 (\d_j\eta_l) \, f^{ij}(x,u_0 + \tilde u)
\,,
\\[1ex]
\tilde c_l   &
    \eta_l c(x,u_0+\tilde u)  - (\d_i \eta_l)\, b^i(x,u_0+\tilde u )
    + (\d_{ij}^2 \eta_l)\,f^{ij}(x,u_0+\tilde u)
\;.
\end{array}
$$

Using the considerations for the heat equation \eqref{heateq},
we again aim to construct the solution up to short time~$T$,
by means of the Banach fixed point
theorem applied in a closed subset
$S:=\{\tilde u\mid \|\tilde u\|_{L^\infty}\le2\eps_0\,,\;\|\tilde
u\|_{C^\halpha}\le M\}$
of~$C^\halpha(\UMT)$. Here $M$ is a possibly large constant that will be chosen
shortly and will depend on $\|u_0\|_{C^\halpha(\UM)}$.
The mapping is again
$\mapT: \tilde u \mapsto \sum_l(\eta_lv_l)\circ\chi_l^{-1}$. We first set a
target contraction constant $\vartheta\in`]\frac12,1`[$. We also commit 
a~priori to an upper bound $T_0$ for $T$ and calculate
$\|\mapT 0\|_{C^\halpha(\UMT)}\le
\|\mapT 0\|_{C^\halpha(\UMT{}_{{}_0})}=:A$. Then the choice
$M:= A/(1-\vartheta)$ will turn out to be expedient.

The estimates for the heat equation from Lemma~\ref{constant} imply:
\begin{equation}\label{heatest-a}
\|v_l\|_{C^\halpha(\RnT)} \le C \left(
    \|\tilde{f}_l\|_{C^\halpha}
   + T^{1/2} \|\tilde{b}_l\|_{C^\halpha}
   +T^{1-\halpha/2} \|\tilde{c}_l\|_{C^\halpha}
\right)
\end{equation}
and for differences of solutions
\begin{equation}\label{heatest-b}
\|v_l-\bar v_l\|_{C^\halpha} \le C \left(
   \|\tilde{f}_l-\bar{\tilde{f}}_l\|_{C^\halpha}
   + T^{1/2} \|\tilde{b}_l-\bar{\tilde{b}}_l\|_{C^\halpha}
   +T^{1-\halpha/2} \|\tilde{c}_l-\bar{\tilde{c}}_l\|_{C^\halpha}
\right)
\end{equation}
with
\begin{equation}\label{heatest-c}
\tilde{f}_l-\bar{\tilde{f}}_l
=
\eta_l \int_0^1 \left[
  \d_u f^{ij}(x,u_0+\bar{\tilde u} + s(\tilde u - \bar{\tilde u}))
  -\delta^{ij} \right]\,ds \, (\tilde u - \bar{\tilde u})
\;.
\end{equation}
The constants $C$ in \eqref{heatest-a}, \eqref{heatest-b} are uniform as long
as $u_0$, $u_0+\tilde u$, $u_0+\bar{\tilde u}$ have range within~$W$. All norms
in \eqref{heatest-a}, \eqref{heatest-b} are $C^\halpha(\RnT)$ norms.

Smallness of the oscillation of $\d_uf^{ij}$ over all pertinent arguments
$(x,u)\in\UM\times W$ is crucial for the contraction estimate. Given $\gamma>0$, there exist
$\delta,\eps_1$ such that
$$
|\d_u f^{ij} (x,u) - \d_u f^{ij} (y,\bar u)| \le \gamma
$$
if $d(x,y) < \delta$ and $|u-\bar u|\le\eps_1$.
By an affine change of local coordinates in a ball, we can actually assume
$\delta^{ij}$ to be the value of $\d_u f^{ij}$ at some $(x,u)$. This will make
our constants dependent on the ellipticity parameters $\lambda,\Lambda$, but
does otherwise not affect the estimates. We choose
$\gamma\le\frac{1}{2CN}$ with~$N$
bounding the number of $\delta$-balls that can be packed in a $3\delta$-ball,
a quantity independent of~$\delta$ by Lemma~\ref{Lem-ballpacking}.
Then, \eqref{heatest-c}~implies, with $K$ a bound for
$\|\d_uf^{ij}\|_{L^\infty(\UM\times W)}$ and
$\|\d_u^2f^{ij}\|_{L^\infty(\UM\times W)}$, that
\begin{equation} \label{fdifference}
\begin{array}{l}
\|\tilde{f}_l-\bar{\tilde{f}}_l\|_{C^\halpha(\RnT)}
\le
\gamma\|\eta_l(\tilde u_l-\bar{\tilde u}_l)\|_{C^\halpha(\RnT)} +
\\[1ex]\kern1.5em +
\max_{s\in[0,1]}
\|\d_uf^{ij}(x,u_0+s\tilde u+(1-s)\bar{\tilde u})-\delta^{ij}\|_{C^\halpha}
\|\eta_l(\tilde u_l-\bar{\tilde u}_l)\|_{L^\infty}
\\[1ex]
\kern1em
\le \Bigl(\gamma + K
 \|u_0+s\tilde u+(1-s)\bar{\tilde u}\|_{C^\halpha}+1)
C T^{\halpha/2}\Bigr) \|\tilde u -\bar{\tilde u}\|_{C^\halpha(\UM)}
\,,
\end{array}
\end{equation}
provided the partition of unity is made with balls of diameter%
$\mbox{}<\delta$. With $\delta$ thus fixed, we have bounds for the $C^\halpha$
norms of $\eta_l$ and its derivatives and can use the smallness of~$T$,
dependent also on $K(\|u_0\|_{C^\halpha(\UM)}+M)$, to
ensure from \eqref{heatest-b} and~\eqref{fdifference} that
$\mapT: \tilde{u}\mapsto \sum(\eta_l v_l)\circ\chi_l^{-1}$ is a contraction
from  $S\cap C^\halpha(\UMT)$ to $C^\halpha(\UMT)$ with contraction
constant~$\vartheta$.

From~\eqref{heatest-a}, we get a fixed (possibly large) bound
for~$\mapT \tilde u:= \|\sum(\eta_l v_l)\circ\chi_l^{-1}\|_{C^\halpha}$, and
again the smallness of~$T$ ensures
that $\|\mapT \tilde u\|_{L^\infty}\le 2\eps_0$.
We also estimate
$$
\|\mapT\tilde u\|_{C^\halpha(\UMT)} \le
\|\mapT\tilde u-\mapT 0\|_{C^\halpha(\UMT)} +
\|\mapT 0\|_{C^\halpha(\UMT)}
\le \vartheta M + A = M
\;.
$$
This guarantees the applicability of
Banach's fixed point theorem, and hence warrants local existence.

The same solution can also be obtained by means of the implicit function
theorem, because Thm.~\ref{linearth} guarantees a bounded inverse for the
linearization operator. The analytic dependence of the equation on the data
then leads to the analytic dependence of the solution on the initial data and parameters.
\end{proof}

We will want to control the deviation of the nonlinear flow from the linear
flow, at least for short times. Contingent upon second order  differentiability
of the flow in the appropriate function space~$C^\halpha$, this deviation
should be of quadratic order in the norm of the space.
But we will need an estimate that involves the weaker
$L^\infty$-norm.  So we claim:
\begin{lemma}[Linear approximation of nonlinear
    semiflow]\label{quadratic-error}
\hskip 0pt plus 8pt
Let $\bar w$ solve the homogeneous linear equation~\eqref{linear}, namely
$(\d_t-\Lop)w=0$ for initial data $w_0$, where in local coordinates, $\Lop w=
\d_{ij}^2 (a^{ij} w) + \d_i (b^i w) + c w$. Let $w$ solve the quasilinear
equation $(\d_t-\Lop)w=\Lopmod(f(w)w)$ for the same initial data $w_0$,
where in local coordinates $\Lopmod=\Lop + \d_i\circ \tilde b^i + \tilde c$
and $f$ is a smooth function from an interval about 0 into $\R$ satisfying
$f(0)=0$. Assume the coefficients are smooth.

Then, for sufficiently short time~$T$ and sufficiently small
$\|w_0\|_{L^\infty}$ 
(dependent on the same quantities as in Thm.~\ref{existence}),
there exists a constant~$K$ (uniform as
$T\to0$) 
such that we have the estimate
\begin{equation}\label{quadratic-est1}
\|w-\bar w\|_{C^\halpha(\UMT)} \le K \|w\|_{C^\halpha(\UMT)}
\|w\|_{L^\infty(\UMT)}
\end{equation}
and from it the time-step estimate (with a different $K$):
\begin{equation}\label{quadratic-est2}
\|\bar w(T)-w(T)\|_{C^\halpha(\UM)}\le
K \|w_0\|_{C^\halpha(\UM)}\, \|w_0\|_{L^\infty(\UM)}
\;.
\end{equation}
\end{lemma}

\begin{proof}
The proof follows the Banach fixed point argument used in proving
Theorems~\ref{linearth} and~\ref{existence}.
We decompose, as before,  the solution $\bar w$
as $\bar w=\sum \eta_l\bar w_l$ with $\bar w_l=\eta_l \bar w$ and
$\{\eta_l^2\}$ a partition of unity.
Similarly, we decompose $w=\sum \eta_l w_l$ in the same way as
$\bar w$, and refer to the proof of Thm.~\ref{linearth}, specifically
Eqn.~\eqref{pre-contraction}.  

Now from Equation~\eqref{BFP-localization}, we can copy
\begin{equation}\label{HE-w-lin}
\begin{array}{r@{}l}  \Dst
\d_t \bar w_l - \Laplace \bar w_l
&\Dst\mbox{}  =
\d_{ij}^2 \bigl[ (a^{ij}- \delta^{ij})  \eta_l  \bar w_l   \bigr] +
\d_i\bigl[ b^i\eta_l \bar w_l - 2 a^{ij}(\d_j\eta_l) \bar w_l \bigr]
\\[1ex]&\Dst\kern1.2em\mbox{}
+ c\eta_l \bar w_l
+ \bigl[a^{ij} (\d_{ij}^2 \eta_l) - b^i (\d_i\eta_l)\bigr] \bar w_l
\end{array}
\end{equation}
and likewise from~\eqref{local}
\begin{equation}\label{HE-w-NL}
\kern-2em 
\begin{array}{r@{}l}  \Dst
\d_t  w_l - \Laplace w_l
&\Dst\mbox{}  =
\d_{ij}^2 \bigl[ (a^{ij}- \delta^{ij})  \eta_l  w_l   \bigr] +
\d_i\bigl[ b^i\eta_l w_l - 2 a^{ij}(\d_j\eta_l) w_l \bigr]
\\[1ex]&\Dst\kern1.2em\mbox{}
+ c\eta_l w_l
+ \bigl[a^{ij} (\d_{ij}^2 \eta_l) - b^i (\d_i\eta_l)\bigr] w_l
\\[1ex]&\Dst \kern1.2em \mbox{}
+ \d_{ij}^2 \bigl[ a^{ij}  \eta_l  f(w_l)w_l   \bigr] +
\d_i\bigl[ b^i\eta_l f(w_l)w_l - 2 a^{ij}(\d_j\eta_l) f(w_l)w_l \bigr]
\\[1ex]&\Dst\kern1.2em\mbox{}
+ c\eta_l f(w_l)w_l
+ \bigl[a^{ij} (\d_{ij}^2 \eta_l) - b^i (\d_i\eta_l)\bigr] f(w_l)w_l
\\[1ex]&\Dst \kern1.2em \mbox{}
+\d_i[\tilde b^i \eta_l f(w_l)w_l] 
+ [\tilde c\eta_l - \tilde b^i(\d_i\eta_l)] f(w_l)w_l
\;.
\end{array}\kern-2em 
\end{equation}
Applying the heat equation estimate to the difference (which has initial
data~0), we conclude that
$$
\begin{array}{l}
\|w_l-\bar w_l\|_{C^\halpha(\RnT)}
\le C \Bigl(
\gamma\|w-\bar w\|_{C^\halpha(\UMT)}
+ \gamma \|\eta_l\|_{C^\halpha}\|w-\bar w\|_{L^\infty}
\\[1.5ex] \kern3.1cm \mbox{}
+\|a^{ij}-\delta^{ij}\|_{C^\halpha}\|w-\bar w\|_{L^\infty}
+T^{1/2} \|\eta_l\|_{C^{2,\halpha}} \|w-\bar w\|_{C^\halpha}
\Bigr)
\\[1.5ex] \kern2.4cm \mbox{}
+ C
\|(a^{ij},b^i,c,\tilde b^i,\tilde c)\|_{C^\halpha}
\|\eta_l\|_{C^{2,\halpha}} \|f(w)w\|_{C^\halpha}
\;.
\end{array}
$$
(All norms of $w$ on the right hand side are space-time norms, even if
notationally suppressed for conciseness.)
Combining contributions from all coordinate patches, and using the
contraction estimate for small times, we conclude that
$$
\|w-\bar w\|_{C^\halpha(\UMT)} \le K \|f(w)w\|_{C^\halpha(\UMT)} \le
K\|w\|_{C^\halpha(\UMT)}\,\|w\|_{L^\infty(\UMT)}
\,,
$$
i.e.~\eqref{quadratic-est1}.

From the maximum principle,
we can estimate, similarly as in the proof of Thm.~\ref{linearth}, briefly
  postponing  details, that
 $\|w\|_{L^\infty(\UMT)} \le
C\|w_0\|_{L^\infty(\UM)}$. Now we assume $\|w_0\|_{L^\infty}<1/(2KC)$, and we
obtain, still in the space-time norms, $\|w\|\le\|\bar w\|+\|w-\bar w\|\le
\|\bar w\| + \frac12\|w\|$,
hence $\|w\|_{C^\halpha(\UMT)}\le 2\|\bar w\|_{C^\halpha(\UMT)}$.
From Theorem~\ref{linearth}, we know that $\|\bar w\|_{C^\halpha(\UMT)}\le
C\|w_0\|_{C^\halpha(\UM)}$. Combining these estimates, we conclude, with a new
constant $K'=2KC$,
$$
\|w-\bar w\|_{C^\halpha(\UMT)} \le
K'\|w_0\|_{C^\halpha(\UM)}\,\|w_0\|_{L^\infty(\UM)}
\,,
$$
and this implies \eqref{quadratic-est2} immediately.

Let us now clarify the details of the comparison principle argument:
After commuting the smooth coefficients in front of the derivatives,
$w$ satisfies (with appropriate new $b^i,c,\tilde b^i, \tilde c$)
\begin{equation} \label{w-eqn}
w_t - a^{ij}\d_{ij}^2 (w+f(w)w) -b^i\d_i w - \tilde
b_i \d_i(f(w)w) - (c+\tilde c f(w)) w=0
\;.
\end{equation}
We assume
$w\mapsto w+f(w)w=:g(w)$ to be monotonic on an interval $|w|\le2\delta_1$
and choose $\delta_0<\delta_1$, assuming $\|w_0\|_{L^\infty}\le\delta_0$.
Without loss of generality we truncate $f$
smoothly to a constant outside the interval $[-2\delta_1,2\delta_1]$, leaving
$f$ unchanged on $[-\delta_1,\delta_1]$, so that the
nonlinearity is globally defined and the equation with the modified nonlinearity
is still uniformly parabolic. The modification of the nonlinearity will not
affect solutions whose range stays in $[-\delta_1,\delta_1]$.
A smallness
assumption on~$T$ dependent on $\delta_1/\delta_0$ will be imposed. Namely,
with $\max_{[-2\delta_1,2\delta_1]}|f(w)|=:M$, we want  $\lambda\ge0$ to be an
upper bound for $|c|+M|\tilde c|$.
For consistency, we require $\lambda T\le\ln(\delta_1/\delta_0)$.

We now compare $w$
with $\uw:= e^{\lambda t}(-\|w_0\|_{L^\infty} - \eps(\mu t + \rho(x))$,
where $\rho$ is the smooth radial function from
Lemma~\ref{growth} and $\lambda\ge0$ as given above, and $\mu,\eps$ are
positive constants.  Given $\eps>0$, we are assured that
$\uw\le w$ outside a compact set, and also for $t=0$.
We need to show that the same operator as in \eqref{w-eqn} applied to
$\uw$ is $\le 0$.
Then Thm.~12 in Sec 3.7 of Protter-Weinberger \cite{MR0219861} guarantees
$w\ge\uw$.  Indeed,
$$
\begin{array}{l}
\uw_t - a^{ij}\d_{ij}^2 g(\uw)
-b^i\d_i \uw - \tilde b^i \d_i g(\uw) - (c+\tilde c f(\uw)) \uw
=\mbox{}
\\[1ex]\kern2em =
\uw_t -a^{ij}g'(\uw)\d_{ij}^2\uw - g''(\uw)a^{ij}(\d_i\uw)(\d_j\uw)
-(b^i+\tilde b^ig'(\uw))\d_i \uw - (c+\tilde c f(\uw)) \uw
= \mbox{}
\\[1ex]\kern2em =
e^{\lambda t}\Bigl( [\lambda-(c+\tilde c f(\uw))](-\|w_0\|-\eps(\mu t+\rho(x)))
-\eps \mu \pm\eps O(1) \Bigr)
\;.
\end{array}
$$
By choosing $\mu$ larger than the constant in $O(1)$, independent of $\eps$,
this quantity is indeed $\le0$. Now we can let $\eps\to0$ and conclude $w\ge
-e^{\lambda t}\|w_0\|_{L^\infty}$.
A~similar argument can be made with $\bar w:= e^{\lambda t}(\|w_0\|_{L^\infty}
+ \eps(\mu t + \rho(x))$, with the inequalities reversed and an upper bound
proved.
\end{proof}

We now return to the fast diffusion equation formulated in the relative
$L^\infty$ norm on the cigar manifold $\M$, which is the motivating example for the
equations studied so far in this chapter.

In order to get a~priori control over the behavior of the dynamics for long
time in the nonlinear case, we
use a comparison principle with Barenblatt
solutions to gain such control in the $L^\infty$ norm, and a parabolic
Nash-Moser-DeGiorgi result to upgrade this control to a H\"older norm. The
a~priori control gained from these arguments will allow to iterate the
short-term existence from Thm.~\ref{existence} and gain global well-posedness.

It is of course well known that the initial value problem for the fast diffusion
equation has unique solutions, and hence the initial value problem for
$v$ is well-posed in suitable function spaces. We need however a more
precise result that establishes, in particular,
differentiable dependence of the semiflow on its initial data.
We now prove this in the H\"older spaces $C^\halpha$,  remarking afterwards
that the same proof extends to smoother spaces $C^{k,\halpha}$ and to
spaces $C^\halpha_\eta$ with more restrictive weights $\eta<0$
(but not with the more permissive weights $\eta>0$, which will be
discussed in Chapter~\ref{SecWeights:p>2}).

\begin{lemma}[Relative $L^\infty$ bounds; cf~\cite{MR1977429}]
\label{comp}
Suppose $\frac{n-2}n = m_0 < m \in `]0,1`[$ and the initial data  $u_0\in
C(\Rn)$ of a solution~$u$ to~\eqref{transformed} satisfies
$c^{-1}\uB(\x) \le u_0(\x)\le c \uB(\x)$ for all $\x \in \Rn$ and
some constant~$c\ge 1$.  Then there exists $C= C(n,m,c,B) \in [1,\infty`[$ such that
$(t,\x) \in [0,\infty`[ \times \Rn$ implies
\begin{equation} \label{relative}
   C^{-1} \uB(\x) \le u(t,\x) \le  C \uB(\x).
\end{equation}
\end{lemma}

\begin{proof}
This follows from the same comparison argument used in V\'azquez
\cite[Thm.~21.1]{MR1977429}.
Consider a quotient of Barenblatt solutions:
\begin{equation}\label{quotient}
  \left(\frac{\rhoBplus(\tau,\y)}{\uB(\y)} \right)^{-(1-m)}
  =
  \left(1+2p\tau \right) ^{-1} 
  \frac{B_+ \left(1+ 2p\tau\right)^{2\spread} + |\y|^2 }%
       {B + |\y|^2}
\;.
\end{equation}
Given any number $c \ge 1$ we can make the quotient
$\rhoBplus(\tau,\y)/\uB(\y)$ larger than $c$ (equivalently, the right
hand side of \eqref{quotient} sufficiently small) by first
choosing $\tau$ large and next $B_+>0$ small. This yields the
inequality
$$
\rhoBplus(\tau_+, \y) \ge c \uB(\y) \ge u_0(\y)
$$
for some $B_+>0$ small enough and $\tau_+$ large enough.
By the comparison principle, the solution 
of \eqref{pm} with initial condition $\rho_0:=u_0$ satisfies
$
\rho(\tau, \y) \le \rhoBplus(\tau+ \tau_+ , \y) \le C_+ \rhoBplus(\tau,\y)
$
for all $(\tau,\y) \in [0,\infty`[ \times \Rn$ and some explicitly computable
$C_+=C_+(n,m,\tau_+)<\infty$.  The corresponding transformed solutions to equation
\eqref{transformed} satisfy $u(t,\x) \le C_+ \uBplus(\x)$.
Since $1 \le \uBplus/\uB \le(B_+/B)^{-1/(1-m)}$ this completes
the proof of the second assertion in \eqref{relative}.

The same estimate with $B_-$ large and $\tau_- > -1/2p$ sufficiently
negative  can be used to get
$\rho(\tau,\y) \ge \rhoBminus(\tau+\tau_-,\y) \ge C_- \rhoBminus(\tau,\y)$
and complete the proof that $u(t,\x) \ge \uB(\x)/C$ for some $C \in [1,\infty`[$
and all $(t,\x) \in [0,\infty`[\times \Rn$.
\end{proof}

In order to use the DeGiorgi, Nash and Moser result, we note that
in local coordinates we can write  equation \eqref{adjoint}
in divergence form as
$$
u_t - \d_i \bigl((\d_u f^{ij})(x,u) \d_j u\bigr)
- \bigl((\d_u b^j)(x,u)+  (\d_i\d_uf^{ij})(x,u)\bigr) \d_j
u  = \tilde c(x,u)
$$
where
$$
\tilde c(x,u) = c(x,u) + (\d_i b^i)(x,u) +  (\d_{ij}^2f^{ij})(x,u)
$$
is bounded if $u$~is, i.e., if $u(\UMT)\subset W$.
Let us clarify the notation here: expressions of the
form $\d_i(g(x,u))$ refer to a partial with respect to $x_i$ in all locations
including implicit in~$u$, whereas $(\d_ig)(x,u)$ would refer only to the
explicit occurrence of $x_i$ in the first argument of~$g$.

We consider this equation as a linear equation of the type
\begin{equation}  \label{linearpar}
u_t - \d_i (\tilde a^{ij} \d_j u) - \tilde b^i \d_i u = \tilde c
\end{equation}
where $\tilde a^{ij}$, $\tilde b^i$ and $\tilde c$ are bounded measurable
functions and there exist
$\lambda,\Lambda >0 $ with
$$
\lambda \delta^{ij} \le \tilde a^{ij} \le \Lambda \delta^{ij}
$$
in the sense of quadratic forms.
Now we can apply

\begin{theorem}[DeGiorgi-Nash-Moser] \label{DNM}
There exists $0<\halpha_0<1$ depending only on $\lambda$ and $\Lambda$
so that  any bounded weak solution $u$ to~\eqref{linearpar}
in the cylinder~$[0,2R^2] \times B_{2R}$ lies in
$C^{\halpha_0}([R^2,2R^2]\times B_R)$ and
$$
\| u \|_{C^{\halpha_0}([R^2,2R^2]\times B_R)}
\le
c(\lambda, \Lambda, R, \| \tilde b \|_{L^\infty})
\,
(\|u\|_{L^\infty([0,2R^2]\times B_{2R})} 
+ \|\tilde c\|_{L^\infty([0,2R^2]\times B_{2R})} )
\,.
$$
\end{theorem}
Proof and statement can be found in Lady\v{z}enskaja, Solonnikov and
Ural'ceva \cite{MR0241822}, III\S10.

We can now state:

\begin{theorem}[Long-time smooth dependence in $C^\halpha$ of the FDE
    on data]   \label{t-3}
  \mbox{}\hskip0pt plus2em \penalty0
  Let $\halpha \in `]0,1`[$,  $m\in`]0,1`[$ and $m> m_0=\frac{n-2}{n}$.
  Then there exists~$R$ such that for $v_0\in C^\halpha$ with
  $\|v_0-1\|_{L^\infty(\Rn)}<R$ there exists a unique solution $v\in
  C^\halpha(\MT)$ to \eqref{cyl} for each $T$.
  The map
  $$
   C^\halpha(\M)  \ni v_0 \longmapsto  v \in C^\halpha (\MT)
  $$
  has continuous derivatives of all orders.
  It is bounded in the sense that $C^{-1} \le v(t,\x)<C$ for all $(t,\x)$.
\end{theorem}
\begin{proof}
The local existence part  is an immediate consequence of Theorem \ref{existence}.

We are now in a position to see that the local existence time asserted in
Thm.~\ref{existence} is controlled only by the $L^\infty$ norm of the initial
data, and not by an otherwise conceivable deterioration of the H\"older norms.

Comparison of $\tilde u$ with $K\pm(m t+\eps_0)$ (with $m$ large enough to
control $c$, $\d b$ and $\d^2f$) shows
that there exists $T_0$ independent of
the H\"older norm of the initial data so that $u$ cannot leave $W$
before that time. We again turn to the local uniformly parabolic equation
which we write as
$$
\tilde u_t = \sum_{i} \d_i (\tilde a^{ij}(x, \tilde u) \d_i \tilde u) +
 \tilde b^i(x,\tilde u )\d_i \tilde u  + \tilde c(x,\tilde u).
$$
By Theorem \ref{DNM} and uniformity of the manifold there exists $\halpha_0 >0$
and $C$ so that
$\| u \|_{C^{\halpha_0}([\tau,\min \{T, T_0\}] \times \UM ) } \le C$
with $C$ depending only on $W$ (via the bounds of derivative of the
nonlinearities, and the ellipticity of the operator)  for all $\tau>0$.

Localizing again on sufficiently small spatial and temporal scale we obtain
again local equations
$$
\tilde u_t - \Laplace \tilde u
= \d_{ij}^2 \Bigl(f^{ij}(x,u_0+\tilde u) - \delta^{ij}(u_0+\tilde u)\Bigr)
+ \d_i b^i(x,u_0+ \tilde u) + c(x,u_0 + \tilde u).
$$
Now
$$
\d_{ij}^2 (f^{ij}(x,u_0+\tilde u)-\delta^{ij} u )
= \d_{ij}^2 (f^{ij}(x,u_0+\tilde u) - f^{ij}(x,u_0)- \delta^{ij} u)
+ \d_{ij}^2 f^{ij}(x,u_0).
$$
Multiplying by a cutoff function we obtain
$$
\| u \|_{C^\halpha( [t_0+ R^2,t_0+2R^2]\times B_R(x_0)) }
\le c \, \| u \|_{L^\infty([t_0, t_0+2R^2] \times B_{2R}(x_0))}
$$
which gives the bound of Thm.~\ref{existence} as long as $u$ assumes values
in~$W$.

Therefore  the only way that a solution may cease to exist is by $v$ tending to
zero or infinity. This is impossible by  the pointwise bounds of Theorem
\ref{comp}.
This proves the global semiflow property.
\end{proof}

\begin{remark}[Smoother spaces]
\label{highreg}
We can carry out the same proof in $C^{k,\halpha}$
spaces, using the analog of \eqref{HE-est1} for these spaces, or by
looking at systems of equations for the derivatives.
\end{remark}

We observe that without any change in the argument we obtain a similar global
result in the function space $\Ceta^\halpha$ for $\eta\le0$. This space is
defined by the norm
\begin{equation}\label{Cetadef0}
\| f \|_{\Ceta^\halpha} := \| (\cosh s)^{-\eta} f \|_{C^\halpha} \;.
\end{equation}
Let us be more precise. We define $\check v = (\cosh s)^{-\eta} v$,
and then $\check v$ satisfies an equation similar to~\eqref{cyl}, for which the
same proof carries over. In this process, $\eta\le 0$ is needed to guarantee that
$\Ceta^\halpha$ still embeds into $L^\infty$.
Otherwise quadratic and higher order terms accumulate positive
powers of $\cosh s$ and cannot be controlled anymore.

We will also obtain certain analogs of Thm.~\ref{existence} and
Lemma~\ref{quadratic-error} in the case $\eta>0$ in
Chapter~\ref{SecWeights:p>2}. They
are of a mixed nature inasmuch as the hypotheses and statements
still need to enforce the boundedness of $w$ that is not enforced by the
$\Ceta^\halpha$ norm itself.

\section{The spectrum of the linearized equation}
\label{SecSpecLinEqn}

Since the spectral calculations of~\cite{MR1982656} and~\cite{MR2126633} have
been done in a different framework, some explanations are at hand to obtain a
valid comparison.
Here we have adopted the normalization conventions of the announcement
  \cite{MR1982656}, which differ from those of \cite{MR2126633}.

Our fast diffusion equation \eqref{pm}, \eqref{transformed} differs from the
one in~\cite{MR2126633} by a factor $1/m$ in front of the
Laplacian and a factor $\frac2{1-m}$ in front of the rescaling term.
This difference is explained by a correspondence
$\x_{\mbox{\tiny\cite{MR2126633}}}=
\x\sqrt{\frac{2m}{1-m}}$, 
$t_{\mbox{\tiny\cite{MR2126633}}} = \frac{2}{1-m}\,t$, 
$C_{\mbox{\tiny\cite{MR2126633}}}=\frac{2m}{1-m}B$.
$\tau$ coincides
between~\cite{MR2126633} and here,
but $\alpha_{\mbox{\tiny\cite{MR2126633}}}=1/\spread$.   The linearization
operator in~\cite{MR2126633} was
$$
\Hop_{\mbox{\tiny\cite{MR2126633}}} \Psi
=
-m \uB^{m-1} \Laplace\Psi + \x \cdot\nabla\Psi
$$
(of which we should drop the $m$ in front of the Laplacian
and add a factor of $\frac{2}{1-m}$ in front of the rescaling term
to account for our
space variable, and which was defined there as the {\em negative\/}
linearization operator). Let us import the spectral results
from~\cite{MR2126633}, but 
{\em adapted to the notation conventions of the present
paper, which coincide with those of \cite{MR1982656}:}
\begin{theorem}[Spectral theory of the linearized
    operator 
    \cite{MR2126633}]\label{spectrum-from-ARMA}
The operator
\begin{equation}\label{hop}
\Hop : \Psi\mapsto \uB^{m-2}\Div[\uB\nabla\Psi] =
\uB^{m-1}\Laplace\Psi - \frac{2}{1-m} \x\cdot\nabla\Psi
\end{equation}
is essentially self-adjoint on the
Hilbert space defined by the norm $(\int \uB |\nabla\Psi|^2 \,d\x)^{1/2}$ (with
constant functions $\Psi$ modded out). The restriction of $\Hop$ to the
eigenspaces of the
spherical Laplacian~$\Laplace_{\Sph}$ with eigenvalue~$-\ell(\ell+n-2)$
(where $\ell=0,1,2,\ldots$) has continuous spectrum for
$$
\lambda\le \lambda_\ell^{\rm cont}
:= -\left[
   \frac{1}{(1-m)^2} +  \Bigl(\frac{n}{2}+\ell-1\Bigr)^2
    - \frac{n-2}{1-m}
\right]
=
   - \ell(\ell+n-2) - \biggl(\frac{p}{2}+1\biggr)^2
$$
and finitely many eigenvalues
$$\textstyle
\lambda_{\ell k} =
   -\frac{2}{1-m}(\ell+2k) + 4k(k+\ell+ \frac{n}{2}-1) 
= -(\ell + 2k) p - n\ell - 4k(1-\ell-k)
$$
for integers $k$ satisfying
$$\textstyle
0\le k < \frac12 \left[\frac{1}{1-m} - \frac{n}2 + 1 -\ell\right] =
\frac12\left[\frac{p}2+1-\ell\right]  ,
$$
with corresponding eigenfunctions $\psi_{\ell k}(r) Y_{\ell \mu}(\x/|\x|)$
where the $Y_{\ell\mu}$ are spherical harmonics and
$$
\psi_{\ell k}(r) = r^\ell {}_2F_1\left(
\begin{array}{r}  k+\ell-1 -p/2,\;  -k\\
                               \ell+n/2\end{array};
-\frac{r^2}{B}
\right),
$$
which are polynomials of degree $\ell+2k$ in~$r$.
\end{theorem}

Moreover, the operator~$\Lop$ studied here is still not exactly the same as the
operator $\Hop$ just imported from~\cite{MR1982656}, \cite{MR2126633}.
This is due to a different
linearization. In~\cite{MR1982656}, \cite{MR2126633}, the linearization was defined in terms of
the pushforward of a measure, $(I+\eps\nabla\Psi)_{\#}\uB$, whereas here we
have the linearization  $\uB(1+\eps \vb)$. At least on a formal level,
$$
\begin{array}{l}
\Dst
\bigl( (I+\eps\nabla\Psi)_{\#}\uB \bigr)(\x) =
\uB\left( (I+\eps\nabla\Psi)^{-1} \x\right) \big/ \det (I+\eps\nabla\Psi)
=\mbox{}
\\[2ex]\Dst\kern4em=
\bigl(\uB(\x) - \eps \nabla \uB(\x) \cdot \nabla\Psi(\x)
\bigr) \bigl(1-\eps\tr D(\nabla\Psi)\bigr)
+O(\eps^2)
\\[2ex]\Dst\kern4em=
\uB(\x) -\eps \bigl(
\nabla \uB(\x) \cdot \nabla\Psi(\x)
+\uB(\x)\Laplace\Psi(\x) \bigr)
+O(\eps^2)
\\[2ex]\Dst\kern4em=
\bigl( \uB -\eps \Div[\uB\nabla\Psi] + O(\eps^2) \bigr) (\x)
\,,
\end{array}
$$
so $-\vb=\uB^{-1}\Div(\uB\nabla\Psi)=:\Lambda\Psi$.  So we should have
$\Lop\circ\Lambda=\Lambda\circ\Hop$. This can be confirmed by a look at the
`factorized' forms in \eqref{hop} and \eqref{linlop-cartes}.
Since $\Lambda=\uB^{1-m}\circ\Hop$,
there is also an easier conjugacy:
$\Lop\circ \uB^{1-m}  = \uB^{1-m} \circ\Hop$. (In the notation
\eqref{conjugated}, this means $\Hop=\Lop_{\eta=-2}$.)

In Theorem~\ref{spectrum-from-ARMA}, an eigenvalue $\lambda_{00}$ with
eigenfunction a
constant, does, strictly speaking, not exist, because constants were modded out
in the Hilbert space $W^{1,2}_{\uB}$ defined by the norm
$(\int\uB|\nabla\Psi|^2d\x)^{1/2}$.
In our context however, the constant function $\psi_{00}$ (which becomes
proportional to $\uB^{1-m}$ under conjugacy) will correspond to the
derivative of the Barenblatt with respect to mass: indeed,
from~\eqref{barenblatt},  $(\d_B\uB)/\uB= \frac{-1}{1-m}\uB^{1-m}$. The mass
was of course fixed a~priori in the (mass-transport based)
linearization formalism of~\cite{MR1982656}\cite{MR2126633}.

We can also introduce the Hilbert space $L^2_{\uB^{2-m}}$, which is defined by
the norm
$(\int\uB^{2-m}|\Psi|^2d\x)^{1/2}$. Now first by abstract functional analysis,
then by explicit integration by parts, we have
$$
\langle \Hop^{1/2} \Psi_1; \Hop^{1/2}\Psi_2 \rangle_{W^{1,2}_{\uB}}  =
\langle \Psi_1;\Hop \Psi_2 \rangle_{W^{1,2}_{\uB}} =
\langle \Hop \Psi_1; \Hop \Psi_2 \rangle_{L^2_{\uB^{2-m}}}  
\,.
$$
Thus, as the quartet \cite{BDGV10} also found, 
we get a Hilbert space isomorphism $\Hop^{1/2}: W^{1,2}_{\uB} \to
L^2_{\uB^{2-m}}$ and $\Hop$ is self-adjoint with the same spectrum and
eigenfunctions in $L^2_{\uB^{2-m}}$ as well, except that $\lambda_{00}=0$ with
the constant eigenfunction is genuine on $L^2_{\uB^{2-m}}$.

For~$\Lop$, which is self-adjoint in the Hilbert space $L^2_{\uB^m}$,
this means (at a formal level, since we want to consider~$\Lop$
in a H\"older space setting) that the eigenfunctions for $\lambda_{\ell k}$ are
now
$\uB^{1-m}\psi_{\ell k}(r) Y_{\ell\mu}(\x/|\x|)$. They will lie in the H\"older
space $C^\halpha$ only if  they are
bounded. Different weights in the H\"older space can however be used to
incorporate these formal eigenfunctions into the space, or, equivalently,
we can conjugate the operators by an appropriate weight and analyze the
conjugated operators in the unweighted space
$C^\halpha$. In any case, the critical growth threshold for a function to be in
$L^2_{\uB^m}$ differs from the growth threshold for a function to be in
$C^\halpha$. This affects the spectral theory.
We define
\begin{equation}\label{Cetadef}
\begin{array}{l}
\Ceta^{k,\halpha}(\M):= \{ g:= (\cosh s)^\eta f \mid f\in C^{k,\halpha}(\M) \}\\
\|g\|_{\Ceta^{k,\halpha}}(\M) := \|(\cosh s)^{-\eta} g\|_{C^{k,\halpha}(\M)}
\,.
\end{array}
\end{equation}
Also note that $\eta_{cr}=\frac{p}{2}-1$, combined with
\eqref{cigarmetric} and \eqref{cigarvol}, yields
the isometry
$$
 \|f\|_{L^2_{\uB^m}(\Rn)}
      = \| (\cosh s)^{-\eta_{cr}} f \|_{L^2(\M)}
     =: \| f \|_{L^2_{\eta_{cr}}(\M)}
$$
and thus displays that $\eta_{cr}$ gives the critical growth for
selfadjointness on~$\M$.

Given an unbounded operator $\Lop:dom \Lop \longrightarrow C^\alpha$
defined on a linear subspace $dom \Lop \subset C^\alpha$, recall that
$\lambda \in \C$ is said to be in the spectrum of $\Lop$ if
$(\Lop -\lambda I)$ does not have a bounded inverse on $C^\alpha$.
There are two subclassifications of spectra that we also refer to:
into {\em point, continuous}, and {\em residual} on the one hand,  and into
{\em essential} and {\em inessential} on the other.
More specifically,  $\lambda$ is said to be {\em point} spectrum
if $\Lop-\lambda I$ fails to be injective,  {\em continuous} spectrum if
$\Lop - \lambda I$ is injective and its range
is dense but not closed (in which case
$(\Lop - \lambda I)^{-1}$ fails to be bounded),  and {\em residual} spectrum
if $\Lop - \lambda I$ is injective but its range fails to be dense.
We say $\Lop - \lambda I$ is {\em essential} spectrum if range$(\Lop -\lambda I)$
fails to be closed (finiteness of its codimension, when closed,
and of the dimension of its kernel, always being satisfied
in the examples below; c.f.~\cite[IV \S 5.6]{Kato76}).

The calculations from~\cite{MR2126633} can be reused, to give the following
result:
\begin{theorem}[Spectrum of~$\Lop$ in H\"older spaces]
\label{HolderSpectrum}
\hskip 0pt plus 2em 
Given $\eta \in \R$,
the formula $\Lop v:= \uB^{-1}\Div[\uB\nabla (\uB^{m-1}v)]$ defines 
an operator in
the space $\Ceta^\halpha(\M) $ with domain $\Ceta^{2,\halpha}(\M)$.
Let $\Lop_\ell$ be
its restriction to the eigenspace of $\Laplace_{\Sph}$ for eigenvalue
$-\ell(\ell+n-2)$ (where $\ell=0,1,2,\ldots$), i.e.,
$\Lop(f(r)Y_{\ell\mu}(\omega)) = (\Lop_\ell f(r)) Y_{\ell\mu}(\omega)$.

The domain of $\Lop_\ell$ is
$$
\Cetaell^{2,\halpha} := \biggl\{ f\in C^{2,\halpha}[0,\infty`[ \biggm|
\left[ \scriptsize \begin{array}{ll} 
f(0)=f'(0)=f''(0)=0 & \mbox{ if } \ell\ge3\\
f(0)=f'(0)=0 & \mbox{ if } \ell=2\\
f(0)=f''(0)=0 & \mbox{ if } \ell=1\\
f'(0)=0 & \mbox{ if } \ell=0
\end{array} \right]
\biggr\}
\,,
$$
considered as an operator in 
$$
\Cetaell^{\halpha} := \{ f\in C^{\halpha}[0,\infty`[ \mid 
f(0)=0 \mbox{ if } \ell\neq0 \}
\;.
$$

Let $\eta_{cr}=\frac{p}{2}-1$, the critical value for~$\eta$ introduced
after~\eqref{cinfty}.  The spectrum of $\Lop_{\ell}$ consists of finitely
many eigenvalues,  plus an unbounded connected component
consisting of a filled-in parabola
(which degenerates to a ray if $\eta = \eta_{cr}$).
The boundary of this parabola forms the essential spectrum.
More precisely:

(1) If $\eta=\eta_{cr}$, then $\Lop_{\ell}$ has only point spectrum. This
consists
of essential spectrum  for $\lambda\le \lambda_{\ell}^{\rm cont}$ (as in
Thm.~\ref{spectrum-from-ARMA}) in addition to discrete
eigenvalues $\lambda_{\ell k}$ of finite multiplicity for integers $0\le k<
\frac12[\frac{p}2+1-\ell]$.

(2) If $\eta>\eta_{cr}$, the spectrum is still only point, with the discrete
eigenvalues exactly the $\lambda_{\ell k}$ for $0\le k <
\frac12[\frac{p}2+1-\ell - |\eta-\eta_{cr}|]$ and
 the parabolic region
\begin{equation}\label{parabolic region}
\re \lambda 
\le 
- (\frac{p}{2}+1)^2 - \ell(\ell+n-2) + (\eta-\eta_{cr})^2 
- \Big(\frac{\im \lambda/2}{\eta - \eta_{cr}}\Big)^2
\,,
\end{equation}
i.e., $\textstyle
\re\sqrt{\T}\le \eta+1-\frac{p}2= |\eta-\eta_{cr}| \mbox{, where }
{\T}= \ell(\ell+n-2) + (\frac{p}2+1)^2+\lambda\,.
$
The largest $\re\lambda$ occurring in the  essential spectrum is precisely
at $\lambda_{\ell,\eta}^{\rm cont}:=
(\eta-\eta_{cr})^2-(\frac{p}2+1)^2-\ell(\ell+n-2)$. 
(This parabolic region
includes now those eigenvalues $\lambda_{\ell k}$ that, in comparison
with the case $\eta=\eta_{cr}$, have been `lost' from the original list of
eigenvalues.) 
For $\ell=0$ and $\eta=0$ (unweighted $C^\halpha$), this threshold
$\lambda_{0,0}^{\rm cont}$ coincides with $\lambda_{01}$.

(3) If $\eta<\eta_{cr}$, the same formula $\re\sqrt{\T}\le |\eta-\eta_{cr}|$
characterizes now the residual spectrum ($\Lop-\lambda$ has 
closed range with codimension 1), 
and the same formulas as in (2) apply
for the eigenvalues $\lambda_{\ell k}$, which now make up the only point
spectrum.

In cases (2) and (3), the boundary of the parabolic region, namely
$\re\sqrt{\T} =  \eta+1-\frac{p}2= |\eta-\eta_{cr}|$, is essential spectrum
(range $\Lop-\lambda$ not closed). The remaining spectrum is non-essential:
the kernel of\/ $\Lop -\lambda$ is finite dimensional, and its range
is closed with finite codimension.

The eigenfunction for eigenvalue $\lambda_{\ell k}$ is
 $v_{\ell k}(|\x|) Y_{\ell\mu}(\x/|\x|)$ where
$$
\begin{array}{r@{}l}
\Dst
v_{\ell k}(r)=
\uB^{1-m}\psi_{\ell k}(r)
& \Dst\mbox{}
=
(\cosh s)^{-2}\psi_{\ell k}(\sinh s)
\\[2ex]
&\Dst\mbox{}
=
\left(\frac{\sinh^\ell s}{\cosh^2 s}\right)\times 
{}_2F_1\left(
\begin{array}{r}  k+\ell-1 -p/2,\;  -k\\
                               \ell+n/2\end{array};
-\sinh^2 s
\right)
\,.
\end{array}
$$
\end{theorem}
\begin{remark}
The domain of\/ $\Lop$ is dense in the little, but not the large,
H\"older spaces~$C^\halpha$.
The latter fact seems a bit annoying from the point of view of abstract operator
theory, but will not cause us any trouble.
\end{remark}

\begin{figure}\unitlength1bp
\centerline{%
\begin{picture}(380,220)
\put(0,0){\includegraphics{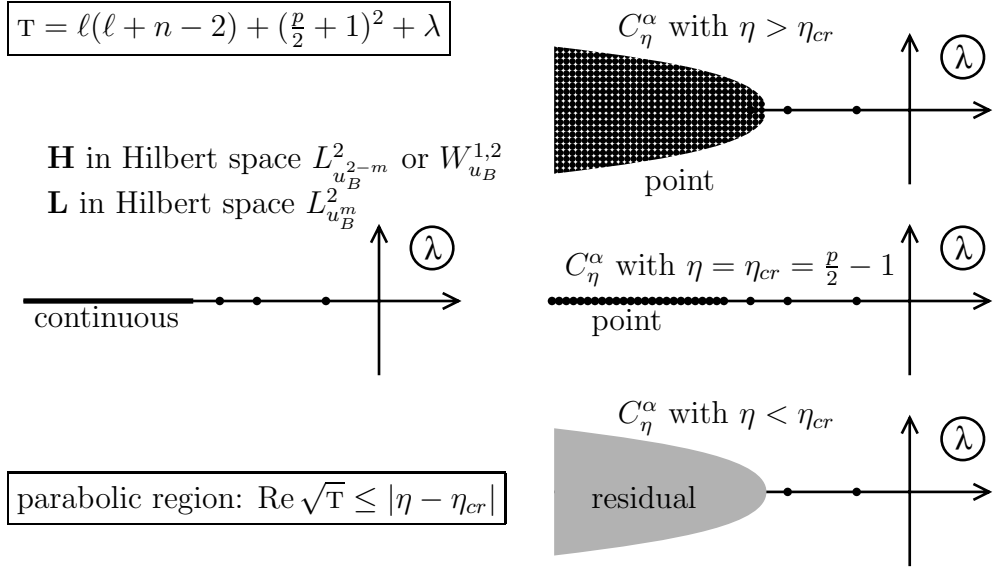}}
\put( 15,158){$\Hop$ in Hilbert space $L^2_{\uB^{2-m}}$ or $W^{1,2}_{\uB}$}
\put( 15,141){$\Lop$ in Hilbert space $L^2_{\uB^{m}}$}
\put(230,208){$\Ceta^\halpha$ with $\eta>\eta_{cr}$}
\put(210,118){$\Ceta^\halpha$ with $\eta=\eta_{cr}=\frac{p}{2}-1$}
\put(230, 62){$\Ceta^\halpha$ with $\eta<\eta_{cr}$}
\put(10,98){continuous}
\put(220,97){point}
\put(240,150){point}
\put(220,30){residual}
\put(0,208){\fbox{$\T=\ell(\ell+n-2) + (\frac p2+1)^2 + \lambda$}}
\put(0,30){\fbox{parabolic region: $\re\sqrt{\T} \le |\eta-\eta_{cr}|$}}
\end{picture}
}
\caption{\sl The spectrum of $\Lop_\ell$, in various spaces, schematically.}
\label{space+spectrum}
\end{figure}

\begin{remark}[Relevant parts of spectrum]
We will rely only on the value of the essential spectral radius of
the semigroup generated by $\Lop_\eta$ (essential spectral abscissa 
of~$\Lop_\eta$), and the eigenvalues outside the essential spectral radius, 
but not on any other aspects of the spectrum.

After decomposition into spherical harmonics, the value of the essential 
spectral radius 
could be obtained qualitatively by asymptotic methods (see, eg., Ch.~2 of
Fedoryuk \cite{MR1295032}) without reference to explicit solutions in terms of
special functions. The precise nature of the spectrum is instructive to know,
but not relevant for our arguments. Note that in contrast to the
weighted space used in \cite{MR2126633}, in the present, unweighted space, the
onset of the essential spectrum of $\Lop_0$ is precisely at $\lambda_{01}$.
\end{remark}

\begin{proof}[Proof of Thm.~\ref{HolderSpectrum}]
Lest there be any doubt, we point out beforehand that the sign of $\eta$,
$\eta_{cr}$ will not be relevant for the proof of this theorem (but will
become relevant later).
The claim that $\Lop$ defines an operator in $\Ceta^\halpha$ with the
domain $\Ceta^{2,\halpha}$ is equivalent to the claim that the conjugated
operator $\Lop_\eta= (\cosh s)^{-\eta} \circ \Lop \circ (\cosh s)^\eta$ is a
mapping $C^{2,\halpha}\to C^\halpha$, and this is clear from the
definition, see~\eqref{conjugated}; or the equation in cartesian 
coordinates to see that there are actually no singularities at the origin. 
By slight abuse of notation, we let
$\Lop_\ell$ operate on functions $f(s)$ rather than
$f(s)Y_{\ell\mu}(\x/|\x|)$. Recalling~\eqref{conjugated}, we have
\begin{equation}
\label{conjugated-ell}
\begin{split}
  \Lop_{\ell,\eta}: = &
      (\cosh s)^{-\eta} \circ \Lop_{\ell} \circ (\cosh s)^\eta
\\
  = & \,  \Bigl( \d_s^2 + \frac{2(n-1)}{\sinh2s}\d_s -
  (\tanh s)^{-2} \, \ell(\ell+n-2) \Bigr)
\\
 & + \Big( n - \frac{2m}{1-m} + 2  \eta \Big) \tanh s \, \d_s
\\
 & + n(\eta+2)
  + \Bigl(  \eta^2 - \frac{2m}{1-m}\eta - \frac{4}{1-m} \Bigr)
    \tanh^2 s
\;.
\end{split}
\end{equation}
For the claimed domains $\Cetaell^{2,\halpha}$, 
notice that if $f(r)Y_{\ell\mu}(\x/|\x|)$ is in~$C^{2,\halpha}(\M)$, 
then $f$~is $C^{2,\halpha}$ by restriction to rays. Moreover, any
$C^{2,\halpha}(\M)$ function can be written, near the origin, in the
form $c_0+\cb_1\cdot\x+c_2r^2 + h_2(\x) + o(r^2)$, with $h_2$ a
harmonic polynomial of homogeneous degree~2. (Readers looking for
details on harmonic polynomials can find them on pp.~159-163
of~\cite{BGM}). If this function is also of the form $f(r)
Y_{\ell\mu}(\x/|\x|)$, then we can multiply with various
$Y_{\ell'\mu'}$ and integrate over the unit sphere. 
Unless $\ell=0$, choosing $\ell'=0$ gives $0=c_0+c_2r^2+o(r^2)$, hence
$c_0=c_2=0$. Unless $\ell=1$, choosing $\ell'=1$ gives
$0=\cb_1$. Unless $\ell=2$, testing with $\ell'=2$ gives $h_2=0$. The
claims about $f$ and its derivatives at~0 follow. 

Now we need to see, conversely, that when $f$ is in the claimed
domains, then $fY_{\ell\mu}$ is in $C^{2,\halpha}(\M)$. It is only in
a neighbourhood of the origin that this is nontrivial. Using Schauder
theory, it suffices to show that $\Laplace(f(r)Y_{\ell\mu}(\omega)= 
(\frac{d^2}{dr^2} + \frac{n-1}{r}\,\frac{d}{dr} -
\frac{\ell(\ell+n-2)}{r^2}) f(r) Y_{\ell\mu}(\omega)$ is (uniformly)
$C^\halpha$ in a punctured neighborhood of the origin. Use of the
initial conditions at~0 in each case establishes this in a
straightforward manner.

The analysis of solutions to $(\Lop_\ell-\lambda)\psi=0$ in terms of
hypergeometric functions carried out in Sec.~4.2 of \cite{MR2126633} carries
over with only trivial modifications, and the discussion whether such solutions
lie in $C^\halpha(\M)$ reduces to growth at infinity (the smoothness being
trivial).
This establishes the precise point spectrum claimed for all cases,
and the eigenfunctions $v_{\ell k}$.

We now assume we are not in the eigenvalue case.
Translating the resolvent formula (4.36) from~\cite{MR2126633}, we can write
any solution to $(\Lop_\ell-\lambda)u=w$ as
\begin{equation}\label{resolventf}
\begin{array}{l}
u(r)= v_2(r) \left(A_0+\int_0^r w v_1 \uB^m r^{n-1} dr \right)
    + v_1(r) \left(A_1+\int_r w v_2 \uB^m r^{n-1} dr \right)
\\[1ex]
u'(r)= v_2'(r) \left(A_0+\int_0^r w v_1 \uB^m r^{n-1} dr \right)
    + v_1'(r) \left(A_1+\int_r w v_2 \uB^m r^{n-1} dr \right)
\;.
\end{array}
\end{equation}
As $r\to\infty$, we have $\uB^mr^{n-1} \sim c r^{-p/2-\eta_{cr}}$ and
$$
v_2(r)\sim c r^{\eta_{cr}-\sqrt{\T}}\;,\qquad
v_1(r)\sim c r^{\eta_{cr}+\sqrt{\T}}
\;.
$$
As $r\to0$, we have  $\uB^mr^{n-1}\sim cr^{n-1}$ and
$$
v_2(r)\sim c r^{2-n-\ell}\;,\qquad
v_1(r)\sim c r^\ell
\;.
$$

First assume $\re\sqrt{\T}>|\eta-\eta_{cr}|$ and let $w=O(r^\eta)$ as
$r\to\infty$. Then the second integrand converges near $\infty$ and we may
specify its upper limit as $\infty$. A growth estimate for~$u$ requires now
$A_1=0$, because else the dominant term $A_1r^{\eta_{cr}+\sqrt{\T}}$
is not in the space. The dominant term as $r\to0$ is
$A_0r^{2-n-\ell}$, so we need $A_0=0$ to get
$u(0)=0$ in the case $\ell\ge1$. In the case $\ell=0$, we still need $A_0=0$ for
$u'(0)$ to vanish. With these choices we do indeed get a solution~$u$ that
satisfies the bounds near 0 and infinity required for it to be in the domain of
$\Lop_\ell$ ($u''(0)=0$ in the cases $\ell\ge3$ uses $w(0)=0$); 
the smoothness is trivial.  So these cases are not in the spectrum.

Next we assume $\re\sqrt{\T}<|\eta-\eta_{cr}|$. We are then in the case
$\eta<\eta_{cr}$ because the other cases have been dealt with as point
spectrum. If $w\in {\rm range}(\Lop_\ell-\lambda)$, then
$\int_0^\infty w v_1 \uB^m r^{n-1}\,dr=0$ by symmetry of $\Lop_\ell$
with respect to the inner product in $L^2(\R^+,\uB^mr^{n-1}dr)$, which
{\em is\/} defined in this particular case because the integrand is
bounded by $O(r^{\eta+\re\sqrt{\T}-\eta_{cr}-1})$ for large~$r$.
Let us assume $w$ satisfies this orthogonality condition, which restricts $w$
to a closed subspace of codimension 1.
We also
have the convergence that allows us to have the upper limit $\infty$
in the second integral of the resolvent formula~\eqref{resolventf}.
Moreover, if $w=(\Lop_\ell-\lambda)u$ we need $A_1=0$ in the resolvent formula
for the same growth reasons.
We can then write
$$
u(r)= A_0v_2(r) - v_2(r)\int_r^\infty w v_1 \uB^m r^{n-1} dr
    + v_1(r) \int_r^\infty w v_2 \uB^m r^{n-1} dr
\;.
$$
We now need $A_0=0$, because the first term is the only one that would grow
faster than $O(r^\eta)$. As before, we now confirm that $u$ does have the
required behavior as $r\to0$ to be in the domain of $\Lop_\ell-\lambda$.
We have (non-essential) residual spectrum in this case (injective, range closed
and with codimension~1).

Let us now return to the point spectrum in $\re\sqrt{\T}<|\eta-\eta_{cr}|$ in
the case $\eta>\eta_{cr}$ to see that this spectrum is non-essential. Indeed,
we want to show that $\Lop_\ell-\lambda$ is onto in this case. We cannot choose
$\infty$ as upper limit of integration in~\eqref{resolventf} and choose~1
instead. $A_1$ is arbitrary anyways, since $v_1$ is a bona-fide  eigenfunction.
The choice $A_0=0$ is required for the same reasons as before, and with this
choice we do get a solution $u$ in the domain for every choice of~$w$ in the
space $\Cetaell^\halpha$.

Finally, we consider the cases on the boundary of the parabolic region. Let
first  $\re\sqrt{\T}=\eta-\eta_{cr}\ge0$. We know we
have point spectrum, and we want to show that the range is {\em not\/} closed.
Not being assured of convergence at $\infty$, we write
$$
u(r)= v_2(r) \left(A_0+\int_0^r w v_1 \uB^m r^{n-1} dr \right)
    + v_1(r) \left(A_1-\int_1^r w v_2 \uB^m r^{n-1} dr \right)
$$
For $w$ to be in the range, it is necessary that the function
$r\mapsto \int_1^r w v_2 \uB^m r^{n-1} dr$ remains bounded as $r\to\infty$,
which is not automatic: $w\in \Ceta^\halpha$ would allow for logarithmic
divergence. Assume this hypothesis is verified. Then the choice
$A_0=0$ does give a solution that satisfies the right bounds (and
automatically the needed smoothness, too). With this characterization
of the range, it can be seen that
$w_*=r^{2(\eta-\eta_{cr})} \bar{v}_2 \chi(r) /\ln r$ (where $\chi$ is a
truncation at~$0$, constant~1 for $r\ge1$) is not in the range, but
$w_N=r^{2(\eta-\eta_{cr})} \bar{v}_2 \chi(r) /(\ln r+\frac rN)$ is in the range,
for every~$N$, and $w_N\to w_*$ uniformly, and also in the H\"older norm.

Now, for $\eta<\eta_{cr}$ and $\re\sqrt{\T}=\eta_{cr}-\eta$, we can write the
second integral as $\int_r^\infty$ again, and need $A_1=0$, $A_0=0$. We get a
similar conclusion, namely that $w$ is in the range if and only if
$r\mapsto \int_0^r w v_1 \uB^m r^{n-1} dr$ is bounded as $r\to\infty$. With the
same argument as before, the range is not closed. Is the range dense or not?
We claim it is not dense, thus identifying the spectrum as residual rather than
continuous.
Indeed, if $\|g\|_{\Ceta^\halpha}<\eps$, then
$w(r):= r^{2(\eta-\eta_{cr})}\bar{v}_1(r) \chi(r) + g(r)$ parametrizes an open
set of functions that is not in the range, where the implied smallness of the
weighted $L^\infty$ norm of $g$ prevents it from compensating for the
logarithmic divergence of the integral caused by the first term.
\end{proof}

The spectrum of $\Lop$ in $\Ceta^\halpha$ is the same as the spectrum of
$\Lop_\eta$ (defined in \eqref{conjugated}) in~$C^\halpha$. So we can now
study the differential equation
$v_t - \Lop_\eta v = 0$.
It follows from Rmk.~\ref{highreg} that the linear semigroup
$\Sop_\eta(t)$ generated is analytic on $C^\halpha(\M)$.
We want to obtain the usual semigroup estimates
on the complement of the space spanned by the eigenfunctions $v_{\ell k}$.
Major work goes into the boundedness result: $\|\exp[t\Lop_\eta]\|\le C
\exp[c_\infty(\eta) t]$ {\em without\/} an extra $\eps$. The loss of
self-adjointness makes this result non-trivial, in particular since we have
established non-zero Fredholm index for the onset of the essential spectrum
of~$\Lop_\eta$. We claim
\begin{theorem}[Semigroup Estimates] \label{semigroupestimates}
   Equation \eqref{conjugated} defines an
   analytic semigroup $\Sop_\eta(t) = \exp t \Lop_\eta$ on $C^{\halpha}(\M)$ or
   on $L^\infty(\UM)$, in the sense of Remark~\ref{analytic-semigroup}.

   The essential spectrum of $\Sop_\eta(t)$ is contained
   in $\overline{B(0,e^{c^\infty(\eta)  t})}$ and in no smaller ball, where
   $c^\infty(\eta)$ is given by \eqref{cinfty} and equals the
   $\lambda_{0,\eta}^{\rm cont}$ from Thm.~\ref{HolderSpectrum}.
   Outside this ball there are only finitely many eigenvalues $e^{\lambda_j t}$
   with $c_\infty(\eta) < \lambda_1 \le \lambda_2\ldots$ where the $\lambda_j$
   are the $\lambda_{\ell k}$ from Thm.~\ref{HolderSpectrum}

   The spectral projections onto the eigenspace corresponding to a set of
   eigenvalues $\lambda_{\ell k}$ of the operator $\Lop$
   in~$L^2_{\uB^m}(\Rn)=L^2_{\eta_{cr}}(\M)$ are, by the same formula,
   also well-defined mappings in~$\Ceta^\halpha(\M)$, or correspondingly
   weighted
   $L_{\eta}^\infty(\M)$, provided the $L^2_{\eta_{cr}}(\M)$-eigenfunctions
   for~$\lambda_{\ell k}$  are still in the space~$\Ceta^\halpha(\M)$.

   Let $v_j$ be the eigenfunctions corresponding to the $\lambda_j$ in
   $L^2(\M)$.
   Then, for $\eta\neq\eta_{cr}$, each solution $\vb=\Sop_\eta(t)\vb_0$
   with initial data $\vb_0\in L^\infty$ can be  written in the  form
   \begin{equation}  \label{slowmfd-decomp}
        \vb(t,\x) = \sum c_j e^{\lambda_j t} v_j(\x) + \vb_{res}(t,\x)
   \end{equation}
   with
   \begin{equation}   \label{slowmfd-est}
       \sup_{t \ge  1  } e^{-c^\infty(\eta) t} \| \vb_{res} \|_{L^\infty(\M)}
       < \infty.
   \end{equation}
  For $\eta=\eta_{cr}$, we at least get the same conclusion with
   $c^\infty(\eta)+\eps$ instead of $c^\infty(\eta)$, for arbitrary
  $\eps>0$.
  For $\eta=\eta_{cr}$, we also get the analog of \eqref{slowmfd-est}
  with the $L^2(\M)$ norm (and no~$\eps$).
\end{theorem}

\begin{proof}
It has been proven in Theorem~\ref{linearth} that the semigroup is analytic. 
The spectral statements  follow from
Thm.~\ref{HolderSpectrum} by the functional calculus for analytic
semigroups (e.g., Cor.~IV, 3.12 in Engel, Nagel \cite{MR1721989}). In using
this reference, note that strong continuity in~0 is assumed there; so we get
the estimate of the spectral radius for the little H\"older
spaces~$\oC^\halpha$
directly, in view of our Remark~\ref{analytic-semigroup}(c). However, the
result also carries over to the big H\"older spaces $C^\halpha$, because the
semigroup maps $C^\halpha$ initial data continuously into $\C^{\halpha'}\subset
\oC^\halpha$ (for $\halpha'>\halpha$) in any
short time $t_0$, as a consequence of the regularization estimate in
Thm.~\ref{linearth}.

Clearly the projections $C^\halpha(\M) \ni u(s,\omega)  \longmapsto
u_{\ell\mu}(s)=\int u(s,\omega)
Y_{\ell\mu}(\omega) \,d\omega \in C^\halpha[0,\infty`[$
are continuous. If $\ell>0$, the functions $u_{\ell\mu}$ satisfy
$u_{\ell\mu}(0)=0$, and we denote the space of those
$C^\halpha[0,\infty`[$-functions that vanish
at 0 by $C^\halpha_0[0,\infty`[$.

Conversely, the imbeddings $C^\halpha_0[0,\infty`[\ni u(s)
\longmapsto u(s)Y_{\ell\mu}(\x/|\x|) \in C^\halpha(\M)$
are continuous for $\ell>0$, and
likewise the analogous imbedding $C^\halpha[0,\infty`[\to C^\halpha(\M)$ in the
case $\ell=0$.
The $L^\infty$ estimate being trivial, we deal with the H\"older quotient:
$$
\frac{|f(r_1)Y(\omega_1)-f(r_2)Y(\omega_2)|}%
     {d((r_1,\omega_1) , (r_2,\omega_2))^\halpha}
\le
\frac{|f(r_1)-f(r_2)|\,|Y(\omega_1)|}%
     {d(r_1 , r_2)^\halpha}
+
\frac{|f(r_2)| \, |Y(\omega_1)-Y(\omega_2)|}%
     {d((r_1,\omega_1) , (r_2,\omega_2))^\halpha}
$$
The first term is bounded in terms of $[f]_\halpha$.
The second term vanishes if $\ell=0$. If however $\ell\ge1$, we use that
$f(0)=0$ and therefore $|f(r)|\le O(d(0,r)^\halpha)=O(s^\halpha)$.
Then the second term is bounded by
$$
c \frac{s_2^\halpha d(\omega_1,\omega_2)}%
{\max\{|s_2-s_1|^\halpha, (\tanh s_2)^\halpha\, d(\omega_1,\omega_2)^\halpha\}}
$$
where $s_1\ge s_2$ without loss of generality. The estimate follows immediately.

Finally,  the inner product
$\langle v,v_{\ell k}\rangle_{L^2_{\uB^m}}$
is well-defined for $v\in C^\halpha[0,\infty`[$ provided
$\eta + \ell + 2k < p$. By the constraint on~$k$, we have indeed
$\eta+\ell + 2k < \eta + \frac{p}2+1-|\eta-\eta_{cr}|=
p+\eta-\eta_{cr}-|\eta-\eta_{cr}| \le p$. So again these projections are
continuous in each $\Ceta^\halpha$, and so are trivially the corresponding
imbeddings.

So, since for given~$m<1$, there are only finitely many eigenvalues
$\lambda_{\ell k}$, we conclude that the spectral projection $\Qop$
onto the span of their eigenspaces that was naturally constructed within the
$L^2_{\uB^m}$ Hilbert space framework, namely
$$
\Qop v:= \sum_{\ell,\mu} \sum_{k\le \bar{k}(\ell)}
\frac{\langle v ; v_{\ell k}Y_{\ell\mu}\rangle_{L^2_{\uB^m}}}%
     {\langle v_{\ell k}Y_{\ell\mu} ; v_{\ell
     k}Y_{\ell\mu}\rangle_{L^2_{\uB^m}}}
v_{\ell k} Y_{\ell\mu}
\;,
$$
defines, by the same formula, a spectral projection in the $\Ceta^\halpha(\M)$
framework, too. Here,
$\bar{k}(\ell)= \frac12[\frac{p}2+1-\ell - |\eta-\eta_{cr}|]$ is from
Thm.~\ref{HolderSpectrum}.
Let $\Pop:=1-\Qop$ be the complementary projection in $\Ceta^\halpha$.

We turn to  the semigroup estimate and consider a solution $\vb$ with bounded
initial data. The key point is that we get estimate \eqref{slowmfd-est} with
precisely $c^\infty(\eta)$ in the exponent, rather than $c^\infty(\eta)+\eps$,
an issue that has become nontrivial because we lost the notion of
self-adjointness when abandoning the Hilbert space setting.
There is nothing to do in the case $\eta=\eta_{cr}$, where we did not claim
such an improvement. Also, the $L^2(\M)$ estimate for $\eta=\eta_{cr}$
is clear because the operator is self-adjoint (with the Riemannian
volume on the cigar as measure) in this case.

We will distinguish the cases $\eta<\eta_{cr}$ and  $\eta>\eta_{cr}$.
We rely on the fact that
$c^\infty(\eta)>c^\infty(\eta_{cr})$ and a trade-off between time decay and
spatial growth that will be exhibited during the proof details.
The key idea is that for $s\to\infty$, the maximum principle `almost' gives
the correct growth rate $c^\infty(\eta)$ in the $L^\infty$ norm. There is a
remnant of the $0^{\rm th}$ order term that decays like $O(1/\cosh^2 s)$ and
that could give rise to secular terms; controlling this effect gives rise to
the said trade-off between spatial growth and time decay. To control the
estimates for small $s$, we use conjugacy and estimates with a better decay
constant than $c^{\infty}(\eta)$ for data in the correct spectral space.

The $\Lop_\eta$ invariant projection~$\Pop_\eta$ corresponding to those
eigenvalues that are strictly
larger than $c^\infty(\eta)$ immediately
gives rise to the decomposition \eqref{slowmfd-decomp}.
In case $c^\infty(\eta)$ is itself among the $L^2_{\uB^m}$ eigenvalues
$\lambda_{\ell k}$, components in its eigenspace immediately satisfy
\eqref{slowmfd-est}. So we can project this eigenvalue away, too, and need to
prove \eqref{slowmfd-est} only for $\vb:= \vb_{res}$ in the complementing space.
Moreover, for the price of a small time shift (regularization estimate in
Thm.~\ref{linearth}), we may assume $\vb_0\in C^\halpha$.

\noindent{\bf Case 1: $\eta < \eta_{cr}$.}
As remarked after \eqref{Cetadef}, multiplication by $(\cosh s)^{\eta_{cr}}$
gives an isometry between $L^2(\M)$ and $L^2_{\uB^m}(\Rn)$.
Let us first conjugate the problem with
$(\cosh s)^{\eta_{cr}-\eta}$ to obtain the self-adjoint operator
$\Lop_{\eta_{cr}}$
with $L^2(\M)$ initial data $\vc_0 := (\cosh s)^{\eta-\eta_{cr}} \vb_0$. We
find some $c^\infty(\eta)-\eps$ strictly above the next smaller of
the $L^2(\M)$ eigenvalues (or, if there are no such eigenvalues, strictly above
$c^\infty(\eta_{cr})$).

\begin{figure}
\fbox{
\begin{minipage}{0.95\textwidth}
\begin{center}
\begin{picture}(345,62)
\put(0,50){Flow of}    \put(185,50){Flow of}
\put(0,36){$\Lop_\eta \mbox{ on } \vb\in C^\halpha\,$,\kern2em or}
\put(0,22){$\Lop \mbox{ on } (\cosh s)^\eta \vb \in \Ceta^\halpha$}
\put(0,6){$\phantom{\Lop \mbox{ on } (\cosh s)^\eta \vb} \subset L^2(\M)$}
\put(185,36){$\Lop_{\eta_{cr}} \mbox{ on } \vc= \mbox{}$}
\put(185,22){\kern1em$\mbox{}= (\cosh s)^{\eta-\eta_{cr}}\vb\in
                                     C^\halpha_{\eta-\eta_{cr}}$}
\put(185,6){\kern1em$\phantom{\mbox{}= (\cosh s)^{\eta-\eta_{cr}}\vb}
                                    \subset C^\halpha(\M)$}
\put(160,6){\footnotesize\it \ldots since $\eta<\eta_{cr}$ \ldots}
\end{picture}\\
We get a spectral gap from the $L^2$ theory, \\
but the estimates do not enforce sufficient decay at $\infty$.

\begin{picture}(285,83)
  \put( 0,  0){\includegraphics{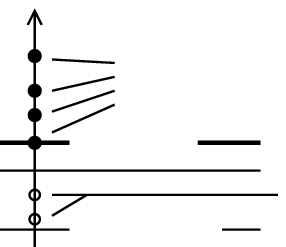}}
  \put(35, 51){\small estimate applies to complement of eigenspace}
  \put(35, 40){\small for these eigenvalues}
  \put(25, 28){\small $c^\infty(\eta)$}
  \put(78, 24.5){\scriptsize$\}$\normalsize$\,\eps$}
  \put(83, 13){\small possible further $L^2$ eigenvalues}
  \put(99,  2){\small (in residual spectrum for
                $\Ceta^\halpha$)}
  \put(25, 3){\small $c^\infty(\eta_{cr})$}
  \put(14,64){\small$\lambda$}
\end{picture}%
\kern-25bp
\fbox{%
\begin{picture}(85,70)
  \put(0,0){\includegraphics{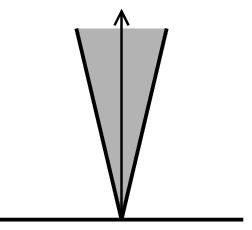}}
  \put(73,3){$\M$}
  \put(39,62){$t$}
\end{picture}
}

\medskip

In the gray wedge the spectral gap can compensate for the weight factor that is
too weak. Outside, the maximum principle has to give growth control.
\end{center}
\end{minipage}
}
\caption{\sl Schematic outline of proof idea for $\eta<\eta_{cr}$: In the shaded
  cone, the spectral gap $\eps$ compensates for the `too slow' decay. In the
  exterior region, the estimate carries over by the maximum principle.}
\label{fastdecay-idea-figure}
\end{figure}

By the spectral properties of $\Lop_{\eta_{cr}}$ in~$C^\halpha(\M)$,
we conclude, uniformly in $t$,
$$
\| \vc(t) \|_{C^\halpha(\M)} \le
c e^{(c^\infty(\eta)-\eps) t}\|vc_0\|_{C^\halpha(\M)} \;;
$$
using $\|vc_0\|_{C^\halpha}\le c\|\vb_0\|_{C^\halpha}$, this gives in
particular the weighted $L^\infty$ estimate
$$
|\vb(t,\x)| \le c (\cosh s)^{\eta_{cr}-\eta}\,
e^{(c^\infty(\eta)-\eps) t}\|\vb_0\|_{C^\halpha} \;.
$$
So we choose  $\delta=\eps/(\eta_{cr}-\eta) >0$ such that
\begin{equation}
\sup_{s\le \delta t} |\vb(t,\x)|
\le c e^{c^\infty(\eta)t} \|\vb_0\|_{C^\halpha}
\;.
\end{equation}
We want to apply the maximum principle to obtain the same kind of estimate
in the set $s \ge \delta t$ as well.
To this end, $\vbb := e^{-c^\infty(\eta)t} \vb$
satisfies $\bigl(\d_t-(\Lop_\eta-c^\infty(\eta))\bigr)\vbb=0$.
While $\Lop_{\eta}-c^\infty(\eta)$ barely fails to satisfy a straightforward
maximum principle --~its $0^{\rm th}$ order term is positive according to
\eqref{conjugated}~--, this term is $\le c/\cosh^2s\le ce^{-2\delta t}$.
So we let $a(t):= c^\infty(\eta)t + \int_0^t c e^{-2\delta t}\,dt$ and
define instead the new $\vbb:= \exp[-a(t)] \vb$, which solves
$\bigl(\d_t-(\Lop_\eta-a(t))\bigr)\vbb=0$ and satisfies a maximum principle.
We get therefore, in the domain $s\ge\delta t$:
$$
|\vb(t,\x)| =  e^{a(t)}|\vbb(t,\x)| \le e^{a(t)} c\|\vb_0\|_{L^\infty}
\le C e^{c^\infty(\eta)t} \|\vb_0\|_{L^\infty} \;.
$$
Together we have shown
$\|\vb(t)\|_{L^\infty} \le C e^{c^\infty(\eta)t} \|\vb_0\|_{C^\halpha}$.
With the regularization estimate
$\|\vb(t+1)\|_{C^\halpha}\le C \|\vb(t)\|_{L^\infty}$ from Thm.~\ref{linearth},
the desired conclusion follows.

\bigskip

\noindent{\bf Case 2:  $\eta > \eta_{cr}$.}
Here we let $\vb^{ext}$ be the solution of an exterior initial boundary value
problem that can be estimated with the maximum principle;  then the
remainder $\vb-\vb^{ext}$ has compactly supported initial data, which allow for
using the conjugacy again. This decomposition idea requires the use of a cutoff
function~$\chi$, which in turn entails writing some explicit projections to get back
into the spectral spaces destroyed by the cutoff. Moreover, the remainder
$\vb-\chi\vb^{ext}$ satisfies the differential equation with an inhomogeneity
created by the cutoff; however the effect of the inhomogeneity can be
controlled, see~\eqref{rhscalc}--\eqref{wint-decompose} below.

In detail, we first choose $\eps$ such that $c^\infty(\eta)-2\eps$  is still
larger than $c^\infty(\eta_{cr})$, and any larger pertinent $L^2$
eigenvalues buried in the essential spectrum for $\Ceta^\halpha$ as well.
and we choose $\delta$ such that $\delta(\eta-\eta_{cr})=\eps$.
We then we choose $(t_1,\x_1)$ (with the geodesic radial
coordinate $s_1$) and
let $\vb^{ext}$ solve $(\d_t-\Lop_\eta)\vb^{ext}=0$
in $s>s_1+\delta(t_1-t)$,
$0\le t\le t_1$, with boundary data $\vb^{ext}=0$ if $s=s_1+\delta(t_1-t)$, and
initial data $\vb^{ext}_0=\vb_0$ for $s>s_1+\delta t_1 + 1$, and smoothly cut
off to~0 for $s\in[s_1+\delta t_1, s_1+\delta t_1+1]$.
$$
\makebox{%
\begin{picture}(195,107)
\put(0,0){\includegraphics{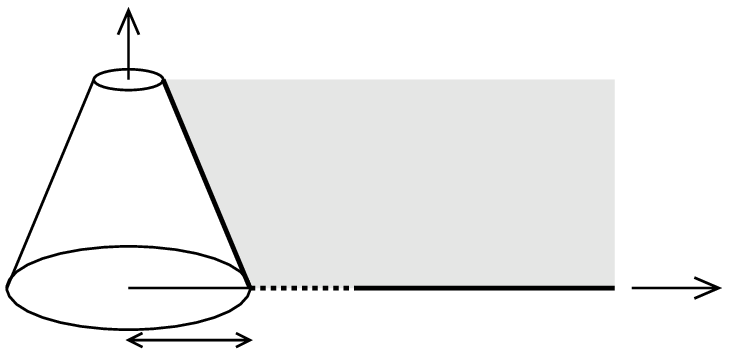}}
\put(48,78){$(t_1,s_1)$}
\put(35,-6){$s_1+\delta t_1$}
\put(73,10){cutoff}
\put(120,10){$\vb_0$}
\put(60,50){0}
\put(95,65){$\vb^{ext}$}
\put(80,50){$(\d_t-\Lop_\eta)\vb^{ext}=0$}
\put(200,8){$s$}
\put(39,98){$t$}
\end{picture}\rule[-6pt]{0pt}{0pt}
}
$$
To prove the existence and uniqueness of the solution to this exterior boundary
value problem, we first map the exterior of the cone to the exterior of a
cylinder by means of the map $s+\delta t -(s_1+\delta t_1)= \sigma$, with
$\sigma$ the new radial
coordinate, and then refer to Cor.~\ref{DBCconstant}. The estimates obtained
there are uniform in $(t_1,s_1)$ as long as $s_1$ is bounded away from
0. Indeed, with $\vb(s,\omega,t)=V(\sigma,\omega,t)$, the equation
$(\d_t-\Lop_\eta)\vb=0$ transforms into
$(\d_t-\delta\d_\sigma-\tilde\Lop_\eta)V=0$, and $\tilde\Lop_\eta$ is the same
as $\Lop_\eta$ in~\eqref{conjugated}, except with $\d_\sigma$ replacing~$\d_s$.

With $a(t):= c^\infty(\eta)t + c\int_0^t e^{-2(s_1+\delta(t_1-t))}dt$,
the function $\vbb^{ext}:= e^{-a(t)}\vb^{ext}$
satisfies a maximum principle again, and we get
\begin{equation}\label{MaxPrincEst}
|\vb^{ext}(t,\x)| = e^{a(t)} |\vbb^{ext}(t,\x)|
\le e^{a(t)} \|\vb_0\|_{L^\infty} \le C
e^{c^\infty(\eta)t}\|\vb_0\|_{L^\infty}
\end{equation}
with a constant $C=C(\delta)$ that is easily calculated to be uniform in
$(t_1,s_1)$ as long as $s_1$ is bounded away from 0.
In particular, $\|\vb^{ext}(t-1)\|_{L^\infty}\le C
e^{c^\infty(\eta)(t-1)} \|\vb_0\|_{C^\halpha}$.
Having been constructed by means of a cutoff function, $\vb^{ext}$
may not be in the appropriate spectral space. But with a regularization
estimate and successive projection, we can conclude that
\begin{equation}\label{vext-est}
  \|\Pop_\eta\vb^{ext}(t_1)\|_{C^\halpha} 
  \le  
  C \|\vb^{ext}(t_1)\|_{C^\halpha} 
  \le
  C e^{c^\infty(\eta)t_1}\|\vb_0\|_{C^\halpha}
  \;.
\end{equation}
Now take a cutoff function~$\chi\in C^\infty(\R)$ for which $\chi(\sigma)=0$ if
$\sigma\le1$ and $\chi(\sigma)=1$ if $\sigma\ge2$. Let
$$
\vb^{int}:= \vb - \Pop_\eta \bigl(\vb^{ext} \chi(s-s_1-\delta(t_1-t)) \bigr)
= \Pop_\eta (\vb-\chi\vb^{ext})
\,,
$$
where we have used in the last step that
$\vb$ is assumed to be in the range of $\Pop_\eta$.
The initial data of $\vcheck^{int} := \vb-\chi \vb^{ext}$ are compactly
supported.
So we can conjugate, letting
\begin{equation} \label{wintdef}
\vcheckc^{int}:= (\cosh s)^{\eta-\eta_{cr}} \vcheck^{int}
\quad\mbox{ and }\quad
w^{int}=\Pop_{\eta_{cr}}\check w^{int} = (\cosh s)^{\eta-\eta_{cr}} \vb^{int}
\;.
\end{equation}
We have
\def\rhs{h}
\begin{equation}\label{rhscalc}
\kern-1em\begin{array}{r@{}l}
(\d_t-\Lop_{\eta_{cr}}) \vcheckc^{int} &\mbox{}
= (\cosh s)^{\eta-\eta_{cr}} (\d_t-\Lop_{\eta}) (\vb-\chi \vb^{ext})
\\ &\mbox{}
= -(\cosh s)^{\eta-\eta_{cr}} (\d_t-\Lop_{\eta}) (\chi\vb^{ext})
\\ &\mbox{}
= -(\cosh s)^{\eta-\eta_{cr}}  [\d_t-\Lop_{\eta}, \chi] \vb^{ext}
\\ &\mbox{}
= (\cosh s)^{\eta-\eta_{cr}}
  \bigl( \hat{\chi}(\sigma) \vb^{ext}
        +2\chi'(\sigma) \d_s\vb^{ext}\bigr) =: f+\d_s g =:\rhs
\end{array}\kern-1em
\end{equation}
where
$$
\begin{array}{l}\Dst
g:= 2(\cosh s)^{\eta-\eta_{cr}}\chi'(\sigma)\vb^{ext}
\quad\mbox{ and }\quad
\\ \Dst
f:= (\cosh s)^{\eta-\eta_{cr}}\hat{\chi}(\sigma)\vb^{ext}
- 2\d_s \bigl( \chi'(\sigma)\,(\cosh s)^{\eta-\eta_{cr}}\bigr)\,
\vb^{ext}
\end{array}
$$
and, using \eqref{conjugated} and \eqref{binfty},
$\hat{\chi}=(\frac{2(n-1)}{\sinh 2s} - b^\infty(\eta)\tanh s)\chi'+\chi'' \in
C_0^\infty`]1,2`[$, and $\sigma:=s-s_1-\delta(t_1-t)$.

We'd like to argue that
$$
\begin{array}{l}
\vcheckc^{int}(t)
=
\Sop_{\eta_{cr}}(t)\vcheckc^{int}(0)
+ \int_0^t \Sop_{\eta_{cr}}(t-\tau) \rhs(\tau)\, d\tau
\,,
\\
\vc^{int}(t)
=
\Sop_{\eta_{cr}}(t) \vc^{int}(0)
+ \int_0^t \Sop_{\eta_{cr}}(t-\tau) \Pop_{\eta_{cr}}\rhs(\tau)\, d\tau
\,.
\end{array}
$$
However, we have proved insufficient regularity for $\rhs(\tau)$ to be an
initial value for the semigroup estimate. While this could be fixed by stronger
regularity estimates, we can more easily average over estimates for
initial times
$t_0\in[0,1]$ and use Thm~\ref{linearth} directly. At a formal level the
argument is written as follows: for $t\ge1$ and all $t_0\in[0,1]$, we have
\begin{equation}\label{vdecomp}
\vcheckc^{int}(t) = \Sop_{\eta_{cr}}(t-t_0)\vcheckc^{int}(t_0) +
\int_0^{t-t_0} \Sop_{\eta_{cr}}(t-t_0-\tau) \rhs(t_0+\tau)\,d\tau
\;;
\end{equation}
so by averaging over $t_0$, we have
\begin{equation}\label{wint-decompose}
\begin{array}{r@{}l}
\vcheckc^{int}(t)
= \mbox{} &
\int_0^1 \Sop_{\eta_{cr}}(t-t_0)\vcheckc^{int}(t_0) \, dt_0
\\[1ex] & \mbox{}
+ \int_0^1 \int_0^{t-1}
      \Sop_{\eta_{cr}}(t-t_0-\tau) \rhs(t_0+\tau)\,d\tau\, dt_0
\\[1ex] & \mbox{}
+ \int_0^1 \int_{t-1}^{t-t_0}
      \Sop_{\eta_{cr}}(t-t_0-\tau) \rhs(t_0+\tau)\,d\tau\, dt_0
\\[1ex]
\mbox{} = \mbox{} &
\Sop_{\eta_{cr}}(t-1) \int_0^1 \Sop_{\eta_{cr}}(1-t_0)\vcheckc^{int}(t_0) \, dt_0
\\[1ex] & \mbox{}
+ \int_0^{t-1} \Sop_{\eta_{cr}}(t-1-\tau) \Bigl(
       \int_0^1 \Sop_{\eta_{cr}}(1-t_0) \rhs(\tau+t_0)\,dt_0 \Bigr) \,d\tau
\\[1ex] & \mbox{}
+ \int_0^1 \Bigl(\int_{0}^{1-t_0} \Sop_{\eta_{cr}}(1-t_0-\tau')
        \rhs(t-(1-t_0)+\tau')\,d\tau' \Bigr)\, dt_0
\\[1ex]
\mbox{} =: \mbox{} &
\Sop_{\eta_{cr}}(t-1) \Vc_0
\\[1ex] & \mbox{}
+ \int_0^{t-1} \Sop_{\eta_{cr}}(t-1-\tau) \Vc_1(\tau) \,d\tau
\\[1ex] & \mbox{}
+ \int_0^1 \Vc_2(1-t_0) \, dt_0
\;.
\end{array}
\end{equation}
The fact that the integrals in~\eqref{wint-decompose} still manifestly apply
the semigroup to a potentially singular distribution~$\rhs$ is of no concern
for the argument. The $\Vc_{0,1,2}$ terms are simply solutions of the
inhomogeneous equation with homogeneous initial data, and they can be
estimated by  Theorem~\ref{linearth}.

$\Vc_0$ is the solution at time~1
of $(\d_t-\Lop_{\eta_{cr}})\Vc= \vcheckc^{int}$
with initial data~0, and so we  have, using \eqref{rhscalc},
\eqref{MaxPrincEst} along with Thm.~\ref{linearth},
$$
\begin{array}{l}
\|\Vc_0\|_{C^\halpha(\M)} \le C \|\vcheckc^{int}\|_{L^\infty([0,1]\times\M)}
\;,
\\[1ex]
\|\vcheckc^{int}\|_{C^\halpha([0,1]\times\M)} \le
C (\|\vcheckc^{int}(0)\|_{C^\halpha(\M)} + \|f,g\|_{L^\infty([0,1]\times\M)})
\\
\phantom{\|\check w\|_{C^\halpha([0,1]\times\M)}}
\le C e^{(s_1+\delta t_1) (\eta-\eta_{cr})} (\|\vb_0\|_{C^\halpha(\M)} +
\|\vb^{ext}\|_{L^\infty([0,1]\times\M)})
\\
\phantom{\|\check w\|_{C^\halpha([0,1]\times\M)}}
\le C e^{(s_1+\delta t_1) (\eta-\eta_{cr})} \|\vb_0\|_{C^\halpha(\M)}
\;.
\end{array}
$$

$\Vc_1(\tau)$ is the solution at time 1 of $(\d_t-\Lop_{\eta_{cr}})\Vc=
\rhs(\tau+\cdot)$ with initial data~0. We have
$$
\begin{array}{l}
\|\Vc_1(\tau)\|_{C^\halpha(\M)} \le C \|f,g\|_{L^\infty([\tau,\tau+1]\times\M)}
\\
\phantom{\|W_1(\tau)\|_{C^\halpha(\M)}}
\le C e^{(s_1+\delta(t_1-\tau))(\eta-\eta_{cr})} \,
\|\vb^{ext}\|_{L^\infty([\tau,\tau+1]\times\M)}
\;.
\end{array}
$$

$\Vc_2(1-t_0)$ is the solution at time $1-t_0\le1$
of $(\d_t-\Lop_{\eta_{cr}})\Vc = \rhs(t-(1-t_0)+\cdot)$ with initial
data~0. We have
$$
\begin{array}{l}
\|\Vc_2(1-t_0)\|_{C^\halpha(\M)}
\le C \|f,g\|_{L^\infty([t-(1-t_0),t]\times\M)}
\\
\phantom{\|W_2(1-t_0)\|_{C^\halpha(\M)}}
\le C e^{(s_1+\delta(t_1-t))(\eta-\eta_{cr})} \,
\|\vb^{ext}\|_{L^\infty([t-1,t]\times\M)}
\;.
\end{array}
$$
With these estimates established, we can now apply the spectral projection and
conclude
$$
\vc^{int}(t) =\Sop_{\eta_{cr}}(t-1)\Pop_{\eta_{cr}} \Vc_0 + \int_0^{t-1}
\Sop_{\eta_{cr}}(t-1-\tau)\Pop_{\eta_{cr}} \Vc_1(\tau)\,d\tau +
\int_0^1 \Pop_{\eta_{cr}} \Vc_2(1-t_0)\,dt_0
\,,
$$
hence
$$
\begin{array}{r@{}l}
\|\vc^{int}(t)\|_{C^\halpha(\M)} \le \mbox{} &
C e^{t(c^\infty(\eta)-2\eps)}\| \Vc_0 \|_{C^\halpha(\M)}
\\[1ex]& \mbox{}
+ C \int_0^{t-1} e^{(t-1-\tau)(c^\infty(\eta)-2\eps)}
  \|\Vc_1(\tau)\|_{C^\halpha(\M)}
\\[1ex]& \mbox{}
+ C \sup_{t_0\in[0,1]} \| \Vc_2(1-t_0) \|_{C^\halpha(\M)}
\;.
\end{array}
$$
The first and third term can immediately be estimated as
$C\exp[t(c^\infty(\eta)-2\eps)+(s_1+\delta t_1)(\eta-\eta_{cr})]
\|\vb_0\|_{C^\halpha(\M)}$  and
$C\exp[(s_1+\delta(t_1-t))(\eta-\eta_{cr})+c^\infty(\eta) t]
\|\vb\|_{L^\infty(\M)}$ respectively, and for $t=t_1$, they are controlled
(with one $\eps$ to spare in the first term)
by $C\exp[t_1 c^\infty(\eta) + s_1(\eta-\eta_{cr}) ]
\|\vb_0\|_{C^\halpha(\M)}$.
The second term is estimated similarly to be
$$
\le C \int_0^{t-1} \exp[(c^\infty(\eta)-2\eps)(t-1-\tau) +
  (s_1 +\delta(t_1-\tau)) (\eta-\eta_{cr}) + c^\infty(\eta)\tau]\, d\tau
\|\vb_0\|_{C^\halpha(\M)}
\,,
$$
which, for $t=t_1$,  is $\le C\exp[t_1 c^\infty(\eta) + s_1(\eta-\eta_{cr})]
\int_0^{t_1-1} e^{-\eps(t_1-1-\tau)}\,d\tau\, \|\vb_0\|_{C^\halpha(\M)}$, and
the integral of order $1/\eps$ can be absorbed in the constant.

So we have proved
$$
\|\vc^{int}(t_1)\|_{C^\halpha(\M)}
\le
C e^{t_1 c^\infty(\eta) + s_1(\eta-\eta_{cr})}
\|\vb_0\|_{C^\halpha(\M)}
\;;
$$
and in particular
$|\vc^{int}(t_1,\x_1)| \le C e^{t_1 c^\infty(\eta) + s_1(\eta-\eta_{cr})}$.
With \eqref{wintdef}, this implies
$|\vb^{int}(t_1,\x_1)| \le C e^{t_1 c^\infty(\eta)}$. The constant is uniform
in $(s_1,t_1)$, as long as $s_1$ is bounded away from 0, say $s_1\ge1$. But for
$s<1$, the estimate with $s_1=1$ still gives $|\vb^{int}(t_1,\x)| \le C e^{t_1
  c^\infty(\eta)}$.
We review \eqref{MaxPrincEst}, which guaranteed
$\|\vb^{ext}(t_1)\|_{L^\infty}\le C e^{c^\infty(\eta)t_1}\|\vb_0\|_{L^\infty}$,
and since the projection operator is bounded from $L^\infty$ into itself, we
also have
$\|(\Pop_{\eta}(\chi\,\vb^{ext}))(t_1)\|_{L^\infty}\le
C e^{c^\infty(\eta)t_1}\|\vb_0\|_{L^\infty}$. In view of
$\vb(t_1,\x_1)= \vb^{int}(t_1,\x_1) +
(\Pop_{\eta}(\chi \vb^{ext}))(t_1,\x_1)$, which we have from \eqref{vdecomp}, we
conclude
$$
|\vb(t_1,\x_1)| \le C e^{c^\infty(\eta)t_1}\|\vb_0\|_{C^\halpha}
$$
uniformly in $(\x_1,t_1)$.  This establishes
$\|\vb(t)\|_{L^\infty}\le C e^{c^\infty(\eta)t}\|\vb_0\|_{C^\halpha}$; and
with another invocation of regularization, we can get the same estimate for the
$C^\halpha$ norm on the left hand side.
\end{proof}

\section{Proof of Theorem \ref{t-1}}
\label{pthm1}

Let $\rho_0$ be an initial density which satisfies the assumptions of Theorem
\ref{t-1}.
After scaling we may assume that its mass agrees with that of $\uB(\x)$
for $B=1$.  We shall actually prove the asserted decay \eqref{te-2} in a norm
$\|\,\cdot\,\|_{C^\halpha(\M)}$ stronger than $\|\,\cdot\,\|_{L^\infty(\Rn)}$.
As long as $m>m_1=\frac{n-1}{n+1}$, the center of mass of $\rho_0$ is defined.
We may shift $\rho_0$, such as to bring its center of mass to 0.
We then claim that the statement of the theorem holds in particular
by choosing $\z=0$.
The hypothesis $m>m_1$ is automatically satisfied for $n=1$, and also whenever
$m>m_n$. In the other cases, we either cannot define a center of mass or else
will not bother to shift it, and we claim that in those cases, the theorem holds
with any~$\z$.
For quick reference, here are the relevant eigenvalues of~$\Lop$ in~$C^\halpha$
(unweighted) in the case $n\ge2$:
$$
\mbox{%
  \begin{picture}(120,35)
     \put(0,-17){\includegraphics{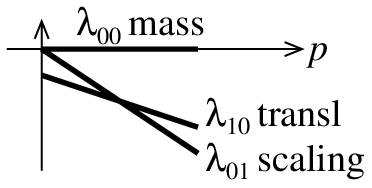}}
  \end{picture}%
}
\qquad
\mbox{eigenvalues: }
\begin{array}{l}
\lambda_{00}= 0\\
\lambda_{10}= -(p+n)\\
\lambda_{01}= -2p
\end{array}
$$
In the case $n=1$ (where $m_n=0$), the picture looks like this:
$$
\mbox{%
  \begin{picture}(120,35)
    \put(0,-17){\includegraphics{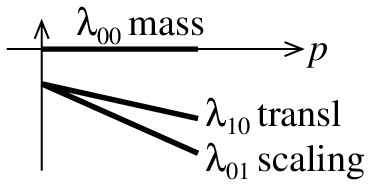}}
  \end{picture}%
}
\qquad
\parbox{7em}{eigenvalues\\when $n=1$: }
\begin{array}{l}
\lambda_{00}= 0\\
\lambda_{10}= -(p+1)\\
\lambda_{01}= -2p
\end{array}
$$
In either case, the onset of the ($\ell=0$ layer of the) essential spectrum
is right at~$\lambda_{01}$.

The evolution
conserves mass and center of mass. In the variables of eqn.~\eqref{cyl}, these
conservation laws amount to $\int(v-1)\uB\,d\x=0$ and $\int(v-1)\uB\,\x\,d\x=0$.
In other words,  using the inner product of
$L^2_{\uB^m}$ makes $v-1$ orthogonal to the eigenfunction $v_{00}=\uB^{1-m}$
and, if $m>m_1$, to $v_{10}=r\uB^{1-m}Y_{1\mu}(\omega)$ for $\mu =1,\ldots,n$.

By \eqref{scalingtrf}, the statement of the theorem, with $\z=0$, is equivalent
to
$$
\sup e^{-\lambda_{01} t} |v(\x,t)-1| < \infty
\;,
$$
recalling $\lambda_{01} = -2p$.
Let us rewrite \eqref{cyl}--\eqref{linlop} as
\begin{equation}\label{cyl-L-NL}
\begin{split}
(\d_t - \Lop) (v-1) = &\
\left(   \frac{\d^2}{\d s^2}  + \frac{2(n-1)}{\sinh 2s} \frac{\d}{\d s}
 +  (\tanh s)^{-2} \Laplace_{\Sph}  \right)  \Bigl(\frac{v^m}{m} - v\Bigr)
\\ & +
    \biggl( n - \frac{2(m+1)}{1-m}\biggr) \tanh s  \frac{\d}{\d s}
              \Bigl(\frac{v^m}{m} - v\Bigr)
\\  &
+ \frac{4}{(1-m)^2}
   \Bigl[  n\frac{1-m}{2}- 1 + (\cosh s)^{-2}\Bigr]
(mv-v^m+1-m)
\;.
\end{split}
\end{equation}

With the trivial substitution $w=v-1$,
its structure becomes clearer in the form
\begin{equation}\label{lineq-w}
(\d_t - \Lop)w = \tilde \Lop ( f(w) w )
\end{equation}
where $f(w)=((1+w)^m-1-mw)/mw$ is analytic and vanishes at  the origin,
and where $\Lopmod=\Lop - \frac{2}{1-m} (\tanh s\,\d_s + n -
\frac{2}{1-m}\tanh^2s)$~can be
read off the right side of~\eqref{cyl-L-NL}. We can apply Thm.~\ref{t-3}.
Let us write $w_j:= w(jT)\in C^\halpha(\M)$, for a  constant~$T$ chosen
sufficiently small for the contraction estimates of Chapter~\ref{SecGlobWP},
specifically from Lemma~\ref{quadratic-error}, to apply.
Also let
$\Sop(t):= \exp[t\Lop]$ be the semigroup generated by the linear
operator~$\Lop$.  We write
\begin{equation}\label{recurs}
w_{j+1} = \Sop(T) w_j + g(w_j) 
\;,
\end{equation}
where
$$
C^{\halpha}(\M) \ni w \longmapsto g(w) \in C^{\halpha}(\M)
$$
is smooth and
vanishes at the origin together with its derivative.
More specifically, since the map $w_j\mapsto w_{j+1}=:F(w_j)$ is smooth
according to Thm.~\ref{t-3}, and $F(0)=0$, $DF(0)=\Sop(T)$, we can write
$w_{j+1}= \int_0^1\frac{d}{d\sigma} F(\sigma w_j)\,d\sigma = \Sop(T)w_j +
[ \int_0^1 (DF(\sigma w_j)-\Sop(T))\,d\sigma ] w_j =: \Sop(T) w_j + G(w_j)
w_j$.
Here $G$ is  smooth with values $G(w)$ being bounded linear maps
$C^\halpha\to C^\halpha$, and $G(0)=0$. This gives the claim.

In the case $m>m_n$, let $\Pop$ be the $\Sop$-invariant projection
to the $L^2_{\uB^m}(\Rn)$ complement of the eigenspaces
for $\lambda_{00}=0$ and $\lambda_{10}= -(p+n)$. In the case $m\le m_n$,
let $\Pop$ be the $\Sop$-invariant projection in $C^\halpha(\M)$ to the
complement of the eigenspace for $\lambda_{00}=0$. It is well-defined according
to Theorem~\ref{semigroupestimates}.
We have seen that the $w_j$ are in
the range of $\Pop$, and therefore
$\|\Sop(t)w_j\|\le C e^{\lambda_{01}t}\|w_j\|$ with a constant $C$ independent
of~$t$, where $\|\cdot\|$ refers to $\|\cdot\|_{C^\halpha(\M)}$.

We introduce a small quantity~$r<1$ and assume $\|w_0\|$ is less than $r$
divided by a constant supplied from Lemma~\ref{comp}. This lemma
then ensures that $\|w_j\|_{L^\infty}<r$ for all~$r$.
Lemma~\ref{quadratic-error} applies to our operators $\Lop$, $\Lopmod$ and
nonlinearity $f(w)w$, and it controls
$\|g(w)\|_{C^\halpha} \le K\|w\|_{C^\halpha} \,\|w\|_{L^\infty}$ for
$\|w\|_{L^\infty}<r$. Now $K$ stays
fixed, while we still may impose further smallness requirements on~$r$.
We denote by $\lambda<0$ whatever spectral gap applies to the parameter~$m$
under consideration.
Iterating \eqref{recurs} gives
\begin{equation}\label{recurs-iter}
\begin{array}{l}
w_k = \Sop(T)^k w_0 + \sum_{j=1}^k \Sop(T)^{k-j} g(w_{j-1})
\\
\|w_k\|_{C^\halpha}
\le Ce^{\lambda kT}\|w_0\|_{C^\halpha} + \sum_{j=1}^k
CKe^{\lambda(k-j)T}\|w_{j-1}\|_{C^\halpha}\,\|w_{j-1}\|_{L^\infty}
\;.
\end{array}
\end{equation}
Here, $C$ is the constant obtained from Thm.~\ref{semigroupestimates}. We know
that
$$
\|w_{j-1}\|_{L^\infty} \le \min\{ r, \|w_{j-1}\|_{C^\halpha}\}
\le r^{1/2} \|w_{j-1}\|_{C^\halpha}^{1/2}
$$
and will prove inductively that
$\|w_k\|_{C^\halpha}\le 2Ce^{\lambda kT}\|w_0\|_{C^\halpha}$. The start being
trivial, the induction step in~\eqref{recurs-iter} follows if we can ascertain
that
$$
1+\sum_{j=1}^k K e^{-\lambda jT} r^{1/2} (2C)^{3/2} e^{3\lambda(j-1)T/2}
\le 2 \;.
$$
By convergence of the geometric series $\sum^\infty e^{\lambda jT/2}$, this
requirement can indeed be met by choosing $r$ sufficiently small.

Apart from the adjustment of $\Pop$, the proof works the same in the
full range of $m$ (with $\lambda+\eps$ in case $m=m_2$, per
Thm.~\ref{semigroupestimates}).

Finally, let's deal with the case $m=m_2$ under the moment
hypothesis~\eqref{te-1}, which, in view of \eqref{cigarmetric} and
\eqref{cigarvol}, amounts to $\int_\M w_0(y)^2\,d\mu < \infty$. 
(Recall $d\mu$ denotes the Riemannian volume element on~$\M$.)
We will argue shortly that 
\begin{equation}\label{wL2decay}
\|w(t)\|_{L^2(\M)} \le C e^{\lambda t} \|w_0\|_{L^2(\M)}
\;.
\end{equation}
We also note that $\|\Sop(T)w\|_{C^\halpha(\M)} \le C \|w\|_{L^2(\M)}$
by regularization: indeed, this is a regularization estimate for a
uniformly parabolic linear PDE in each coordinate chart~$U_l$
separately, with the $\|w\|_{L^2(U_l)}$ all being trivially estimated
by the global $\|w\|_{L^2(\M)}$; recall that the 
distortion between the metric on the $U_l$ and the euclidean metric
is uniformly bounded. 

From these two ingredients we conclude, in view of the close-to-sharp 
$O(e^{(\lambda+\eps)t})$ decay already established, that
$$
\|w(t+T)\|_{C^\halpha} 
\le 
\|\Sop(T) w(t)\|_{C^\halpha} + \|g(w(t))\|_{C^\halpha} 
\le
C e^{\lambda t} \|w_0\|_{L^2(\M)} + C e^{2(\lambda+\eps)t}
\,,
$$
and the claim follows since $2(\lambda+\eps)<\lambda$.
So all we have to establish is \eqref{wL2decay}, and we do it by direct
calculation from the PDE. 

Letting $h(w) = \frac{(1+w)^m-1}{m}=w+O(w^2)$ and
$H(w)=\frac12w^2+O(w^3)$ its antiderivative, we use  
$E(t):= \int_{\M} H(w(t))\,d\mu$ as a proxy for the $L^2$-norm. 
For the technical reason of properly caring about a flux term at infinity,
we will also need $E(r,t):= \int_{B(r)}
H(w(t))\,d\mu$, with integration over the euclidean $r$-ball around
the origin (equivalently the geodesic $\arsinh r$-ball). 
Formally, we calculate
\begin{equation}\label{formalEderiv}
\kern-2em 
\begin{array}{l}\Dst
\d_t E(t) = \int_\M h(w) w_t\,d\mu = \mbox{}
\\[2ex] \Dst \kern3em = 
\int_\M h(w)\Lop h(w)\, d\mu + 
 \frac{2}{1-m} \int_\M h(w) (\x\cdot\nabla)(w-h(w))\,d\mu + \mbox{}
\\[1.5ex] \Dst\kern4em 
+ \frac{2}{1-m} \int_\M 
        \Bigr( n - \frac{2\uB^{1-m}}{1-m}|\x|^2
          \Bigl) (w-h(w))\,h(w) \,d\mu 
\;.
\end{array}
\kern-3em
\end{equation}
Returning to the euclidean volume element 
$d^n\x = (1+|\x|^2)^{n/2}d\mu$, we want to integrate the second term
of this sum by parts to avoid derivatives of $w$. 
With $h(w)\nabla(w-h(w)) = \nabla (H(w)-\frac12h(w)^2)$, we obtain for
$n\neq2$ (with routine modifications in case $n=2$)
$$
\begin{array}{l}\Dst
\d_t E(r,t) = \int_{B(r)} h(w) w_t\,d\mu 
= \int_{B(r)} h(w)\Lop h(w)\, d\mu +  \mbox{}
\\[2ex] \Dst \kern3em +
 \frac{2}{(1-m)(2-n)} \int_{B(r)} \Div \Bigl(
 (H(w)-{\Tst\frac12}h(w)^2) \,
   \nabla(1+|\x|^2)^{-n/2+1} \Bigr) \,d^n\x 
\\[2ex] \Dst \kern3em 
-  \frac{2}{(1-m)(2-n)}
     \int_{B(r)}  (H(w)-{\Tst\frac12}h(w)^2) \,
        \Laplace (1+|\x|^2)^{-n/2+1}\,d^n\x 
\\[2ex] \Dst \kern3em 
+ \frac{2}{1-m} \int_{B(r)} 
       \Bigr( n - \frac{2\uB^{1-m}}{1-m}|\x|^2
          \Bigl) (w-h(w))\,h(w) \,d\mu 
\;.
\end{array}
$$
The last integrand is estimated as $O(w^3)$, or $O(wH(w))=
e^{(\lambda+\eps)t} O(H(w))$. The second last term is integrable even
over $\M$, or can be estimated in the same manner as the last term. 
We obtain (with $dS$ the {\em Riemannian\/} surface element)
$$
\d_t E(r,t) \le \int_{B(r)} h\Lop h\, d\mu + C_1 e^{(\lambda+\eps)t}
\int_{\boundary B(r)}|w(t)|^2\, dS + C_2 e^{(\lambda+\eps)t} E(r,t)
\;.
$$
Letting $\gamma(t):= - \int_0^t C_2 e^{(\lambda+\eps)t}\,dt$, we get 
\begin{equation}\label{energygrowth}
\begin{array}{l}\Dst
\d_t (e^{\gamma(t)}E(r,t)) \le 
e^{\gamma(t)} \int_{B(r)} h(w(t))\Lop h(w(t))\,d\mu + 
\mbox{} 
\\[1.5ex]\Dst\kern7.5em +
C_1 e^{\gamma(t)+(\lambda+\eps)t}\int_{\boundary B(r)}|w(t)|^2\, dS
\end{array}
\end{equation}
and therefore
$$
e^{\gamma(t)} E(r,t) - E(r,0) \le \int_0^t \int_{B(r)} h(w(t))\Lop
h(w(t))\,d\mu\, dt + C\|w(t)\|_{L^\infty}^2
\;.
$$
Now from $\|w_0\|_{L^2(\M)}^2<\infty$, we get $E(r,0)\to
E(0)<\infty$  as $r\to\infty$. If $w(t)\in L^2$, we conclude 
$\int_{B(r)} h(w(t))\Lop h(w(t))\,d\mu \to \int_\M h(w(t))\Lop
h(w(t))\,d\mu \le \lambda h(w(t))^2\,d\mu \le0$. Without {\em assuming\/} 
$w(t)\in L^2$, we still maintain the inequality 
$\limsup_{r\to\infty} \int_{B(r)} h(w(t))\Lop h(w(t))\,d\mu \le0$ by
approximation, and this guarantees the finiteness of
$\lim_{r\to\infty} e^{\gamma(t)}E(r,t)$, and therefore $w(t)\in
L^2(\M)$ for every~$t$. 

Knowing this, we can actually justify the formal calculation
\eqref{formalEderiv} and proceed with the same estimates and
integration by parts done for $E(r,t)$ to obtain 
$$
\d_t E(t) \le \lambda \int_\M h(w(t))\Lop h(w(t))\, d\mu
+ C e^{(\lambda+\eps)t} E(t)
\le 
(2\lambda + C e^{(\lambda+\eps)t})E(t)
\;.
$$
The estimate 
$E(t)\le C e^{2\lambda t} E(0)$ follows from this; and then
immediately the same estimate for $\|w\|_{L^2(\M)}^2$, hence
\eqref{wL2decay}.

\section{Asymptotic estimates in weighted spaces: The case $m< \frac{n}{n+2}$}
\label{SecWeights:p<2}

In the previous section, we have seen in particular that for
$m<m_2=\frac{n}{n+2}$, we get a convergence rate of $O(1/\tau)$, and no better,
to the appropriately shifted Barenblatt with respect to the relative $L^\infty$
norm. This is the same rate as was obtained by Kim and McCann 
\cite{MR2246356} via Newton potentials, 
and later extended to a wider range of~$m$ by  the quintet \cite{BBDGV09}
at the cost of an extra $\eps>0$ in the rate exponent. 
The slightly larger spectral gap in the weighted Hilbert
space of~\cite{MR2126633} (where $\lambda_{01}$ was not an eigenvalue any more
for $m<m_2$ and the continuous spectrum started further below) was not actually
expected to be dynamically effective from that setting, because the
linearization formalism relied on the existence of second moments, in contrast
to the reality in the case $m<m_2$.

In the present setting, the absence of a weight in $C^\halpha$ has changed the
spectrum in comparison to~\cite{MR2126633}. To recover that spectrum, we need
to study $\Lop$ in $\Cetacr^\halpha$, or equivalently, the operator
$\Lop_{\eta_{cr}}$ in~$C^\halpha$. Since $\eta_{cr}=\frac{p}{2}-1<0$
for $m<m_2$, this means we need to study the case where initial data
deviate from Barenblatt less than what would automatically be achieved
by V\'azquez' result \eqref{reluni}. 
For initial data with this special tail behavior, we
can indeed get improved convergence rates from a study of the conjugated
operator. These better convergence rates will apply to {\em appropriately
  weighted\/} norms.

We are careful to point out that the {\em nonlinear\/} evolution was
constructed for unweighted norms only: this is in accordance with the fact
that, in the relative $L^\infty$ setting, unweighted norms keep a uniform
distance from the singularity $u=0$ of the Fast Diffusion Equation~\eqref{pm}
and allow, on a technical level, uniformly parabolic estimates on the whole
space. (However, in the present case $m<m_2$, Thm.~\ref{t-3} does carry over,
as noted near the end of Section~\ref{SecGlobWP}.)
In contrast, the linearized flow can be studied with the same ease in
weighted and unweighted spaces.

Let's have a quick look at the spectrum of $\Lop_{\eta_{cr}}$ in this range of
parameters, where the only eigenvalues are $\lambda_{00}=0$ and
$\lambda_{10}= -(p+n)$, and the onset of the essential spectrum
is~$\lambda_0^{\rm cont} = -(\frac{p}{2}+1)^2$:

In the case $n=1$, one has $m_1=0$, so Barenblatt always has 1st moments and
we can shift the
center of mass. We also need to do this because $\lambda_{10}$ is above the
onset of the essential spectrum.
$$
\mbox{%
  \begin{picture}(120,35)
      \put(0,-17){\includegraphics{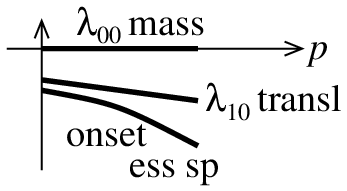}}
  \end{picture}}
\quad
\parbox{10em}{eigenvalues of\\ $\Lop_{\eta_{cr}}$ when $n=1$: }
\begin{array}{l}
\lambda_{00}=0\\
\lambda_{10}=-(p+1)\\
\lambda_{0}^{\rm cont}=-(\frac{p}{2}+1)^2
\end{array}
$$

In the case $n\ge2$ the center of mass may or may not be defined (depending on
whether $m>m_1$ or $m\le m_1$), but there is no need to shift the center of
mass, because the
intersection of $-(\frac{p}{2}+1)^2$ with $\lambda_{10}=-(p+n)$
occurs at $p=2\sqrt{n-1}\ge2$, hence
$\lambda_{10}$ is below the onset of the essential spectrum (ranging from
$-1$ to $-4$ monotonically)  in the whole
parameter range $m_0<m<m_2$.

As we have resolved to work in the space $\Cetacr^\halpha$, we let
$\tilde{w}=(\cosh s)^{-\eta_{cr}}w$, and now equation~\eqref{lineq-w} becomes
\begin{equation}\label{lineq-w-conj<}
(\d_t - \Lop_{\eta_{cr}})\tilde{w} =
\tilde \Lop_{\eta_{cr}}
     \bigl( f(1+(\cosh s)^{\eta_{cr}}\tilde{w}) \tilde{w} \bigr)
\end{equation}
Since $\eta_{cr}<0$, we can estimate
$\|(\cosh s)^{\eta_{cr}}\tilde{w}\|_{C^\halpha(\M)}\le c \|\tilde{w}\|_{C^\halpha(\M)}$,
and this is why Theorem~\ref{t-3} carries over
to this case along with its proof, as remarked at \eqref{Cetadef0}.
We now can repeat the proof from
Section~\ref{pthm1} almost verbatim.
We conclude: If the initial data $w_0$ satisfy
\begin{equation}\label{fine-hypothesis}
(\cosh s)^{-\eta_{cr}} w_0 \in L^\infty(\M)   \;,
\end{equation}
then
\begin{equation}\label{fine-conclusion}
\sup_{t>1} e^{(-\lambda_0^{\rm cont}-\eps)t}
\|(\cosh s)^{-\eta_{cr}}w(t)\|_{L^\infty} <\infty
\end{equation}
for every $\eps>0$. (The same statement works also for the $C^\halpha$-norm.)
Here
$-\lambda_0^{\rm cont} = (\frac{p}{2}+1)^2= 
       (\frac{1}{1-m} + 1 - \frac{n}{2})^2
$.
If we put $L^\infty\cap L^2$ in the hypothesis
\eqref{fine-hypothesis}, we get conclusion \eqref{fine-conclusion}
without the~$\eps$.

We therefore get the following:
\begin{theorem}[Fine Asymptotics for $m<m_2$, with restricted initial data]
\label{fine:p<2} \  \\
Assume $n\ge2$ and $m \in `]m_0,m_2`[=`]\frac{n-2}{n},\frac{n}{n+2}`[$, or else
$n=1$ and $0<m<m_2=\frac13$. Further assume that the mass of $\rho_0$ is
one, and if $n=1$ also that the center of mass of $\rho_0$ is 0.
Let $\delta:= \frac{p/2-1}{p+n} = \frac{m}{2}-\frac{n(1-m)}{4}$ (note
$\delta<0$).
If $\rhoB(0,\cdot)^\delta|1-\rho_0/\rhoB(0,\cdot)|\in L^\infty$,
then, for every $\eps>0$
\begin{equation} \label{bestdecay}
\limsup_{\tau\to\infty} \sup_{\y} \tau^{\gamma-\eps}
\left(\frac{\rhoB(\tau,\y)}{\rhoB(\tau,0)}\right)^{\delta}
\left| \frac{\rho(\tau,\y)}{\rhoB(\tau,\y)} -1 \right|
< \infty
\end{equation}
where
$\gamma =  \frac{1}{2p}(\frac{p}{2}+1)^2$.
Likewise, we get the same conclusion without the $\eps$, provided we
add the hypothesis that 
$\rhoB(0,\cdot)^\delta|1-\rho_0/\rhoB(0,\cdot)|\in L^\infty$ is also
square integrable with respect to the measure $(1+|\x|^2)^{-n/2}\,d^n\x$.
\end{theorem}
A graph for $\gamma,\delta$ in dependence on $m$ can be found in
Figure~\ref{graph:t-2}.
The tremendous improvement over the
$O(1/\tau)$ convergence rate established by Kim and McCann~\cite{MR2246356} and in our
present Thm.~\ref{t-1} depends on the strong decay assumptions for the
relative deviation of the initial data.

Also note that our Theorem \ref{fine:p<2} is a generalization of the comparison
estimates by V\'azquez (of which we use a modified form as Lemma~\ref{comp}).
Namely this comparison asserted that if a solution is close to Barenblatt in
relative $L^\infty$ norm, then it stays close to Barenblatt in this norm.
Here we have shown the same in a stronger norm, showing that proximity to
Barenblatt in a norm that enforces an even closer fit with Barenblatt in the
tails, is also preserved under the dynamics.

\section{Higher asymptotics in weighted spaces: The case $m>
  \frac{n}{n+2}$. Proof of Theorem~\ref{t-subquadratic} and its corollaries.}
\label{SecWeights:p>2}

This section is devoted to proving Theorems \ref{t-subquadratic} and
\ref{t-superquadratic}.
When applicable, the latter 
extends the conclusions of the former to decay rates
$\Lambda$ outside the range $`] 2\lambda_{01},\lambda_{01}]$.  For an
example of its applicability, see Remark \ref{even higher asymptotics}.
The restriction $\Lambda > \lambda + \lambda_{01}$ appearing in this
refinement arises from a limitation of our method, which relies on a
weighted control of the quadratic term in the nonlinearity by a
product of a weighted norm \eqref{amplification-hypothesis}
decaying at rate $\lambda$ and an unweighted norm decaying at the rate
$\lambda_{01}$ (given by Theorem \ref{t-1}).

It seems convenient to repeat Theorem~\ref{t-subquadratic} here:
%
\newcount\restoreeqn \restoreeqn=\csname c@equation\endcsname          
\newcount\restoresec \restoresec=\csname c@section\endcsname           
\newcount\restorethm \restorethm=\csname c@remark\endcsname            
\csname c@equation\endcsname=\recalleqn
\csname c@section\endcsname =\recallsec
\csname c@remark\endcsname  =\recallthm
\begin{quote} \footnotesize
\def\label#1{\relax}
\rememberthistheorem 
\end{quote}
\csname c@section\endcsname =\restoresec
\csname c@remark\endcsname  =\restorethm
\csname c@equation\endcsname=\restoreeqn

\begin{theorem}[Higher-order asymptotics in weighted H\"older spaces]
\label{t-superquadratic} \hfill
Fix $p= 2(1-m)^{-1} -n>2$ and
$\Lambda\in[\lambda_0^{\rm cont},\lambda_{01}]$, and choose
$\lambda\le\lambda_{01}$ such that $\lambda+\lambda_{01}<\Lambda$.
If $u(t,\x)$ satisfies the hypotheses of Theorem \ref{t-subquadratic},
and either $\lambda=\lambda_{01}$ or else we have the stronger
hypothesis
\begin{equation}\label{amplification-hypothesis}
\limsup_{t\to\infty} e^{-\lambda t}\left\|
\frac{u(t,\x)/\uB(\x) - 1}%
     {(B+|\x|^2)^{\frac{p}{2} -1 - \sqrt{(\frac{p}{2}+1)^2 +\Lambda} }}
\right\|_{C^\halpha(\M)}
<\infty
\,,
\end{equation}
then the conclusion \eqref{subquadratic-asymptotics}
of Theorem \ref{t-subquadratic} remains true and
$u_{\ell k}=0$ for $\lambda_{\ell k} >\lambda$.
\end{theorem}

\begin{proof}[Proof of Thm.~\ref{t-subquadratic}]
In case $\Lambda=\lambda_{01}=-2p$,
the weights cancel in \eqref{subquadratic-asymptotics},
and the result was established  in Section \ref{pthm1}
during the proof of Theorem \ref{t-1}.
The improvement to $\Lambda<\lambda_{01}$ is based on the same
principle as Section \ref{SecWeights:p<2}, except this time
$\eta_{cr} > 0$, and the situation is dual to the previous one: we have a
chance to obtain faster decay by introducing a weight which
relaxes the strength of the norm as $s \to \infty$.
We need not conjugate all the way to $\eta_{cr}$,  but instead have the
flexibility to choose $\eta \in [0,\eta_{cr}]$
to balance the error term in Theorem \ref{t-subquadratic} against
the severity of this relaxation.
Indeed, the choice
\begin{equation}\label{choose-eta}
\eta = \eta_{cr} - \sqrt{\Lambda - \lambda_0^{\rm cont}}
\end{equation}
makes the radius $\lambda_{0,\eta}^{\rm cont}$ of the essential spectrum of
$\Lop_\eta$ given by Theorem \eqref{HolderSpectrum} coincide with
$\Lambda \in [\lambda_0^{\rm cont}, \lambda_{01}]$.
Choosing $\eta>\eta_{cr}$ would only prove the statement with a weaker norm,
without giving better rates.

Conjugating with $\eta>0$ however requires a reconsideration of the proof of
Theorem~\ref{t-3}, because a~priori, uniform parabolicity cannot be expected
for weights that allow growth relative to the Barenblatt.
The redeeming feature will first be that we do retain the {\em unweighted\/}
relative $L^\infty$ hypothesis on the initial data;
then since the nonlinearity enters through a term
$f(w)w$, after conjugacy $\tilde{w}=(\cosh s)^{-\eta} w$, this term is still
$f(w)\tilde{w}$, and the unweighted estimates continue to control the
nonlinearity, whereas linear theory applies to the weighted estimates.
We now carry this out in detail:

We again study equation \eqref{lineq-w} 
for initial data $w_0\in C^\halpha$ (in particular bounded), but using weighted
$\Ceta^\halpha$ norms with the more permissive weight $\eta>0$
from~\eqref{choose-eta}. Letting
$\tilde w:= (\cosh s)^{-\eta} w$,
Eq.~\eqref{lineq-w} becomes
\begin{equation}\label{lineq-w-conj}
(\d_t-\Lop_\eta)\tilde w =\Lopmod_\eta(f(w) \tilde w)
\end{equation}
with $\Lopmod_\eta =
\Lop_\eta - \frac{2}{1-m}
(\tanh s\; \d_s + n + (\eta - \frac{2}{1-m})\tanh^2s)$
and with the same $f(w)$ as before. We can use the fact, from the
unweighted norm, that a solution $w\in C^\halpha(\M_T)$
to~\eqref{lineq-w} exists; in
particular the conjugated~$\tilde{w}(\cdot,t)$ will also be in
$C^\halpha(\MT)$ still (even in $C_{-\eta}^\halpha$).

We obtain an analog to Lemma~\ref{quadratic-error}, 
but stated specifically for the FDE on the cigar manifold, because the
a~priori estimate on relative $L^\infty$ norms from Lemma~\ref{comp}
is used: namely, we claim
\begin{lemma}[Linear approximation of nonlinear
    semiflow; weighted norm]\label{quadratic-error-weighted}
\mbox{} \kern3em \hskip0pt minus3em
Let~$\bar {\tilde w}$ solve the homogeneous linear
equation~\eqref{linear}, namely
$(\d_t-\Lop_\eta)\tilde w=0$ for initial data $\tilde w_0$, where $\Lop_\eta$
  is given
by~\eqref{conjugated} and can be written, in local coordinates, as
$\Lop_\eta \tilde w=
  \d_{ij}^2 (a^{ij} \tilde w) + \d_i (b^i \tilde w) + c \tilde w$.
Let $\tilde w$ solve the quasilinear
equation
\eqref{lineq-w-conj},
for the same initial data $\tilde w_0$,  and $w=(\cosh s)^\eta \tilde w$.
In local coordinates, we can write  $\Lopmod=\Lop + \d_i\circ \tilde b^i +
\tilde c$,
and $f$ is a smooth function from an interval about 0 into $\R$ satisfying
$f(0)=0$. The coefficients are smooth.

Then, for sufficiently short time~$T$, there exists a constant~$K$ (uniform as
$T\to0$) such that we have the estimate
\begin{equation}\label{quadratic-est1-weighted}
\|\tilde w-\bar {\tilde w}\|_{C^\halpha(\MT)} \le
K \|w\|_{L^\infty(\MT)} \|\tilde w\|_{C^\halpha(\MT)}
\end{equation}
and from it the time-step estimate (with a different $K$):
\begin{equation}\label{quadratic-est2-weighted}
\|\bar {\tilde w}(T)-\tilde w(T)\|_{C^\halpha(\M)}\le
K \|\tilde w_0\|_{C^\halpha(\M)}\, \|w_0\|_{L^\infty(\M)}
\;.
\end{equation}
\end{lemma}

\begin{proof}[Proof of the Lemma]
The proof is modeled right after the proof of
Lemma~\ref{quadratic-error}. We can basically copy equations
\eqref{HE-w-lin} and \eqref{HE-w-NL} almost verbatim,
the only changes being  that all $w$'s have tildes now, except the
ones inside $f(\cdot)$, and the insignificant fact that coefficients
implicitly depend on the conjugacy parameter~$\eta$ (not to be
confused with the partition of unity $\eta_l$). The proof ensues as
before, with the estimate
$\|f(w)\tilde w\|_{C^\halpha(\MT)} \le K \|f(w)\|_{L^\infty(\MT)}\,
\|\tilde w\|_{C^\halpha(\MT)}$ provided by
Corollary~\ref{weighted-algebra-cigar}.
\end{proof}

{\it Proof of Thm.~\ref{t-subquadratic} continued: }
As before in Ch.~\ref{pthm1}, we write $\tilde{w}_j:= \tilde{w}(jT)\in
C^\halpha(\M)$, where $T$ is sufficiently small for
Lemma~\ref{quadratic-error-weighted} to apply; and we write
$\Sop_\eta(t):=\exp(t\Lop_\eta)$ for the semigroup.
We continue to use Thm.~\ref{t-3} for the unconjugated flow and copy
from \eqref{recurs} and the paragraph following it that
$$
w_{j+1}=  \Sop(T) w_j + G(w_j)w_j
\;.
$$
After conjugating, this becomes
\begin{equation}\label{conj-timemap}
\tilde w_{j+1}=  \Sop_\eta(T) \tilde w_j + G_\eta(w_j)\tilde w_j
\,,
\end{equation}
where $G_\eta$ is a smooth function on~$C^\halpha$ with values
$G_\eta(w) = (\cosh s)^{-\eta}\circ G(w) \circ (\cosh s)^\eta$ being
bounded linear maps on $C_{-\eta}^\halpha(\M)$. The
estimate~\eqref{quadratic-est2} guarantees that $G_\eta(w_j)$ is also
a bounded linear map on $C^\halpha(\M)$ (namely with norm $\le
K\|w_j\|_{L^\infty}$).

From Theorem \ref{HolderSpectrum},  the only spectrum of $\Lop_\eta$
closer to zero than $\lambda_{0,\eta}^{\rm cont}=\Lambda$ consist of
eigenvalues \eqref{eigenvalues} of finite multiplicity and indexed by
non-negative integers $\ell,k \in \N$ such that $\ell + 2k < p/2 +1 -
|\eta - \eta_{cr}|$.
Enumerate the eigenvalues $\lambda_{\ell k}$ which lie in the range
$`]\Lambda,0]$ by $\Lambda_l \le \Lambda_{l-1}\le \cdots \le\Lambda_1$,
counting them with multiplicity.
Set $\Lambda_{l+1} = \Lambda$ by convention,
whether or not this is an eigenvalue.

Before continuing the analysis of~\eqref{remainder}, a brief remark
about the heuristics of the proof strategy seems useful:
We are not  constructing
a `slow' or `pseudocenter' manifold corresponding to the eigenvalues
$\Lambda_i$ using invariant manifold methods with respect to which one
would naturally measure convergence rates of the transversal `fast'
dynamics. Rather we assess convergence rates in terms of distances to
the linear eigenspaces using spectral projections.
This is appropriate since the spectral ratio $\Lambda/\lambda_{01}$ is
less than 2, the order of the nonlinearity. For larger spectral
ratios, the deviation of the slow
manifold from its tangential space, rather than the eigenvalues alone,
would indeed dominate the asymptotic behavior of the `fast'
dynamics transversal to the slow manifold. This is also the reason why
products of eigenfunctions and linear combinations of eigenvalues do
not occur in the statement of the theorem. (Fig.~\ref{Fig-heuristics}
explains this heuristics in a simple 2D model.)

\begin{figure}[htb]
\fbox{%
\begin{picture}(340,108)
\put(0,0){\includegraphics{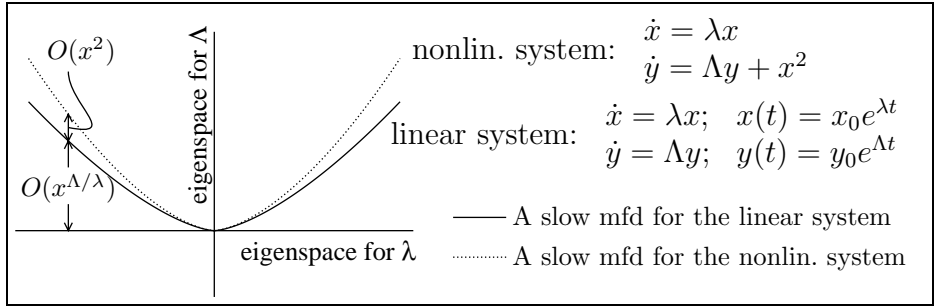}}
\put(142,58){linear system: 
     $\begin{array}{ll}\dot{x}=\lambda x; & x(t) = x_0 e^{\lambda t} \\
                       \dot{y}=\Lambda y; & y(t) = y_0 e^{\Lambda t}\\
      \end{array}$}
\put(150,90){nonlin.\ system: 
     $\begin{array}{l}\dot{x}=\lambda x\\
                      \dot{y}=\Lambda y+x^2\\
      \end{array}$}
\put(12,90){\footnotesize$O(x^2)$}
\put(2,38){\footnotesize$O(x^{\Lambda/\lambda})$}
\put(188,27){\footnotesize A slow mfd for the linear system}
\put(188,12){\footnotesize A slow mfd for the nonlin.\ system}
\end{picture}
}
\caption{\sl A simple 2D model illustrating the heuristics of working with
linear eigenspaces with $\Lambda<\lambda<0$  in Thm.~\ref{t-subquadratic},
rather than slow manifolds.  Linear asymptotics is
$x_0 e^{\lambda t} + O(e^{\Lambda t})$ (solid curve).  Nonlinear asymptotics
(dotted curve) should involve $e^{2\lambda t}$ etc., which gets absorbed into $O(e^{\Lambda t})$
if $\Lambda/\lambda < 2$.}
\label{Fig-heuristics}
\end{figure}

Returning to \eqref{conj-timemap},  let
$\Qop^{(i)}$ be the spectral projection onto
the one-dimensional eigenspace corresponding to the eigenvalue
$\Lambda_{i}$ of $\Lop$.
Take $\Qop = \sum_{i=1}^l \Qop^{(i)}$, and let $\Pop$ be the
complementary projection.

Setting
$\Qop^{(i)}_\eta =
(\cosh s)^{-\eta} \circ \Qop^{(i)} \circ (\cosh s)^\eta$ and
$\Pop_\eta = (\cosh s)^{-\eta} \circ \Pop \circ (\cosh s)^\eta$,
we obtain
$$
\Pop_\eta \tilde{w}_{j+1} = \Sop_\eta(T)(\Pop_\eta\tilde{w}_j)
              + (\Pop_\eta \circ G(w_j))  \tilde{w}_j
\;.
$$
By iteration, we obtain inductively
$$
\Pop_\eta\tilde{w}_k  = \Sop_\eta(T)^k (\Pop_\eta\tilde{w}_0) +
\sum_{j=0}^{k-1} \Sop_\eta(T)^{k-1-j} (\Pop_\eta\circ G(w_j)) \tilde{w}_j
\;.
$$
We set $\vartheta_{l+1}:=\exp[\Lambda_{l+1}T]$, so
$\|\Sop_\eta(T)^k\|\le C \vartheta_{l+1}^k$ on the range
of~$\Pop_\eta$. We also use what we have already established in
Thm.~\ref{t-1}, namely that $\|w_j\|_{C^\halpha}\le C\vartheta^j$
with $\vartheta:= \exp(\lambda_{01}T)$, with a similar estimate
following for $\tilde{w}_j$. Then we conclude
\begin{equation}\label{pre-remainder}
\|\Pop_\eta \tilde{w}_k\|_{C^\halpha(\M)}
\le
C \vartheta_{l+1}^k \Bigl(
   \|\Pop_\eta\tilde{w}_0\|_{C^\halpha(\M)}
   + \sum_{j=0}^{k-1} \vartheta_{l+1}^{-j-1}
          KC^2\vartheta^{2j} \Bigr)
\;.
\end{equation}
Since $\vartheta^2/\vartheta_{l+1}<1$ (here we use
$2\lambda_{01}<\Lambda$), the parenthesis is bounded independently
of~$k$. So we get
\begin{equation}\label{remainder}
\|\Pop_\eta \tilde{w}_k\|_{C^\halpha(\M)} \le c_2\vartheta_{l+1}^k
\;.
\end{equation}
A similar estimate can be made for the complementary projection
$$
 \Qop = \sum_{i=1}^l \Qop^{(i)}
\,.
$$
On the $\Lambda_i$ eigenspace,  $\Sop_\eta(T)$ operates as the scalar
$\vartheta_i = \exp[\Lambda_i T]$, hence
$$
\vartheta_i^{-k}\Qop^{(i)}_\eta \tilde w_k = \Qop^{(i)}_\eta \tilde w_0 + \sum_{j=0}^{k-1} \vartheta_i^{-j-1} \Qop^{(i)}_\eta G(w_j)\tilde w_j
$$
leads to a 2-sided Cauchy sequence estimate
$$
\left\|
   \frac{\Qop^{(i)}_\eta\tilde{w}_{k+q}}{\vartheta_i^{k+q}} -
   \frac{\Qop^{(i)}_\eta\tilde{w}_{k}}{\vartheta_i^{k}}
\right\|_{C^\halpha(\M)}
\le
\left({\vartheta^2}/{\vartheta_i}\right)^k
\sum_{j=0}^{q-1} c_3 \vartheta_i^{-1}
\left({\vartheta^2}/{\vartheta_i}\right)^j.
$$
We conclude $\tilde{v}_i := \lim_{k \to \infty }\Qop^{(i)}_\eta
\tilde{w}_k/\vartheta_i^k$
exists, and $\tilde{v}_i$ must be an eigenvector of $\Lop_\eta$ in the
$\Lambda_i$~eigenspace.  The convergence rate follows from this estimate:
\begin{equation}\label{eigenrate}
\left\|
   \Qop^{(i)}_\eta\tilde{w}_{k} - \tilde v_i \vartheta_i^{k}
\right\|_{C^\halpha(\M)}
\le
c_3 (\vartheta^2)^k
\;.
\end{equation}
Returning to the unconjugated functions, with $v_i=(\cosh s)^\eta
\tilde{v_i}$, we use from
Theorems \ref{spectrum-from-ARMA} and \ref{HolderSpectrum}, that
this eigenvector $v_i(\x) = u_i(\x)/(B+|\x|^2)$
is given by a degree $\ell + 2k'$
polynomial $u_i(\x)$,  and $\Lambda_i=\lambda_{\ell k'}$.
Since $\uB(\x)$ has the same mass and the same center of mass as the
initial data by hypothesis,
we have $u_{00}(\x) = 0 = u_{10}(\x)$ as remarked before
equation \eqref{cyl-L-NL},  hence $\ell + 2k >1$.

The last two inequalities combine with
the identity $1 = \Pop + \sum_{i=1}^l \Qop^{(i)}$ to yield
$$
\Big\| \frac{(B+|\x|^2)w_k - \sum_{i=1}^l \vartheta_i^k u_i}%
            {(B+|\x|^2)^{1+\eta/2}} \Big\|_{C^\halpha(\M)}
=
\Big\| \frac{w_k - \sum_{i=1}^l \vartheta_i^k v_i}%
            {(\cosh s)^\eta} \Big\|_{C^\halpha(\M)}
\le c_4 \vartheta_{l+1}^k
\;.
$$
Our choices
$\eta = \eta_{cr} - \sqrt{\Lambda - \lambda_0^{\rm cont}}
= \frac{p}{2} -1 - \sqrt{\Lambda + (\frac{p}{2}+1)^2 }$
and $\vartheta_i=\exp[\Lambda_i T]$ for $T$ sufficiently large
convert this into the desired estimate
\eqref{subquadratic-asymptotics}. This concludes the proof of
Thm.~\ref{t-subquadratic}.
\end{proof}

\begin{proof}[Proof of Thm.~\ref{t-superquadratic}]
Since the case $\lambda=\lambda_{01}$ is Thm.~\ref{t-subquadratic}, we
assume $\lambda<\lambda_{01}$ and
hypothesis~\eqref{amplification-hypothesis}.
The proof of Thm.~\ref{t-subquadratic} carries over with only minor
changes:

We keep conjugating with the $\eta$ from~\eqref{choose-eta} and
instead of using $\|\tilde w_j\|\le C\exp[\lambda_{01}Tj]=:
\vartheta^j$ obtained from Thm~\ref{t-1}, we use the
stronger~\eqref{amplification-hypothesis},
$\|\tilde w_j\|\le C\exp[\lambda Tj]=:\bar\vartheta^j\ll
\vartheta^j$.
As in the previous proof, we obtain \eqref{pre-remainder}, only with
$(\vartheta\bar\vartheta)^j$ instead of $\vartheta^{2j}$, and use
$(\vartheta\bar\vartheta)/\vartheta_{l+1}<1$ from
$\lambda+\lambda_{01}<\Lambda$. The same change applies to the
estimates of $\Qop$.

It is easy to see that the largest eigenvalue $\lambda_{\ell k}$ for
which the polynomial $u_{\ell k}(\x) \ne 0$ is non-vanishing
cannot exceed $\lambda$ without \eqref{subquadratic-asymptotics}
contradicting either \eqref{amplification-hypothesis} (in case
$\lambda < \lambda_{01}$) or Theorem \ref{t-1}
(in case $\lambda = \lambda_{01}$).
\end{proof}

\begin{cor}[Coefficient formulae]\label{c-coefficient formulae}
The coefficients appearing in the asymptotic
expansion~\eqref{subquadratic-asymptotics}
from Theorems \ref{t-subquadratic} and \ref{t-superquadratic} are given by
$u_{\ell k}(\x) = \psi_{\ell k}(|\x|
) \sum_{\mu} c_{\ell k \mu} Y_{\ell \mu}(\x/|\x|)$
where $\psi_{\ell k}$ and $Y_{\ell \mu}$ are the polynomial
eigenfunctions of $\Hop$ and spherical harmonics from Theorem
\ref{spectrum-from-ARMA} and
\begin{equation}\label{coefficient limits}
c_{\ell k \mu} =
\lim_{t \to \infty} e^{-\lambda_{\ell k}t}
\langle u(t,\cdot) - \uB, \psi_{\ell k} Y_{\ell \mu} \rangle_{L^2(\Rn)}
/ \| \psi_{\ell k} Y_{\ell \mu} \uB^{1-m/2}\|^{2}_{L^2(\Rn)}
\;.
\end{equation}
For $k=0 \ne \ell$, the $\psi_{\ell\, 0}Y_{\ell \mu}$ are the
homogeneous harmonic polynomials of degree $\ell$ and the expression
under the limit \eqref{coefficient limits} is independent of time,
hence
\begin{equation}\label{harmonic coefficients}
u_{\ell\, 0}(\x) = |\x|^\ell \sum_{\mu} Y_{\ell \mu}\Big(\frac{\x}{|\x|}\Big)
\frac{\langle u(0,\cdot), \psi_{\ell\, 0} Y_{\ell \mu} \rangle_{L^2(\Rn)}}
{\| \psi_{\ell\, 0} Y_{\ell \mu} \uB^{1-m/2}\|^2_{L^2(\Rn)}}
\;.
\end{equation}
\end{cor}

\begin{proof}
Rewrite \eqref{subquadratic-asymptotics} in the form
$$
\frac{u(t,\cdot) - \uB}{\uB^{2-m}}=
\sum_{\Lambda<\lambda_{\ell k}<0} u_{\ell k} e^{\lambda_{\ell k}t}
+O(e^{\Lambda t})(B + |\x|^2)^{\left(p+ 2 - \sqrt{(p+2)^2+4\Lambda}\right)/4}
\,,
$$
where we have $L^\infty$ control on the factor $O(e^{\Lambda t})$,
and $u_{\ell k}$
lies in the $\lambda_{\ell k}$ eigenspace of the self-adjoint
operator $\Hop$ on $L^2_{\uB^{2-m}}$ spanned by the family of
degree $\ell + 2k < p/2+1$ polynomials $(\psi_{\ell k} Y_{\ell\mu})_\mu$.
Multiplying this expression by
$e^{-\lambda_{\ell k}t}\psi_{\ell k} Y_{\ell \mu}$ and by
$\uB^{2-m} = (B + |\x|^2)^{-(n+p+2)/2}$, 
integration yields
$$
e^{-\lambda_{\ell k}t}
\langle u(t,\cdot) - \uB, \psi_{\ell k} Y_{\ell \mu} \rangle_{L^2(\Rn)} =
\langle u_{\ell k}, \psi_{\ell k} Y_{\ell \mu} \rangle_{L^2_{\uB^{2-m}}}
+ O(e^{(\Lambda-\lambda_{\ell k}) t})
\;.
$$
The remainder term vanishes in the limit $t \to \infty$ to establish
\eqref{coefficient limits}.

In case $k=0\ne\ell$ we have $\psi_{\ell\, 0}(|\x|) = |\x|^\ell$,
hence $\{\psi_{\ell\,0} Y_{\ell \mu}\}_\mu$
are the homogeneous harmonic polynomials of degree $\ell$.
These integrate to zero against the radial distribution $\uB$.
It remains only to show time independence of the expression under the
limit \eqref{coefficient limits} to complete
the proof of \eqref{harmonic coefficients}.
Transforming back to the original variables \eqref{back-trf},  we see
that the integral of $\rho(\tau,\y)$ against any harmonic polynomial
of degree less than $p$ is independent of $\tau$
from the evolution equation \eqref{pm};
the spatial decay $|\nabla \rho(\tau,\y)| = O(1/|\y|^{n+p+1})$ in
e.g.~\cite[Corollary 9]{MR2246356}
justifies the integration by parts.  Thus
$$
\int  \rho(\tau,\y) |\y|^\ell Y_{\ell \mu}(\frac{\y}{|\y|}) d\y
=\int  u(\ln (1 + 2p\tau)^{1/2p},\x)
|(1+2p\tau)^\spread\x|^\ell Y_{\ell\mu}(\frac{\x}{|\x|}) d\x
$$
is independent of
$t = \frac{1}{2p} \ln (1+ 2p\tau)$.
Since $\lambda_{\ell\,0}=  -2p\spread \ell$
from \eqref{scalingtrf} and \eqref{eigenvalues},
this establishes \eqref{harmonic coefficients}.
\end{proof}

\begin{proof}[Proof of Corollary \ref{c-2}]
Let us quickly review the top of the spectrum from Theorem \ref{HolderSpectrum}
(omitting the eigenvalues rendered ineffective by the mass / center of mass adjustments):
$$
\quad
\parbox{10em}{effective eigenvalues\\ of $\Lop_{\eta}$ when $n\ge3$: }
\begin{array}{l}
\lambda_{01}= -2((1-m)^{-1} -n) = -2p\\
\lambda_{02}= -8((1-m)^{-1} - 1 - n/2) = -4p+8\\
\lambda_{20}= -2p-2n \\
\lambda_0^{\rm cont}= -(p+2)^2/4
\end{array}
$$
The same picture, with the interval $[m_6,m_{n+4}]$ shrinking to a point,
applies to $n=2$, whereas in the case $n=1$ we have the intersection of
$\lambda_0^{\rm cont}$ and $\lambda_{20}$ at $p=2+2\sqrt{2}$.

Apart from $\lambda_{01}={-2p}$,  which forms the top of
the spectrum of $\Lop_{\eta_{cr}}$
in the full range $p>2$, the spectral gap is given by
$\Lambda = \max\{\lambda_0^{\rm cont}, \lambda_{02}, \lambda_{20} \}
= - \min\{(\frac{p}{2}+1)^2,4(p-2),2(p+n)\}$.
Since $\lambda_0^{\rm cont}/\lambda_{01} <2$, $\lambda_{02}/\lambda_{01}<2$
and $\lambda_{20}/\lambda_{01} <2$ are easily checked in the relevant ranges
$p\in`]2,6]$, $p \in `]6,n+4]$ and $p>n+4$,  Corollary
\ref{c-subquadratic-rho-asymptotics}
yields an asymptotic expansion
weighted by the denominator
$\tau^{-\gamma} (\rho(\tau,\origin)/\rho(\tau,\y))^{\delta+{2/(n+p)}}$
where $\gamma = \Lambda / \lambda_{01}$ is given by \eqref{d-gamma}, and
$\delta = \frac{1}{p+n} (\frac{p}{2} - 1 - \sqrt{(\frac{p}{2}+1)^2 + \Lambda})$
agrees with \eqref{d-delta}.  The sum in this asymptotic expansion consists of
only one term,  and it corresponds to the eigenvalue $\lambda_{01}$.  Thus
$$
\frac{\rho(\tau,\y)}{\rhoB(\tau,\y)}=
1 + \frac{u_{11}((1+2p\tau)^{-\spread}\y)}%
         {B(1+2p\tau)} 
    \left(\frac{\rhoB(\tau,\y)}{\rhoB(\tau,0)}\right)^{2/(p+n)}  
+ O(1/\tau^\gamma)
\,,
$$
as $\tau \to \infty$, where
$u_{11}(\x)/(B+|\x|^2)$ belongs to the
one-dimensional eigenspace of $\lambda_{01}$,  and the error is
measured in an $L^\infty$ norm with the desired weight.  Similarly
Corollary \ref{c-subquadratic-rho-asymptotics} asserts 
$$
\frac{\rhoB(\tau-\tau_0,\y)}{\rhoB(\tau,\y)}
= 
1 +
c(\tau_0)\frac{u_{11}((1+2p\tau)^{-\spread}\y)}%
              {B(1+2p\tau)} 
         \left(\frac{\rhoB(\tau,\y)}{\rhoB(\tau,0)}\right)^{2/(p+n)}  
+ O(1/\tau^\gamma)
\,,
$$
in the same norm, since the eigenspace is one-dimensional.  If
$\tau_0$ can be chosen to make $c(\tau_0)=1$,  then subtracting these
identities and using the fact \eqref{reluni} that
$\|\rhoB(\tau-\tau_0,\y)/\rhoB(\tau,\y) - 1 \|_{L^\infty(\Rn)} \to 0$
will conclude the proof of Corollary~\ref{c-2}.

On the other hand, from definition \eqref{d-rhoB} we compute
$$
\frac{\rhoB(\tau-\tau_0,\origin)}%
     {\rhoB(\tau,\origin)}
= 
\left(\frac{1 + 2p(\tau-\tau_0)}%
           {1 + 2p\tau}\right)^{-n\spread}
= 1 + n\spread \tau_0/\tau + o(1/\tau)
\,,
$$
from which we deduce that $c(\tau_0)$ does not generally vanish but
depends linearly on $\tau_0$. 
Thus a suitable choice of $\tau_0$ yields $c(\tau_0)=1$ to complete the proof.
\end{proof}

\begin{proof}[Proof of Corollary \ref{c-10}]
In the range $m > m_{n+4}$ and $n \ge 2$,  the next largest spectral
value after $\lambda_{20}$ is $\lambda_{11}$ 
(and not $\lambda_0^{\rm cont}$).
Taking $\Lambda = \lambda_{11}$ and observing
that the spectral ratio $\lambda_{11}/\lambda_{01} = 2 - (p+4-n)/2p$
is strictly less than $2$ in the range $p > n + 4$ of interest,  as in
the preceding proof we can choose $\tau_0$ to obtain
$$
\frac{\rho(\tau-\tau_0,\y)}{\rhoB(\tau,\y)} 
= 
1+ 
\frac{u_{20}((1+2p\tau)^{-\spread} \y)}{(1+2p\tau)^{(3p + n -4)/2p}}
\left(\frac{\rhoB(\tau,\y)}{\rhoB(\tau,0)}\right)^{\frac{2}{p+n}}  + 
O\left(\tau^{-\gamma}\right),
$$
as $\tau \to \infty$, where
the error is measured in an $L^\infty$
norm with the desired weight according to Theorem \ref{t-subquadratic},
and $u_{20}(\y)$ is a
homogeneous harmonic polynomial of degree $2$ as in Corollary
\ref{c-coefficient formulae}. 

A number of authors starting with Titov and Ustinov \cite{MR0857985}
and Tartar (circa 1986, unpublished but cited in \cite{MR2286292}) and
including two of us \cite{DenzlerMcCann08}, have independently observed that
the family of functions $\uB(\Sigma^{-1/2}\y) \, \det \Sigma^{-1/2}$
parametrized  by positive definite symmetric matrices $\Sigma>0$ form
a invariant manifold of dimension $n(n+1)/2$
under the porous medium and fast diffusion dynamics \eqref{pm}.
Each solution
$ \tilde \rho(\tau,\y) =  
\uB(\Sigma^{-1/2}(\tau) \y) \, \det \Sigma^{-1/2}(\tau)$
satisfies
$$
\frac{\rho(\tau-\tilde \tau_0,\y)}{\rhoB(\tau,\y)} = 1 +
    \frac{\tilde u_{20}((1+2p\tau)^{-\spread} \y)}{(1+2p\tau)^{(3p + n -4)/2p}}
\left(\frac{\rhoB(\tau,\y)}{\rhoB(\tau,0)}\right)^{\frac{2}{p+n}}  + 
O\left(\tau^{-\gamma}\right)
$$
as above. Since $\Sigma(\tau)$ is proportional to the moment of inertia tensor
$\int_\Rn \y \otimes \y \rho(\tau,\y) d\y$,
Corollary \ref{c-coefficient formulae} shows
the traceless part $\Sigma_0$ of $\Sigma(\tau)$ is independent of $\tau$
and can be selected  to make $\tilde u_{20} = u_{20}$.
The evolution equation $(d\sigma/d\tau)^{p+n} = c_B \det \Sigma(\tau)$
for $\Sigma(\tau) = \Sigma_0 + \sigma(\tau) I > 0$
is from Corollary~4 of \cite{DenzlerMcCann08}.
Subtracting the two equations above and using the fact that
$\|\tilde \rho(\tau-\tilde \tau_0,\y)/\rhoB(\tau,\y) - 1\|_{L^\infty(\Rn)} \to 0$
as in \eqref{reluni} yields the desired limit \eqref{affine asymptotics},
after translating the solution $\tilde \rho$ in time by $\tau_0-\tilde \tau_0$.
\end{proof}

\section{Appendix: Pedestrian derivation of all Schauder Estimates}

Although the announced results have now been established, for the sake 
of convenience, we include this appendix giving a self-contained 
derivation of all Schauder estimates for linear equations that we 
have relied on.

\begin{proof}[Self-contained proof of Lemma~\ref{constant}]
By superposition, the estimates are assembled from four cases,
in each of which
exactly one of the quantities $f,b,c,v_0$ is non-zero.

We repeat and modify the arguments from~\cite{MR0241822}, IV,
appropriately to match our situation, denoting
by~$\Gamma$ the heat kernel:
$\Gamma(t,x):= (4\pi t)^{-n/2} \exp[-|x|^2/4t]$.

We estimate the \underline{spatial H\"older quotients for the contribution
  from~$f$:}

In this estimate, $B(2r,x)$ denotes the ball centered at~$x$ with
radius $2r=2|x-x'|$. We have
$$
\begin{array}{l}  \Dst
v(t,x)-v(t,x') =
\int_0^t\int_{\Rn} \d_{ij}^2\Gamma(t-\tau,x-y)
                  \bigl(f^{ij}(\tau,y)-f^{ij}(\tau,x)\bigr)\,dy\,d\tau
\\[2ex]\Dst\phantom{v(t,x)-v(t,x')=\mbox{}}
-
\int_0^t\int_{\Rn} \d_{ij}^2\Gamma(t-\tau,x'-y)
                  \bigl(f^{ij}(\tau,y)-f^{ij}(\tau,x')\bigr)\,dy\,d\tau
\\[2ex] \Dst \mbox{} =
\int_0^t\int_{B(2r,x)} \d_{ij}^2\Gamma(t-\tau,x-y)
                  \bigl(f^{ij}(\tau,y)-f^{ij}(\tau,x)\bigr)\,dy\,d\tau
\\[2ex]\Dst\kern2em \mbox{}
-
\int_0^t\int_{B(2r,x)} \d_{ij}^2\Gamma(t-\tau,x'-y)
                  \bigl(f^{ij}(\tau,y)-f^{ij}(\tau,x')\bigr)\,dy\,d\tau
\\[2ex]\Dst\kern2em \mbox{}
+
\int_0^t\int_{B(2r,x)^c}
       \bigl(\d_{ij}^2\Gamma(t-\tau,x-y) - \d_{ij}^2\Gamma(t-\tau,x'-y)\bigr)
       \bigl(f^{ij}(\tau,y)-f^{ij}(\tau,x)\bigr)\,dy\,d\tau
\\[2ex]\Dst\kern2em \mbox{}
-
\int_0^t \bigl(f^{ij}(\tau,x)-f^{ij}(\tau,x')\bigr)
       \int_{B(2r,x)^c} \d_{ij}^2\Gamma(t-\tau,x'-y) \,dy\,d\tau
\;.
\end{array}
$$
We note: If $|x-y|\le2r$, then $|x'-y|\le3r$. In the third term, if
$|x-y|\ge2r$, then $|\sigma x+(1-\sigma)x'-y|\ge\frac12|x-y|$ for any
$\sigma\in[0,1]$. Therefore
\begin{equation}\label{HQx:f}
\begin{array}{l}  \Dst
|v(t,x)-v(t,x')| \le
\int_0^t\int_{B(2r,x)} |\d_{ij}^2\Gamma(t-\tau,x-y)| \,
                  |x-y|^\halpha\, [f^{ij}(\tau)]_{x;\halpha}\,dy\,d\tau
\\[2ex]\Dst\kern2em\mbox{}
+
\int_0^t\int_{B(2r,x)} |\d_{ij}^2\Gamma(t-\tau,x'-y)| \,
                  |x'-y|^\halpha\,[f^{ij}(\tau)]_{x;\halpha}\,dy\,d\tau
\\[2ex]\Dst\kern2em \mbox{}
+
\int_0^t\int_{B(2r,x)^c}
       \bigl|\d_{ij}^2\Gamma(t-\tau,x-y) - \d_{ij}^2\Gamma(t-\tau,x'-y)\bigr|
       \,|x-y|^\halpha\, [f^{ij}(\tau)]_{x;\halpha}\,dy\,d\tau
\\[2ex]\Dst\kern2em \mbox{}
+
\int_0^t|x-x'|^\halpha \, [f^{ij}(\tau)]_{x;\halpha} \biggl| \int_{B(2r,x)^c}
       \d_{ij}^2\Gamma(t-\tau,x'-y) \,dy\biggr| \,d\tau
\\[2ex]\Dst\kern1em \mbox{} \le
C\int_0^t\int_{B(2r,x)}
            (t-\tau)^{-\frac n2-1}
            \exp\bigl[-C{\Tst\frac{(x-y)^2}{t-\tau}}\bigr] \,
            |x-y|^\halpha\, [f^{ij}(\tau)]_{x;\halpha}\,dy\,d\tau
\\[2ex]\Dst\kern2em\mbox{}
+
C\int_0^t\int_{B(3r,x')}
            (t-\tau)^{-\frac n2-1}
            \exp\bigl[-C{\Tst\frac{(x'-y)^2}{t-\tau}}\bigr] \,
            |x'-y|^\halpha\, [f^{ij}(\tau)]_{x;\halpha}\,dy\,d\tau
\\[2ex]\Dst\kern2em\mbox{}
+
C\int_0^t\int_{B(2r,x)^c}
      |x-x'|       (t-\tau)^{-\frac n2 - \frac32}
      \exp\bigl[-C{\Tst \frac{(x-y)^2}{t-\tau}}\bigr] \,
      |x-y|^\halpha\, [f^{ij}(\tau)]_{x;\halpha}\,dy\,d\tau
\\[2ex]\Dst\kern2em \mbox{}
+
C\int_0^t|x-x'|^\halpha \, [f^{ij}(\tau)]_{x;\halpha}
       \int_{\boundary B(2r,x)}
       |\nabla\Gamma(t-\tau,x'-y)|\,dS(y) \,d\tau
\;.
\end{array}
\end{equation}
Now if  we estimate $[f(\tau)]_{x;\halpha}\le
\|f\|_{C^\halpha(\RnT)}$ and in the first three integrals
evaluate the time integration first,
using that, for $k>1$,
\begin{equation}\label{aux-int}
\int_0^t (t-\tau)^{-k}\exp\bigl[{\Tst-\frac{A^2}{t-\tau}}\bigr]\,d\tau =
(A^2)^{1-k}\int_0^{t/A^2}  \!\!s^{-k}\exp\bigl[{\Tst-\frac1s}\bigr]\,ds
\le C (A^2)^{1-k}
\,,
\end{equation}
we conclude that each term comes to the same estimate, namely we get
$$
|v(t,x)-v(t,x')|
\le C \|f\|_{C^\halpha(\RnT)} \, |x-x'|^{\halpha}
\;.
$$
In comparison, if we estimate $[f(\tau)]_{x;\halpha}\le
\tau^{-\halpha/2} \|f\|_{C^\halpha(\RnT)}^*$, we obtain the following time
integral instead, and we split it in the middle, obtaining
\begin{equation}\label{aux-int*}
\begin{array}{l} \Dst
\biggl(\int_0^{t/2}+\int_{t/2}^t\biggr)
(t-\tau)^{-k}\exp\bigl[-\frac{A^2}{t-\tau}\bigr]\tau^{-\halpha/2}\,d\tau
\\[2ex]\kern2em\Dst \mbox{}
\le
\max_{\sigma\in[t/2,t]}\bigl(\sigma^{1-k}\exp[-A^2/\sigma]\bigr)
\int_0^{t/2}(t-\tau)^{-1}\tau^{-\halpha/2}\,d\tau
\\[2ex]\kern4em\Dst\mbox{}
+
(t/2)^{-\halpha/2}(A^2)^{1-k} \int_0^{t/2A^2}s^{-k}\exp[-1/s]\,ds
\\[2ex]\kern2em\Dst \mbox{}
\le
C t^{-\halpha/2} (A^2)^{1-k}
\,,
\end{array}
\end{equation}
with the same estimate for each summand. This results in
$$
t^{\halpha/2} \, |v(t,x)-v(t,x')|
\le C \|f\|_{C^\halpha(\RnT)}^* \, |x-x'|^{\halpha}
\;.
$$

We now estimate the
\underline{time H\"older quotients for the contribution from~$f$:}

We assume $t'>t$ and let $t'-t=:d$.
Then
$$
\begin{array}{l}\Dst
v(t',x)-v(t,x) =
\int_0^{t'}\int_{\Rn} \d_{ij}^2\Gamma(t'-\tau,x-y)
                  \bigl(f^{ij}(\tau,y)-f^{ij}(\tau,x)\bigr)\,dy\,d\tau
\\[2ex]\Dst\phantom{v(t',x)-v(t,x)=\mbox{}}
-
\int_0^t\int_{\Rn} \d_{ij}^2\Gamma(t-\tau,x-y)
                  \bigl(f^{ij}(\tau,y)-f^{ij}(\tau,x)\bigr)\,dy\,d\tau
\\[2ex] \Dst \mbox{} =
\int_{(t-d)_+}^{t'}\int_{\Rn} \d_{ij}^2\Gamma(t'-\tau,x-y)
                  \bigl(f^{ij}(\tau,y)-f^{ij}(\tau,x)\bigr)\,dy\,d\tau
\\[2ex]\Dst\kern2em\mbox{}
-
\int_{(t-d)_+}^t\int_{\Rn} \d_{ij}^2\Gamma(t-\tau,x-y)
                  \bigl(f^{ij}(\tau,y)-f^{ij}(\tau,x)\bigr)\,dy\,d\tau
\\[2ex]\Dst\kern2em\mbox{}
+
\int_0^{(t-d)_+}  \!\! \int_{\Rn} \bigl[
   \d_{ij}^2\Gamma(t'-\tau,x-y) - \d_{ij}^2\Gamma(t-\tau,x-y) \bigr]
                  \bigl(f^{ij}(\tau,y)-f^{ij}(\tau,x)\bigr)\,dy\,d\tau
\;.
\end{array}
$$
Hence
\begin{equation}\label{HQt:f}
\begin{array}{l}\Dst
|v(t',x)-v(t,x)|
\le
\\[1ex]\Dst\kern2em\mbox{}
C\int_{(t-d)_+}^{t'}\int_{\Rn}
    (t'-\tau)^{-\frac n2-1} \exp\bigl[{\Tst-C\frac{|x-y|^2}{t'-\tau}}\bigr]
        |x-y|^{\halpha} \,dy
                \,  [f(\tau)]_{x;\halpha}\,d\tau
\\[2ex]\Dst\kern3em\mbox{}
+
C\int_{(t-d)_+}^t\int_{\Rn}
    (t-\tau)^{-\frac n2-1} \exp\bigl[{\Tst-C\frac{|x-y|^2}{t-\tau}}\bigr]
        |x-y|^{\halpha} \,dy
                \,  [f(\tau)]_{x;\halpha}\,d\tau
\\[2ex]\Dst\kern3em\mbox{}
+
C\int_0^{(t-d)_+}  \!\! (t'-t)\int_{\Rn}
    (t^*-\tau)^{-\frac n2-2} \exp\bigl[{\Tst-C\frac{|x-y|^2}{t^*-\tau}}\bigr]
        |x-y|^{\halpha} \,dy
                \,  [f(\tau)]_{x;\halpha}\,d\tau
\,,
\end{array}
\end{equation}
where $t^*\in[t,t']$; and therefore, in the last domain of integration,
we have
$\frac12(t'-\tau)\le t_*-\tau\le t'-\tau$.
Evaluating the space integrals directly first,
we obtain
\begin{equation} \label{HQt:f2}
\begin{array}{l} \Dst
|v(t',x)-v(t,x)|
\le
\\[1ex]\Dst\kern2em\mbox{}
C\int_{(t-d)_+}^{t'}
    (t'-\tau)^{\frac \halpha2-1}
                \,  [f(\tau)]_{x;\halpha}\,d\tau
+ \text{same with $t$ instead of $t'$}
\\[2ex]\Dst\kern3em\mbox{}
+
C\int_0^{(t-d)_+}  \!\! (t'-t)
    (t'-\tau)^{\frac \halpha2-2}
                \,  [f(\tau)]_{x;\halpha}\,d\tau
\;.
\end{array}
\end{equation}
Each term is dominated by
$C(t'-t)^{\halpha/2}\sup_\tau\|f(\tau)\|_{C^\halpha(\Rn)}
\le C(t'-t)^{\halpha/2}\|f\|_{C^\halpha(\RnT)}$, as desired.

If we use the weighted norms in \eqref{HE-est3} instead,
we obtain
\begin{equation} \label{HQt:f2*}
\begin{array}{l}\Dst
|v(t',x)-v(t,x)|
\le
\\[1ex]\Dst\kern2em\mbox{}
C \int_{(t-d)_+}^{t'}
    (t'-\tau)^{\frac \halpha2-1}
    \tau^{-\halpha/2} \|f\|_{C^\halpha(\RnT)}^* \,d\tau
\\[2ex]\Dst\kern3em\mbox{}
+
C \int_{(t-d)_+}^t
    (t-\tau)^{\frac \halpha2-1}
    \tau^{-\halpha/2} \|f\|_{C^\halpha(\RnT)}^* \,d\tau
\\[2ex]\Dst\kern3em\mbox{}
+
C(t'-t)  \int_0^{(t-d)_+}
    (t-\tau)^{\frac \halpha2-2}
    \tau^{-\halpha/2} \|f\|_{C^\halpha(\RnT)}^* \,d\tau
\\[2ex]\Dst\kern1em\mbox{}
\le
C (t'-t)^{\halpha/2}  t^{-\halpha/2} \|f\|_{C^\halpha(\RnT)}^*
\;.
\end{array}
\end{equation}
To justify the last estimate for the first two integrals, we distinguish two
cases: If $d\ge\frac12t$, our estimate
$[\frac{t'-t}{t}]^{\halpha/2}=(d/t)^{\halpha/2}$ is weaker than a constant,
whereas the integrals can be extended to a lower limit $\tau=0$ and readily
estimated by a constant. If $d\le\frac12t$, the term
$\tau^{-\halpha/2}$ is $\le (\frac12t)^{-\halpha/2}$, and the other
factor can be integrated and found to be bounded by $O((t'-t)^{\halpha/2})$.

For the third integral, in the case that $d\ge\frac12t$,
we argue that $(t-\tau)^{\frac\halpha2-2} \le
d^{\frac\halpha2-2} = (t'-t)^{\frac\halpha2-2}$, and the coefficient of
$\|f\|^*$ from the third term is
dominated by $C(t'-t)^{\frac\halpha2-1} (t-d)_+^{1-\frac\halpha2}$.
This quantity is again bounded by a constant, whereas our claimed
estimate is weaker than a constant.
On the other hand, if $d<\frac12t$, we split the integral in the middle,
estimate the bounded factor under each integral by its maximum (at
$\tau=\frac12t$ for each integral) and integrate the remaining factor.
Now $d\int_0^{t/2}\ldots d\tau\le C d/t\le C(d/t)^{\halpha/2}$
and $d\int_{t/2}^{t-d}\ldots d\tau\le C (d/t)^\halpha/2$,
hence the desired estimate.

We now estimate the
\underline{supremum norm  for the contribution from~$f$:}

This is the easy estimate
\begin{equation}\label{sup:f}
\begin{array}{l} \Dst
|v(t,x)|\le \int_0^t
\int_{\Rn}|\d_{ij}^2\Gamma(t-\tau,x-y)|\,|f^{ij}(\tau,y)-f^{ij}(\tau,x)|
\,dy\,d\tau
\\[2ex]\Dst\kern3em\mbox{}\le
C\int_0^t
\int_{\Rn} (t-\tau)^{-\frac n2-1} \exp\bigl[{\Tst -C\frac{|x-y|^2}{t-\tau}}\bigr]
|x-y|^{\halpha} \,dy \, [f(\tau)]_{x;\halpha}\,d\tau
\\[2ex]\Dst\kern3em\mbox{}\le
C\int_0^t (t-\tau)^{\frac\halpha2-1} [f(\tau)]_{x;\halpha}\,d\tau
\le C t^{\halpha/2}\|f\|_{C^\halpha(\RnT)}
\;.
\end{array}
\end{equation}
Since $t$ is bounded, we have estimated $|v(t,x)|$ by the H\"older norm of $f$.
In the same way, we can also estimate
$t^{\halpha/2}|v(t,x)|\le Ct^{\halpha/2}\|f\|_{C^\halpha(\RnT)}^*
\le C\|f\|_{C^\halpha(\RnT)}^*$.

We now estimate \underline{the spatial H\"older quotient contributed
  from~$b$\vphantom{$f$}:}

They can be estimated as in \eqref{HQx:f}, but without
splitting the space integral, as
\begin{equation}\label{HQx:b}
\begin{array}{l}\Dst
\frac{|v(t,x)-v(t,x')|}{|x-x'|^\halpha} \le
\int_0^t\int_\Rn |\d_i\Gamma(t-\tau,y)| \,
                  [b^{i}(\tau)]_{x;\halpha}\,dy\,d\tau
\\[2ex]\Dst\kern4em\mbox{}
\le C  \int_0^t (t-\tau)^{-\frac12}\,
   [b(\tau)]_{x;\halpha}\,d\tau
\\[2ex]\Dst\kern4em\mbox{}
\le
\min\bigl\{ C t^{\frac12}\|b\|_{C^\halpha(\RnT)}
\,,\;
C t^{\frac12-\frac\halpha2}\|b\|_{C^\halpha(\RnT)}^*
\bigr\}
\;.
\end{array}
\end{equation}

We now estimate \underline{the time H\"older quotients contributed
  from~$b$\vphantom{$f$}:}

For the time H\"older quotients, we can argue as in~\eqref{HQt:f2}, only with
an extra power of $(t'-\tau)^{1/2}$  under the
integrals. We get
$$
\begin{array}{l} \Dst
|v(t',x)-v(t,x)| \le C \bigl( (t'-t)^{\frac\halpha2+\frac12} +
(t'-t)(t'-t)^{\frac\halpha2-\frac12} \bigr) \sup_\tau [b(\tau)]_{x;\halpha}
\\[1.5ex] \Dst \phantom{|v(t',x)-v(t,x)|}
\le C(t'-t)^{\frac\halpha2} T^{\frac12} \|b\|_{C^\halpha(\RnT)}
\;.
\end{array}
$$

If we use the weighted norms $\|\cdot\|^*$ instead, we get by modification
of~\eqref{HQt:f2*}:
\begin{equation}\label{HQt:b*}
\begin{array}{l}\Dst
|v(t',x)-v(t,x)|
\le
\\[0.5ex]\Dst\kern4em\mbox{}
C  \|b\|_{C^\halpha(\RnT)}^*
\int_{(t-d)_+}^{t'} (t'-\tau)^{\halpha/2-1/2} \tau^{-\halpha/2}\,d\tau
\\[2ex]\Dst\kern5em\mbox{}
+
C  \|b\|_{C^\halpha(\RnT)}^*
\int_{(t-d)_+}^{t} (t-\tau)^{\halpha/2-1/2} \tau^{-\halpha/2}\,d\tau
\\[2ex]\Dst\kern5em\mbox{}
+
C(t'-t) \|b\|_{C^\halpha(\RnT)}^*  \int_0^{(t-d)_+}
    (t-\tau)^{\frac \halpha2-\frac32}
    \tau^{-\halpha/2}  \,d\tau
\;.
\end{array}
\end{equation}
The same splitting argument as for \eqref{HQt:f2*} proves that each term is
dominated by $C(t'-t)^{\halpha/2+1/2}t^{-\halpha/2}\|b\|_{C^\halpha(\RnT)}^*
\le CT^{1/2}(t'-t)^{\halpha/2}t^{-\halpha/2}\|b\|_{C^\halpha(\RnT)}^*$.

We now estimate the \underline{supremum norm contributed
  from~$b$\vphantom{$f$}:}

As in \eqref{sup:f}, we conclude
$|v(t,x)|\le C t^{\frac\halpha2+\frac12}\|b\|_{C^\halpha(\RnT)}$
and
$|v(t,x)|\le C t^{\frac12}\|b\|_{C^\halpha(\RnT)}^*$.
An even simpler version of the same estimate yields
$|v(t,x)|\le C t^{\frac12}\|b\|_{L^\infty(\RnT)}$ for later use in
proving~\eqref{HE-est2}.

We now estimate the \underline{space and time H\"older quotients contributed
  from~$c$\vphantom{$f$}:}

In contrast to $b$ and $f$, we cannot follow the paradigms of \eqref{HQx:f} and
\eqref{HQt:f} here, because $\Gamma$ carries no space derivative here.
This impacts the results of $[v]_{t;\halpha/2}$ since we cannot benefit from
using spatial H\"older quotients of~$c$. We get 
$$
\frac{|v(t,x)-v(t,x')|}{|x-x'|^\halpha} \le
\int_0^t\int_\Rn \Gamma(t-\tau,y)\, [c(\tau)]_{x;\halpha} dy\,d\tau
\le t \|c\|_{C^\halpha(\RnT)}
\;.
$$
Similarly, $t^{\halpha/2}[v(t)]_{x;\halpha} \le t \|c\|_{C^\halpha(\RnT)}^*$.

For the time H\"older quotients, we have (with $t'>t$)
$$
\begin{array}{l} \Dst
\frac{|v(t',x)-v(t,x)|}{|t'-t|^{\halpha/2}}
\le
\frac{1}{|t'-t|^{\halpha/2}}
\int_{t}^{t'} \int_{\Rn} \Gamma(\tau,y) c(t'-\tau,x-y)\,dy\,d\tau
\\[2ex]\kern6em \Dst \mbox{}
+
\int_0^{t} \int_{\Rn} \Gamma(\tau,y)
\frac{|c(t'-\tau,x-y)-c(t-\tau,x-y)|}{|t'-t|^{\halpha/2}} \,dy\,d\tau
\\[2ex]\kern4em \Dst \mbox{}
\le (t'-t)^{1-\halpha/2}\|c\|_{L^\infty(\RnT)} + T [c]_{t;\halpha/2}
\le T^{1-\halpha/2} \|c\|_{C^\halpha(\RnT)}
\;.
\end{array}
$$
With $|v(t,x)|\le t^{\halpha/2}[v]_{t;\halpha/2}$,
the $c$-estimate in~\eqref{HE-est1} follows immediately. The example
$c\equiv1$, $v(t,x)=t$ shows that the estimate is optimal.

Estimating the time H\"older quotients in terms of the weighted norm,
we get analogously
$$
\begin{array}{l} \Dst
\frac{|v(t',x)-v(t,x)|}{|t'-t|^{\halpha/2}}
\le
\frac{1}{|t'-t|^{\halpha/2}}
\int_{t}^{t'}  \|c(t'-\tau)\|_{L^\infty(\Rn)} \,d\tau
\\[2ex]\kern6em \Dst \mbox{}
+
\int_0^{t} (t-\tau)^{-\halpha/2} \|c\|_{C^\halpha(\RnT)}^* \,d\tau
\\[2ex]\kern4em \Dst \mbox{}
\le  C (t'-t)^{1-\halpha}\|c\|_{C^\halpha(\RnT)}^*
+ C t^{1-\halpha/2}\|c\|_{C^\halpha(\RnT)}^*
\;.
\end{array}
$$
Hence $t^{\halpha/2}\frac{|v(t',x)-v(t,x)|}{|t'-t|^{\halpha/2}}
\le C T^{1-\halpha/2} \|c\|_{C^\halpha(\RnT)}^*$; again, as a corollary, we get
$|v(t',x)|\le CT  \|c\|_{C^\halpha(\RnT)}^*$.
This gives the $c$ estimate in~\eqref{HE-est3}.

The estimates for the \underline{contributions from $v_0$} have been
given in Sec.~\ref{SecHeatEq},
Equations \eqref{vmax},\eqref{vxalpha},\eqref{vtalpha}
already.

Likewise, the estimates for the \\
\underline{contributions of $\|b\|_{L^\infty}$ and $\|c\|_{L^\infty}$
to the alternate bound~\eqref{HE-est2}} have been given
in Sec.~\ref{SecHeatEq} already.

This concludes the proof of Lemma~\ref{constant}.
\end{proof}

\bibliographystyle{plain}
\bibliography{fast}

\def\cprime{$'$} \def\cprime{$'$}
\begin{thebibliography}{10}

\bibitem{MR956056}
Sigurd Angenent.
\newblock Large time asymptotics for the porous media equation.
\newblock In {\em Nonlinear diffusion equations and their equilibrium states, I
  (Berkeley, CA, 1986)}, volume~12 of {\em Math. Sci. Res. Inst. Publ.}, pages
  21--34. Springer, New York, 1988.

\bibitem{MR0936323}
Sigurd Angenent.
\newblock Local existence and regularity for a class of degenerate parabolic
  equations.
\newblock {\em Math. Ann.}, 280(3):465--482, 1988.

\bibitem{AngenentAronson96}
Sigurd~B. {Angenent and Donald G. Aronson}.
\newblock Optimal asymptotics for solutions to the initial value problem for
  the porous medium equation.
\newblock In T.S.~Angell et~al, editor, {\em Nonlinear Problems in Applied
  Mathematics: in honor of Professor Ivar Stakgold on his 70th birthday}, pages
  10--19. Society for Industrial and Applied Mathematics, Philadelphia, 1996.

\bibitem{Barenblatt52}
Grigory~I. Barenblatt.
\newblock On some unsteady motions of a liquid or gas in a porous medium.
\newblock {\em Akad. Nauk. SSSR. Prikl. Mat. Mekh.}, 16:67--78, 1952.

\bibitem{BGM}
Marcel Berger, Paul Gauduchon, and Edmond Mazet.
\newblock {\em Le spectre d'une vari\'et\'e riemannienne}.
\newblock Lecture Notes in Mathematics, Vol. 194. Springer-Verlag, Berlin,
  1971.

\bibitem{BBDGV07}
Adrien Blanchet, Matteo Bonforte, Jean Dolbeault, Gabriele Grillo, and
  Juan-Luis V{\'a}zquez.
\newblock Hardy-{P}oincar\'e inequalities and applications to nonlinear
  diffusions.
\newblock {\em C. R. Math. Acad. Sci. Paris}, 344(7):431--436, 2007.

\bibitem{BBDGV09}
Adrien Blanchet, Matteo Bonforte, Jean Dolbeault, Gabriele Grillo, and
  Juan~Luis V{\'a}zquez.
\newblock Asymptotics of the fast diffusion equation via entropy estimates.
\newblock {\em Arch. Ration. Mech. Anal.}, 191(2):347--385, 2009.

\bibitem{BDGV10}
Matteo Bonforte, Jean Dolbeault, Gabriele Grillo, and Juan~L. V{\'a}zquez.
\newblock Sharp rates of decay of solutions to the nonlinear fast diffusion
  equation via functional inequalities.
\newblock {\em Proc. Natl. Acad. Sci. USA}, 107(38):16459--16464, 2010.

\bibitem{BGV10}
Matteo Bonforte, Gabriele Grillo, and Juan~Luis V{\'a}zquez.
\newblock Special fast diffusion with slow asymptotics: entropy method and flow
  on a {R}iemann manifold.
\newblock {\em Arch. Ration. Mech. Anal.}, 196(2):631--680, 2010.

\bibitem{MR2255281}
Jos\'e~A. Carrillo, Marco Di~Francesco, and Giuseppe Toscani.
\newblock Strict contractivity of the 2-{W}asserstein distance for the porous
  medium equation by mass-centering.
\newblock {\em Proc. Amer. Math. Soc.}, 135(2):353--363 (electronic), 2007.

\bibitem{CaLeMaTo03}
Jos\'e~A. Carrillo, Claudia Lederman, Peter~A. Markowich, and Giuseppe Toscani.
\newblock Poincar\'e inequalities for linearizations of very fast diffusion
  equations.
\newblock {\em Nonlinearity}, 15(3):565--580, 2002.

\bibitem{MR1777035}
Jos\'e~A. Carrillo and Giuseppe Toscani.
\newblock Asymptotic {$L\sp 1$}-decay of solutions of the porous medium
  equation to self-similarity.
\newblock {\em Indiana Univ. Math. J.}, 49(1):113--142, 2000.

\bibitem{MR1986060}
Jos{\'e}~A. Carrillo and Juan~L. V{\'a}zquez.
\newblock Fine asymptotics for fast diffusion equations.
\newblock {\em Comm. Partial Differential Equations}, 28(5-6):1023--1056, 2003.

\bibitem{MR2061425}
Bennett Chow and Dan Knopf.
\newblock {\em The {R}icci flow: an introduction}, volume 110 of {\em
  Mathematical Surveys and Monographs}.
\newblock American Mathematical Society, Providence, RI, 2004.

\bibitem{MR1009117}
Bj{\"o}rn E.~J. Dahlberg and Carlos~E. Kenig.
\newblock Nonnegative solutions to fast diffusions.
\newblock {\em Rev. Mat. Iberoamericana}, 4(1):11--29, 1988.

\bibitem{DaskalopoulosSesum06}
Panagiota Daskalopoulos and Natasa Sesum.
\newblock Eternal solutions to the {R}icci flow on {$\Bbb R^2$}.
\newblock {\em Int. Math. Res. Not.}, pages Art. ID 83610, 20, 2006.

\bibitem{MR787404}
Klaus Deimling.
\newblock {\em Nonlinear functional analysis}.
\newblock Springer-Verlag, Berlin, 1985.

\bibitem{MR1940370}
Manuel Del~Pino and Jean Dolbeault.
\newblock Best constants for {G}agliardo-{N}irenberg inequalities and
  applications to nonlinear diffusions.
\newblock {\em J. Math. Pures Appl. (9)}, 81(9):847--875, 2002.

\bibitem{MR1982656}
Jochen Denzler and Robert~J. McCann.
\newblock Phase transitions and symmetry breaking in singular diffusion.
\newblock {\em Proc. Natl. Acad. Sci. USA}, 100(12):6922--6925 (electronic),
  2003.

\bibitem{MR2126633}
Jochen Denzler and Robert~J. McCann.
\newblock Fast diffusion to self-similarity: complete spectrum, long-time
  asymptotics, and numerology.
\newblock {\em Arch. Ration. Mech. Anal.}, 175(3):301--342, 2005.

\bibitem{DenzlerMcCann08}
Jochen Denzler and Robert~J. McCann.
\newblock Nonlinear diffusion from a delocalized source: affine
  self-similarity, time reversal, and nonradial focusing geometries.
\newblock {\em Ann. Inst. H Poincar\'e Anal. Non Lineaire}, 25:865--888, 2008.

\bibitem{DT11}
Jean Dolbeault and Giuseppe Toscani.
\newblock Fast diffusion equations: matching large time asymptotics by relative
  entropy methods.
\newblock {\em Kinet. Relat. Models}, 4(3):701--716, 2011.

\bibitem{MR1721989}
Klaus-Jochen Engel and Rainer Nagel.
\newblock {\em One-parameter semigroups for linear evolution equations}, volume
  194 of {\em Graduate Texts in Mathematics}.
\newblock Springer-Verlag, New York, 2000.
\newblock With contributions by S. Brendle, M. Campiti, T. Hahn, G. Metafune,
  G. Nickel, D. Pallara, C. Perazzoli, A. Rhandi, S. Romanelli and R.
  Schnaubelt.

\bibitem{MR1295032}
Mikhail~V. Fedoryuk.
\newblock {\em Asymptotic analysis}.
\newblock Springer-Verlag, Berlin, 1993.
\newblock Linear ordinary differential equations, Translated from the Russian
  by Andrew Rodick.

\bibitem{MR0181836}
Avner Friedman.
\newblock {\em Partial differential equations of parabolic type}.
\newblock Prentice-Hall Inc., Englewood Cliffs, N.J., 1964.

\bibitem{MR586735}
Avner Friedman and Shoshana Kamin.
\newblock The asymptotic behavior of gas in an {$n$}-dimensional porous medium.
\newblock {\em Trans. Amer. Math. Soc.}, 262(2):551--563, 1980.

\bibitem{MR1210168}
Thierry Gallay.
\newblock {A center-stable manifold theorem for differential equations in
  Banach spaces}.
\newblock {\em Comm. Math. Phys.}, 152(2):249--268, 1993.

\bibitem{MR1912106}
Thierry Gallay and C.~Eugene Wayne.
\newblock Invariant manifolds and the long-time asymptotics of the
  {Navier-Stokes} and vorticity equations on $\bold {R}\sp 2$.
\newblock {\em Arch. Ration. Mech. Anal.}, 163(3):209--258, 2002.

\bibitem{MR797051}
Miguel~A. Herrero and Michel Pierre.
\newblock The {C}auchy problem for {$u\sb t=\Delta u\sp m$} when {$0<m<1$}.
\newblock {\em Trans. Amer. Math. Soc.}, 291(1):145--158, 1985.

\bibitem{Kato76}
Tosio Kato.
\newblock {\em Perturbation theory for linear operators}.
\newblock Springer-Verlag, Berlin, second edition, 1976.
\newblock Grundlehren der Mathematischen Wissenschaften, Band 132.

\bibitem{MR2246356}
Yong~Jung Kim and Robert~J. McCann.
\newblock Potential theory and optimal convergence rates in fast nonlinear
  diffusion.
\newblock {\em J. Math. Pures Appl. (9)}, 86(1):42--67, 2006.

\bibitem{KochHabil}
Herbert Koch.
\newblock {Non-Euclidean Singular Integrals and the Porous Medium Equation}.
\newblock Habilitation Thesis, Unversit\"at Heidelberg, Germany, 1999.

\bibitem{MR1406091}
N.~V. Krylov.
\newblock {\em Lectures on elliptic and parabolic equations in {H}\"older
  spaces}, volume~12 of {\em Graduate Studies in Mathematics}.
\newblock American Mathematical Society, Providence, RI, 1996.

\bibitem{MR0241822}
O.~A. Lady{\v{z}}enskaja, V.~A. Solonnikov, and N.~N. Ural{\cprime}ceva.
\newblock {\em Linear and quasilinear equations of parabolic type}.
\newblock Translated from the Russian by S. Smith. Translations of Mathematical
  Monographs, Vol. 23. American Mathematical Society, Providence, R.I., 1967.

\bibitem{Lamb80}
George~L. Lamb, Jr.
\newblock {\em Elements of soliton theory}.
\newblock John Wiley \& Sons Inc., New York, 1980.
\newblock Pure and Applied Mathematics, A Wiley-Interscience Publication.

\bibitem{LedermanMarkowich03}
Claudia Lederman and Peter~A. Markowich.
\newblock On fast-diffusion equations with infinite equilibrium entropy and
  finite equilibrium mass.
\newblock {\em Comm. Partial Differential Equations}, 28(1-2):301--332, 2003.

\bibitem{MR2211152}
Robert~J. McCann and Dejan Slep{\v{c}}ev.
\newblock Second-order asymptotics for the fast-diffusion equation.
\newblock {\em Int. Math. Res. Not.}, pages Art. ID 24947, 22, 2006.

\bibitem{MR1842429}
Felix Otto.
\newblock The geometry of dissipative evolution equations: the porous medium
  equation.
\newblock {\em Comm. Partial Differential Equations}, 26(1-2):101--174, 2001.

\bibitem{Pattle59}
R.E. Pattle.
\newblock Diffusion from an instantaneous point source with concentration
  dependent coefficient.
\newblock {\em Quart. J. Mech. Appl. Math.}, 12:407--409, 1959.

\bibitem{MR0219861}
Murray~H. Protter and Hans~F. Weinberger.
\newblock {\em Maximum principles in differential equations}.
\newblock Prentice-Hall Inc., Englewood Cliffs, N.J., 1967.

\bibitem{MR1419319}
Thomas Runst and Winfried Sickel.
\newblock {\em Sobolev spaces of fractional order, {N}emytskij operators, and
  nonlinear partial differential equations}, volume~3 of {\em de Gruyter Series
  in Nonlinear Analysis and Applications}.
\newblock Walter de Gruyter \& Co., Berlin, 1996.

\bibitem{MR0857985}
S.S. Titov and V.A. Ustinov.
\newblock {Investigation of polynomial solutions of the two-dimensional
  Le\u{i}benzon filtration equation with an integral exponent of the adiabatic
  curve (in Russian)}.
\newblock In A.F. Sidorov and S.V. Vershinin, editors, {\em Approximate methods
  for solving boundary value problems of continuum mechanics (in Russian)},
  volume~91, pages 64--70. Ural. Nauchn. Tsentr, Sverdlovsk, 1985.

\bibitem{MR694373}
Juan~Luis V{\'a}zquez.
\newblock Asymptotic behaviour and propagation properties of the
  one-dimensional flow of gas in a porous medium.
\newblock {\em Trans. Amer. Math. Soc.}, 277(2):507--527, 1983.

\bibitem{MR1977429}
Juan~Luis V{\'a}zquez.
\newblock Asymptotic behaviour for the porous medium equation posed in the
  whole space.
\newblock {\em J. Evol. Equ.}, 3(1):67--118, 2003.
\newblock Dedicated to Philippe B\'enilan.

\bibitem{MR2286292}
Juan~Luis V\'azquez.
\newblock {\em The porous medium equation. Mathematical theory}.
\newblock Oxford Mathematical Monographs. The Clarendon Press, Oxford
  University Press, Oxford, 2007.

\bibitem{MR1465095}
C.~Eugene Wayne.
\newblock Invariant manifolds for parabolic partial differential equations on
  unbounded domains.
\newblock {\em Arch. Rational Mech. Anal.}, 138(3):279--306, 1997.

\bibitem{MR1491842}
Thomas~P. Witelski and Andrew~J. Bernoff.
\newblock Self-similar asymptotics for linear and nonlinear diffusion
  equations.
\newblock {\em Stud. Appl. Math.}, 100(2):153--193, 1998.

\bibitem{ZeldovichBarenblatt58}
Ya.B. Zel'dovich and G.I. Barenblatt.
\newblock The asymptotic properties of self-modelling solutions of the
  nonstationary gas filtration equations.
\newblock {\em Sov. Phys. Doklady}, 3:44--47, 1958.

\bibitem{ZeldovichKompaneets50}
Ya.B. Zel'dovich and A.S. Kompaneets.
\newblock Theory of heat transfer with temperature dependent thermal
  conductivity.
\newblock In {\em Collection in Honour of the 70th Birthday of Academician
  A.F.~Ioffe}, pages 61--71. Izdvo. Akad. Nauk. SSSR, Moscow, 1950.

\end{thebibliography}

\end{document}